\documentclass[11pt,reqno]{amsart}
\usepackage{amsmath,amsthm,amsfonts,amssymb,bm}
\usepackage{fancyhdr,amscd}
\usepackage{anysize}
\usepackage{graphicx,epsfig}
\usepackage{amsthm,latexsym,amsmath,amssymb}
\usepackage[left=1.5cm,right=1.5cm,top=2cm,bottom=2cm]{geometry}
\usepackage{mathrsfs}
\usepackage{cite}
\usepackage{indentfirst}
 \usepackage[titletoc]{appendix}
\numberwithin{equation}{section}

\newcommand{\beq}{\begin{equation}}
\newcommand{\eeq}{\end{equation}}
\newcommand{\ben}{\begin{eqnarray}}
\newcommand{\een}{\end{eqnarray}}
\newcommand{\beno}{\begin{eqnarray*}}
\newcommand{\eeno}{\end{eqnarray*}}

\newcommand{\pa}{\partial}

\newcommand{\eps}{\epsilon}

\newcommand{\RR}{\mathbb{R}}

\newcommand{\ZZ}{\mathbb{Z}}
\newcommand{\NN}{\mathbb{N}}

\newcommand{\DD}{\Delta}

\newcommand{\dd}{\partial}
\newcommand{\f}{\frac}
\newcommand{\supp}{\mbox{supp }}

\newcommand{\abs}[1]{\left\vert#1\right\vert}
\newcommand{\set}[1]{\left\{#1\right\}}
\newcommand{\psca}[1]{\left\langle#1\right\rangle}

\newcommand{\pare}[1]{\left(#1\right)}
\newcommand{\norm}[1]{\left\Vert#1\right\Vert}

\newtheorem{theorem}{\textbf Theorem}[section]
\newtheorem{lemma}{\textbf Lemma}[section]
\newtheorem{rem}{\textbf Remark}[section]
\newtheorem{cor}{\textbf Corollary}[section]

\newtheorem{prop}{\textbf Proposition}[section]
\newtheorem{defin}{\textbf Definition}[section]

\newcounter{remark}
{\par \stepcounter{remark} {\it Remark
\arabic{section}.\arabic{remark}.}~}{\rm \end Proof\par}

\allowdisplaybreaks \voffset=-0.2in \numberwithin{equation}{section}

\voffset=-0.2in \numberwithin{equation}{section}\allowdisplaybreaks
\subjclass[2010]{35Q30, 76D03}
\keywords{Hyperbolic Perturbation Of NS, Littlewood-Paley theory, Well-posedness.}

\begin{document}
\title{Global well-posedness of 2D Hyperbolic perturbation of the Navier-Stokes system in a thin strip}

\author[N. AARACH]{Nacer Aarach$^{1}$}

\address{$^1$ Universit\'e de Bordeaux, Institut de Math\'ematiques de Bordeaux, F-33405 Talence Cedex, France}

\email{nacer.aarach@math.u-bordeaux.fr}

\begin{abstract}  In this paper, we study a hyperbolic version of the Navier-Stokes equations, obtained by using the approximation by relaxation of the Euler system, evolving in a thin strip domain. The formal limit of these equations is a hyperbolic Prandtl type equation, our goal is to prove the existence and uniqueness of a global solution to these equations for analytic initial data in the tangential variable, under a uniform smallness assumption. Then we justify the limit from the anisotropic hyperbolic Navier-Stokes system to the hydrostatic hyperbolic Navier-Stokes system with small analytic data.   
 \end{abstract} 
 
\maketitle
\section{introduction}
 In our work, we consider a hyperbolic perturbation of the incompressible Navier-Stokes equations in $\mathbb{R}\times (0,\epsilon)$ such that the viscosity $\nu = \eps^2$ (this viscosity called vanishing viscosity). We study this system in a thin like region and have given it slip-free boundary conditions. We denote $\mathcal{S}^\eps = \lbrace (x,Y) \in \mathbb{R}^2 : 0<Y<\eps \rbrace $ where $\eps$ is the width of the domain. Our system is of the following form:
\begin{equation}
	\label{HPNS}
	\left\{\;
	\begin{aligned}
		&\tau \pa_t^2 U^{(\tau,\eps)} + \partial_t U^{(\tau,\eps)} + U^{(\tau,\eps)} .\nabla U^{(\tau,\eps)}    \\ &\quad \quad \quad \quad \quad \quad -\eps^2 \Delta U^{(\tau,\eps)} + \nabla P^{(\tau,\eps)} = 0, && \qquad \mbox{in }\, ]0,\infty[ \times \mathcal{S}^\eps\\
		&\text{ div } \ U^{(\tau,\eps)} = 0, && \qquad \mbox{in }\,]0,\infty[ \times \mathcal{S}^\eps \\
		&U^{(\tau,\eps)}_{/t=0} = U^{(\tau,\eps)}_0, \text{ \ \ \ } \pa_t U^{(\tau,\eps)}_{/t=0} =U^{(\tau,\eps)}_1, && \qquad \mbox{in }\,\mathcal{S}^\eps 		
	\end{aligned}
	\right.
\end{equation}
 
where $U^{(\tau,\eps)}(t,x,Y) = \pare{U_1^{(\tau,\eps)}(t,x,Y), U_2^{(\tau,\eps)}(t,x,Y)}$ denotes the velocity of the fluid and $P^{(\tau,\eps)}(t,x,y)$ the scalar pressure function, which guarantees the divergence-free property of the velocity field $U^{(\tau,\eps)}$. The system $\eqref{HPNS}$ is complemented by the no-slip boundary condition
\begin{equation*}
	U^{(\tau,\eps)}_{\vert Y=0} = 0 \quad\mbox{and}\quad U^{(\tau,\eps)}_{\vert Y=\eps} = 0 
	\text{ \ in \ } ]0,\infty[ \times \mathcal{S}^\eps.
\end{equation*}
Here, in the equation of the velocity, the Laplacian is $\Delta = \partial_x^2 + \partial_Y^2$ .\\

The Navier-Stokes equations are a type of nonlinear partial differential equation that describes the motion of Newtonian fluids. They have been a subject of intense research since their invention in the 1930s.  These equations are defined by the following Newtonian incompressible viscous fluids:
\begin{equation}
	\label{NSEQ}
	\left\{\;
	\begin{aligned}
		& \partial_t U + U \cdot \nabla U -\nu \Delta U + \nabla P = 0, \\
		&\text{ div } \ U = 0.   		
	\end{aligned}
	\right.
\end{equation}
here $U$ denotes the velocity field, $P$ the scalar pressure function, and $\nu$ is the fluid's viscosity coefficient. \\

Nonetheless, the system \eqref{NSEQ} creates a physical conundrum due to the fact that the velocity equation has an infinite propagation speed. To avoid this nonphysical aspect, Cattaneo in \cite{CC,CR22} advocated modifying the heat equation to a hyperbolic form known as Cattaneo's heat transfer law. More specifically, they advocated replacing the Fourier law  they proposed to replace Fourier's law which describes the following stress tensor 
$$ \tau(t) = -PId + \nu\Big(\nabla U + (\nabla U)\top \Big)(t)^, $$ 
with the hyperbolic model shown below
 $$ \frac{1}{c^2} \pa_t^2 \theta + \frac{1}{\beta}\pa_t \theta - \Delta \theta = 0. $$
 
 This equation is called \textbf{The Telegraph equation}. It has a finite propagation speed and is compatible with both principle of relativity and second law of thermodynamics, so it is a satisfactory physical model. As a result, it makes sense to think of a hyperbolic Navier-Stokes system by including the term $\tau \pa_t^2$ to the classical Navier-Stokes system \eqref{NSEQ}, where $\tau$ is a small parameter we can assume that is more small than $1/c^2$ where $c$ is the speed of light.\\

The dissipative hyperbolic Navier-Stokes equation \eqref{HPNS} is obtained after relaxing the Euler equations and a change of scale variables. This perturbation, considered as a relaxation of Euler's equations, was considered by Brenier, Natalini, and Puel in \cite{YRM2004}. They  introduced a hyperbolic version of equations, based on a relaxation approximation of the incompressible Navier-Stokes equations, following the scheme described by Jin and Xin in \cite{SZ1995}.

\begin{equation}
	\label{HPNS22}
	\left\{\;
	\begin{aligned}
		\pa_t U^\eps + div \ V^\eps &= \nabla Q^\eps \\
		\pa_t V^\eps + \frac{\eps^2}{\tau}\nabla U^\eps &= -\frac{1}{\tau}(V^\eps - U^\eps \otimes U^\eps) \\
		div \ U^\eps &= 0 \\
		(U^\eps,V^\eps)|_{t=0}  &= (U_0^\eps,V_0^\eps)  
	\end{aligned}
	\right.
\end{equation}

In their work they proved global existence and uniqueness for the perturbed Navier-Stokes equation with initial data in $\mathcal{H}^2(\mathcal{T}^2)^2\times \mathcal{H}^1(\mathcal{T}^2)^2 $, where $\mathcal{T}^2$ is the periodic square $\mathbb{R}^2/\mathbb{Z}^2$. Moreover, they proved the convergence of the solution for perturbed Navier-Stokes to a smooth solution for Navier-Stokes. 

Later this equation was considered by Paicu and Raugel in \cite{MR2007,MR2008}. In their work, they also proved a global existence and uniqueness result with significantly improved regularity for the initial data, when $\tau$ is small enough. In fact, they only require the regularity in $\mathcal{H}^1(\mathbb{R}^2)^2 \times L^2(\mathbb{R}^2)^2$. Also Hachicha in \cite{IHachicha}, obtained a global result of existence and uniqueness of the perturbed Navier-Stokes in two and three space dimensions and under suitable smallness assumptions on the initial data in the space $ [(\mathcal{H}^{\frac{n}{2} + \delta} \cap \mathcal{H}^{\frac{n}{2}-1 + \delta})(\mathbb{R}^n)]^n $ with $n=2,3$. Moreover, for all positive times $T$, it proved the convergence to perturbed Navier-Stokes towards solutions of the Navier-Stokes system (NS) with initial data in $ \mathcal{H}^{\frac{n}{2}-1 + s}(\mathbb{R}^n)^n $, $s>0$. 

We finally mention a recent result obtained by O. Coulaud, I. Hachicha and G. Raugel in \cite{OIG}. They considered a hyperbolic quasi-linear version of the Navier-Stokes equations in $\mathbb{R}^2$ and proved the existence and uniqueness of solutions to these equations, and exhibit smallness assumptions on the data, under which the solutions are global in time in the 2D case. 

The way these authors introduce their system of equations is by using
methods that are usually devoted to the study of numerical patterns, and that can be applied to all preservation laws. In this article, we take the problem from another point of view. In order to describe hydrodynamical flows on the earth, in geophysics, it is usually assumed that vertical motion is much smaller than horizontal motion and that the fluid layer depth is small compared to the radius of the sphere, thus, they are a good approximation of global atmospheric and oceanic flows. The thin domain in the system \eqref{HPNS} is considered to take into account this anisotropy between horizontal and vertical directions. Under this assumption, it is believed that the dynamics of fluids on large scale tend towards a geostrophic balance (see \cite{G1982}, \cite{H2004} or \cite{PZ2005}). 

The purpose of this paper is to show the existence and uniqueness of solutions to \eqref{HPNS} in the thin domain $\mathbb{R} \times (0,\epsilon)$, For some analytically small initial data in the tangential variable. To simplify our system we eliminate the $\tau$-dependency. To that extend, we perform the re-scaling 
\begin{align}\label{rescaling}
U^{(\tau,\eps)}(t,X) = \tau^{\alpha}U^\eps(\tau^\beta t, \frac{X}{\sqrt{\tau}}), \quad \quad \quad  P^{(\tau,\eps)}(t,X) = \tau^{\alpha'}P^\eps(\tau^\beta t, \frac{X}{\sqrt{\tau}}).
\end{align}

We replace in system \eqref{HPNS}, we find that $\alpha = -\frac{1}{2}$, $\beta = -1$ and $\alpha^{'} = -1$, then our re-scaling \ref{rescaling} have the following form 
\begin{align}\label{rescaling1}
U^{(\tau,\eps)}(t,X) = \frac{1}{\sqrt{\tau}} U^\eps(\frac{t}{\tau}, \frac{X}{\sqrt{\tau}}), \quad \quad \quad  P^{(\tau,\eps)}(t,X) = \frac{1}{\tau}P^\eps(\frac{t}{\tau}, \frac{X}{\sqrt{\tau}}).
\end{align}

This scaling transforms the $\tau$-dependent equations \eqref{HPNS} into the following system of equations with initial data which depend on $\tau$:
\begin{equation}
	\label{HPNS1}
	\left\{\;
	\begin{aligned}
		 &\pa_t^2 U^{\eps} + \partial_t U^{\eps} + U^{\eps} .\nabla U^{\eps}  -\eps^2 \Delta U^{\eps} + \nabla P^{\eps} = 0, && \qquad \mbox{in }\, ]0,\infty[\times \mathcal{S}^\eps\\
		&\text{ div } \ U^{(\tau,\eps)} = 0, && \qquad \mbox{in }\,]0,\infty[\times \mathcal{S}^\eps \\
		&U^{\eps}_{/t=0} = \sqrt{\tau} U_0^{\tau,\eps}(\sqrt{\tau}X) = U^{\eps}_0,  && \qquad \mbox{in }\, \mathcal{S}^\eps \\
		&\pa_t U^{\eps}_{/t=0}  = \tau^{\frac{3}{2}}U^{\tau,\eps}_1(\sqrt{\tau}X)=U^{\eps}_1, && \qquad \mbox{in }\, \mathcal{S}^\eps \\
		& U^\eps/_{y=0} = U^\eps/_{y=1} = 0, && \qquad \mbox{in }\,]0,\infty[\times \mathcal{S}^\eps
		\end{aligned}
	\right.
\end{equation}

In a formal way, as in \cite{PZZ2019} and \cite{AN2020}, taking into account this anisotropy, we also consider the initial data of the following form,
$$ U^\eps_{\vert t=0} = U_{0}^{\eps} = \left( u_0\pare{x,\frac{Y}{\eps}},\eps v_{0}\pare{x,\frac{Y}{\eps}}\right),$$ and $$ U^\eps_{\vert t=1} = U_{1}^{\eps} = \left( u_1\pare{x,\frac{Y}{\eps}},\eps v_{1}\pare{x,\frac{Y}{\eps}}\right).$$

In this paper, we look for solutions in the form  
\begin{equation} 
	\label{S1eq3bis}
	\left\{\;
	\begin{aligned}
		&U^\eps(t,x,Y)=\pare{u^\eps \pare{t,x,\frac{Y}{\eps}}, \eps v^\eps \pare{t,x,\frac{Y}{\eps}}}\\
		&P^\eps(t,x,Y)=p^\eps \pare{t,x,\frac{Y}{\eps}}.
	\end{aligned}
	\right.
\end{equation}

 Let $\mathcal{S} := \left\{(x,y)\in\mathbb{R}^2:\ 0<y<1\right\}$, we can rewrite the system $\eqref{HPNS1}$ as follows
\begin{equation}
	\label{eq:HPNhydroPE}
	\left\{\;
	\begin{aligned}
		&\partial_t^2 u^\eps + \partial_t u^\eps +u^\eps \partial_x u^\eps + v^\eps \partial_yu^\eps -\eps^2\partial_x^2u^\eps-\partial_y^2u^\eps+\partial_x p^\eps=0, && \mbox{in } \; ]0,\infty[\times \mathcal{S} \\
		&\eps^2(\partial_t^2 v^\eps + \partial_tv^\eps+u^\eps \partial_x v^\eps+ v^\eps \partial_yv^\eps -\eps^2\partial_x^2v^\eps-\partial_y^2v^\eps)+\partial_yp^\eps=0, && \mbox{in } \; ]0,\infty[\times \mathcal{S}\\
		&\partial_x u^\eps+\partial_yv^\eps=0, && \mbox{in } \; ]0,\infty[\times \mathcal{S}\\
		&\left(u^\eps, v^\eps \right)|_{t=0}=\left(u_0, v_0 \right) \text{ \ and \ } \pa_t\left(u^\eps, v^\eps \right)|_{t=0}=\left(u_1, v_1 \right), && \mbox{in } \; \mathcal{S}\\
		&\left(u^\eps, v^\eps \right)|_{y=0}=\left(u^\eps, v^\eps \right)|_{y=1}=0.
	\end{aligned}
	\right.
\end{equation}
Formally taking $\eps\to 0$ in the system $\eqref{eq:HPNhydroPE}$, we obtain the following perturbation hydrostatic Navier-Stokes equations,
	\begin{equation}
	\label{eq:HPNhydrolimit}
	\left\{\;
	\begin{aligned}
		&\pa_t^2 u + \pa_tu+u\pa_x u+v\pa_yu -\pa_y^2u+\pa_xp=0, && \mbox{in } \; ]0,\infty[\times \mathcal{S} \\
		&\pa_y p= 0, && \mbox{in } \; ]0,\infty[\times \mathcal{S}\\
		&\pa_xu+\pa_yv=0,  && \mbox{in } \; ]0,\infty[\times \mathcal{S}\\
		&u|_{t=0}=u_0,  && \mbox{in } \; \mathcal{S}\\
		&\pa_tu|_{t=0}=u_1, && \mbox{in } \; \mathcal{S},
	\end{aligned}
	\right.
\end{equation}
where the velocity $U = (u,v)$ satisfy the Dirichlet no-slip boundary condition
\begin{equation} 
	\label{S1eq7}
	\left(u, v\right)|_{y=0}=\left(u, v\right)|_{y=1}=0.
\end{equation}

Now let us state our main results.

The first result obtained in this paper is the global well-posedness of the system \eqref{eq:HPNhydrolimit} with small analytic data in the tangential variable. The global well-posedness and the global analyticity of the solutions to the classical 2-D perturbed hydrostatic Navier-Stokes system are well-known (see \cite{MR2007} for instance). We remark that a similar global result seems open for the Prandtl equation, where only a lower bound of the lifespan to the solution was obtained (see \cite{ZZ2016},\cite{MV2016}).
\begin{theorem}\label{th:11}
	Let $a > 0$. There exists a constant $c_0 > 0$ sufficiently small, such that, for any data $(u_0,u_1)$ verifying the  compatibility condition $\int_0^1 u_0 dy =0$, the smallness condition 
	\begin{align}\label{eq:condini}
		\norm{e^{a|D_x|} (u_0+u_1)}_{\mathcal{B}^\f12} + \norm{e^{a|D_x|} \pa_y u_0}_{\mathcal{B}^\f12} + \|e^{a|D_x|}u_1 \|_{\mathcal{B}^{\f12}} \leq c_0a,
	\end{align}
	then the system \eqref{eq:HPNhydrolimit} has a unique global solution $u$ satisfying the estimate
	\begin{multline} \label{estimate 1}
	 \Big( \frac{1}{2}\|e^{\mathcal{R} t}(u + \pa_tu)_\phi \|_{\tilde{L}_t^\infty(\mathcal{B}^{\f12})} + \|e^{\mathcal{R} t} \pa_yu_\phi \|_{\tilde{L}_t^\infty(\mathcal{B}^{\f12})} \Big)+ \f12 \|e^{\mathcal{R} t}(\pa_tu)_\phi\|_{\tilde{L}_t^2(\mathcal{B}^{\f12})} + \f12\|e^{\mathcal{R} t}\pa_yu_\phi\|_{\tilde{L}_t^2(\mathcal{B}^{\f12})}  \\+ \frac{1}{2}\|e^{\mathcal{R} t}(\pa_tu)_\phi\|_{\tilde{L}_t^\infty(\mathcal{B}^{\f12})} \leq C\|e^{a|D_x|}\pa_yu_0 \|_{\mathcal{B}^{\f12}} +C\|e^{a|D_x|}(u_0+u_1) \|_{\mathcal{B}^{\f12}} + C\|e^{a|D_x|}u_1 \|_{\mathcal{B}^{\f12}},
\end{multline}
where $u_\phi$ is given by \eqref{eq:Anphichp3}, $\mathcal{R}$ is a constant smaller than the constant that comes out of the Poincaré inequality for the domain $\mathcal{S}$ (see \eqref{poincarebis1}), and the functional spaces will be presented in Section \ref{sec:1}.
\end{theorem}

 The second result is the global well-posedness of the perturbed Navier-Stokes system \eqref{eq:HPNhydroPE} with small analytic data in the tangential variable. 
\begin{theorem}\label{th:22}
Let $a>0$. We assume that our initial satisfies the following smallness condition
	\begin{align}\label{u0v0}
		\|e^{a|D_x|}\pa_y(u_0,\eps v_0) \|_{\mathcal{B}^{\f12}} + \eps \|e^{a|D_x|}\pa_x(u_0,\eps v_0) \|_{\mathcal{B}^{\f12}} \nonumber \\ +\|e^{a|D_x|}(u_0+u_1,\eps (v_0 + v_1)) \|_{\mathcal{B}^{\f12}} + \|e^{a|D_x|}(u_1,\eps v_1) \|_{\mathcal{B}^{\f12}}\leq c_1a,
	\end{align}
	for some $c_1$ sufficiently small. Then System \eqref{eq:HPNhydroPE} has a unique global solution $(u,v)$, so that
	
	\begin{multline} \label{estimate2}
	 \Big( \frac{1}{2}\|e^{\mathcal{R} t}(u + \pa_tu,\eps(v+ \pa_t v))_\Theta \|_{\tilde{L}_t^\infty(\mathcal{B}^{\f12})} + \|e^{\mathcal{R} t} \pa_y(u,\eps v)_\Theta \|_{\tilde{L}_t^\infty(\mathcal{B}^{\f12})} + \eps \|e^{\mathcal{R} t} \pa_x(u,\eps v)_\Theta \|_{\tilde{L}_t^\infty(\mathcal{B}^{\f12})}\Big) \\ + \frac{1}{2}\|e^{\mathcal{R} t}(\pa_tu,\eps \pa_tv)_\Theta\|_{\tilde{L}_t^\infty(\mathcal{B}^{\f12})}  + \|e^{\mathcal{R} t}(\pa_tu,\eps \pa_tv)_\Theta\|_{\tilde{L}_t^2(\mathcal{B}^{\f12})}  + \|e^{\mathcal{R} t}\pa_y(u,\eps v)_\Theta\|_{\tilde{L}_t^2(\mathcal{B}^{\f12})} \\ + \eps \|e^{\mathcal{R} t}\pa_x(u,\eps v)_\Theta\|_{\tilde{L}_t^2(\mathcal{B}^{\f12})}   \leq C\Big(\|e^{a|D_x|}\pa_y(u_0,\eps v_0) \|_{\mathcal{B}^{\f12}} + \eps \|e^{a|D_x|}\pa_x(u_0,\eps v_0) \|_{\mathcal{B}^{\f12}} \\+ \|e^{a|D_x|}(u_1,\eps v_1) \|_{\mathcal{B}^{\f12}} +\|e^{a|D_x|}(u_0+u_1,\eps (v_0 + v_1)) \|_{\mathcal{B}^{\f12}}\Big),
\end{multline}
where $(u_\Theta,v_\Theta)$ is given by \eqref{analybis}.
\end{theorem}

The main idea to prove the above two theorems is to control the new unknown $u_\phi$ defined by \eqref{eq:Anphichp3}, where $u$ is the horizontal velocity and $u_\phi $ is a weighted function of $u$ in the dual Fourier variable with an exponential function of $(a - \lambda \theta(t))|\xi|$. By the classical Cauchy-Kovalevskaya theorem, one expects the radius of the analytically of the solutions to decay in time and so the exponent, which corresponds to the width of the analytical, is allowed to vary with time. Using energy estimates on the equation satisfied by $u_\phi $ and the control of the quantity which describes " the loss of the analytical radius ", we shall show that the analytical persists globally in time. Consequently, our result is a global Cauchy-Kovalevskaya type theorem.

The third result concern the study of the convergence from the scaled anisotropic perturbed Navier-Stokes system \eqref{eq:HPNhydroPE} to the limit system  \eqref{eq:HPNhydrolimit}, so in this theorem, we proved that the convergence across globally in time. 
\begin{theorem}\label{th:33}
Let $a>0$, and $(u_0^\eps,v_0^\eps)$ satisfying \eqref{u0v0}. Let $(u_0,u_1)$ satisfy $e^{a|D_x|} (u_0,u_1) \in (\mathcal{B}^\f12 \cap \mathcal{B}^{\f72})^2$, $e^{a|D_x|} \pa_y (u_0,u_1) \in (\mathcal{B}^{\f32})^2$, and the  compatibility condition $ \int_0^1 u_0 dy =0$ and 
\begin{align}\label{smallnn}
&\norm{e^{a|D_x|} (u_0+u_1)}_{\mathcal{B}^\f12} + \norm{e^{a|D_x|} \pa_y u_0}_{\mathcal{B}^\f12} + \|e^{a|D_x|}u_1 \|_{\mathcal{B}^{\f12}} \nonumber \\ &\quad\quad\quad \leq \frac{c_2 a}{2+ \norm{e^{a|D_x|} (u_0+u_1)}_{\mathcal{B}^{\f32}} + \norm{e^{a|D_x|} \pa_y u_0}_{\mathcal{B}^{\f32}}+ \|e^{a|D_x|}u_1 \|_{\mathcal{B}^{\f32}}},
\end{align}
for some $c_2$ sufficiently small, then we have 

\begin{multline} \label{estimate3}
	 \Big( \frac{1}{2}\|(R^1 + \pa_tR^1,\eps (R^2+\pa_tR^2) )_\varphi \|_{\tilde{L}_t^\infty(\mathcal{B}^{\f12})} + \| \pa_y(R^1,\eps R^2)_\varphi \|_{\tilde{L}_t^\infty(\mathcal{B}^{\f12})} + \eps \| \pa_x(R^1,\eps R^2)_\varphi \|_{\tilde{L}_t^\infty(\mathcal{B}^{\f12})} \\ + \frac{1}{2} \|(\pa_tR^1)_\varphi,\eps (\pa_tR^2)_\varphi \|_{\tilde{L}_t^\infty(\mathcal{B}^{\f12})}\Big)+\|(\pa_tR^1,\eps\pa_tR^2)_\varphi\|_{\tilde{L}_t^2(\mathcal{B}^{\f12})} + \|\pa_y(R^1,\eps R^2)_\varphi\|_{\tilde{L}_t^2(\mathcal{B}^{\f12})} \\ +\eps \|\pa_x(R^1,\eps R^2)_\varphi\|_{\tilde{L}_t^2(\mathcal{B}^{\f12})} 
	\leq \|e^{a|D_x|}((u^\eps_1 -u_1),\eps(v^\eps_1 - v_1)) \|_{\mathcal{B}^{\f12}} \\ +C\Big( \|e^{a|D_x|}\pa_y(u^\eps_0 - u_0,\eps (v^\eps_0 - v_0)) \|_{\mathcal{B}^{\f12}} + \eps \|e^{a|D_x|}\pa_x(u^\eps_0 - u_0,\eps (v^\eps_0 - v_0)) \|_{\mathcal{B}^{\f12}}  \\ +\|e^{a|D_x|}((u^\eps_0-u_0)+(u^\eps_1 -u_1),\eps (v^\eps_0 - v_0) + \eps(v^\eps_1 - v_1)) \|_{\mathcal{B}^{\f12}} + M\eps \Big).
\end{multline}
where 
\begin{align} \label{eq:psiphiqBIH}
	\left\{
	\begin{aligned}
		R^{1} = u^\eps - u, \\
		R^{2} = v^\eps - v,
	\end{aligned}
		\right.
\end{align}
and $v_0$ is determined from $u_0$ via $\pa_x u +\pa_y v =0 $ and $v_0|_{y=0} = v_0|_{y=1}= 0$, $(R_\varphi^1,\eps R_\varphi^2)$ is given by \eqref{analytiqueff**}.
\end{theorem}

We remark that without the smallness conditions \eqref{u0v0} and \eqref{smallnn}, we can not prove the convergence from the solutions to System \eqref{eq:HPNhydroPE} to the solutions System \eqref{eq:HPNhydrolimit} on a fixed time interval $[0,t]$ for $t <T^\star$, where $T^\star$ is the lifetime of the solution of the hydrostatic perturbed Navier-Stokes equation with the large initial data $u_0$.

\textbf{Organisation of the paper}: Our paper will be divided into several sections as follows. In section \ref{sec:1}, we present some basic notions of the Littlewood-Paley Theory and some technical lemmas. In Section \ref{sec:3}, we prove the global wellposedness of the system \eqref{eq:HPNhydrolimit} for small data in the analytic framework. Section \ref{sec:4} is devoted to the study of the system \eqref{eq:HPNhydroPE} and the proof of Theorem \ref{th:22}. In section \ref{sec:5} we present some proposition stating the propagation for any $\mathcal{B}^s$ regularity. In Section \ref{sec:6}, we prove the convergence of System \eqref{eq:HPNhydroPE} towards System \eqref{eq:HPNhydrolimit} when $\eps \to 0$. Finally, in the last section, we give the proofs of some technical estimates.

We end this introduction by the notations that will be used in all that follows. By $f\lesssim g$, we mean that there is a uniform constant $C$, which may be different from line to line, such that $f\leq Cg$. We denote by $\psca{f,g}_{L^2}$ the inner product of $f$ and $g$ in $L^2(\mathcal{S})$. Finally, we denote by $(d_q)_{q\in\ZZ}$(resp. $(d_q(t))_{q\in\ZZ}$) to be a generic element of $ \ell^1(\ZZ)$ so that $\sum_{q\in\ZZ} d_q = 1$ (resp. $\sum_{q\in\ZZ} d_q(t) = 1$).

\section{Littlewood-Paley Theory and Some technical lemmas} \label{sec:1}

\subsection{Littlewood-Paley Theory}
To introduce the result of this paper, we will recall some elements of the Littlewood-Paley theory and also introduce the function space and technique used in the proof of our result. We define the dyadic operator in the horizontal variable, (of $x$ variable) and for all $q \in \ZZ$, we recall from \cite{BCD2011book} that  
	\begin{align*}
	\Delta_q^h f(x,y) &= \mathcal{F}_h^{-1}\pare{\varphi(2^{-q}\abs{\xi})\widehat{f}(\xi,y)},\\
	S_q^h f(x,y) &= \mathcal{F}_h^{-1}\pare{\psi(2^{-q}\abs{\xi})\widehat{f}(\xi,y)}.
	\end{align*}
	
where $\psi$ and $\varphi$ are a smooth function such that 
\begin{align*}
     &supp \ \varphi \subset \lbrace z\in \RR/ \text{ \ } \frac{3}{4} \leq |z| \leq \frac{8}{3} \rbrace \text{ \ and \ } \forall z \neq 0, \text{ \ } \sum_{q\in \ZZ} \varphi (2^{-q} z) =1,\\
     & supp \ \psi \subset \lbrace z\in \RR/ \text{ \ } |z| \leq \frac{4}{3} \rbrace \text{ \ and \ } \text{ \ } \psi(z) +  \sum_{q \geq 0} \varphi (2^{-q} z) =1.
 \end{align*}
 and 
 \begin{equation*}
\forall\, q,q' \in \NN, \; \abs{q-q'}\geq 2, \quad \supp \varphi(2^{-q} \cdot) \cap \supp \varphi(2^{-q'} \cdot) = \emptyset.
\end{equation*}
And in all that follows, $\mathcal{F}_hf$ and $\widehat{f}$ always denote the partial Fourier transform of the distribution $f$ with respect to the horizontal variable (of $x$ variable), that is, $\widehat{f}(\xi,y) = \mathcal{F}_{x \rightarrow \xi} (f)(\xi,y) = \mathcal{F}_{h} (f)(\xi,y).$ We refer to \cite{BCD2011book} and \cite{B1981} for a more detailed construction of the dyadic decomposition. Combining the definition of the dyadic operator to the fact that 
\begin{equation}
\label{decompositionbis1} \forall z\in \RR, \quad\psi(z) + \sum_{j\in\NN} \varphi (2^{-j} z) = 1,
\end{equation}
implies that all tempered distributions can be decomposed with respect to the horizontal frequencies as $$ f = \sum_{q \in \ZZ} \Delta_q^h f.$$
We now introduce the function spaces used throughout the paper. As in \cite{PZZ2019}, we define the anisotropic Besov-type spaces $\mathcal{B}^s$, $s\in\RR$ as follows.
\begin{defin}
	Let $s \in \RR$ and $\mathcal{S} = \RR \times ]0,1[$. For all tempered distributions $u \in S'_h(\mathcal{S})$, \emph{i.e.}, $f$ belongs to $S'(\mathcal{S})$ and $\lim_{q\to -\infty} \norm{S^h_q f}_{L^\infty} = 0$, we set
	\begin{equation*}
	\norm{f}_{\mathcal{B}^{s}} \stackrel{\tiny def}{=} \;\| (2^{qs} \norm{\Delta_q^h f}_{L^2})_{q \in \ZZ} \|_{\ell^1(\ZZ)}.
	\end{equation*}
	\begin{enumerate}
		\item[(i)] For $s \leq \frac{1}{2}$, we define 
		\begin{equation*}
		\mathcal{B}^s(\mathcal{S}) \stackrel{\tiny def}{=} \; \set{f \in S'_h(\mathcal{S}) \ : \ \norm{f}_{\mathcal{B}^s} < +\infty}.
		\end{equation*}
		\item[(ii)] For $s \in \; ]k-\frac{1}{2}, k+\frac{1}{2}]$, with $k \in \NN^*$, we define $\mathcal{B}^s(\mathcal{S})$ as the subset of distributions $f$ in $S'_h(\mathcal{S})$ such that $\dd_x^kf \in \mathcal{B}^{s-k}(\mathcal{S})$.
	\end{enumerate}
\end{defin}
For a better use of the smoothing effect given by the diffusion terms, we will work in the following Chemin-Lerner type spaces and also the time-weighted Chemin-Lerner type spaces.
\begin{defin} \label{def:CLspacesbis1}
	Let $p \in [1,+\infty]$ and $T \in ]0,+\infty]$. Then, the space $\tilde{L}^p_T(\mathcal{B}^s(\mathcal{S}))$ is the closure of $C([0,T];S(\mathcal{S}))$ under the norm
	\begin{equation*}
	\norm{f}_{\tilde{L}^p_T(\mathcal{B}^s(\mathcal{S}))} \stackrel{\tiny def}{=} \; \sum_{q \in \ZZ} 2^{qs} \pare{\int_0^T \norm{\DD^h_q f(t)}_{L^2}^p dt}^{\frac{1}{p}},
	\end{equation*}
	with the usual change if $p = +\infty$.
\end{defin} \label{def:wCLspacesbis1}
\begin{defin}
	\label{def:CLweightbis12}
	Let $p \in [1,+\infty]$ and let $\delta \in L^1_{\tiny loc}(\RR_+)$ be a non negative function. Then, the space $\tilde{L}^p_{t,\delta(t)}(\mathcal{B}^s(\mathcal{S}))$ is the closure of $C([0,T];S(\mathcal{S}))$ under the norm
	\begin{equation*}
	\norm{f}_{\tilde{L}^p_{t,\delta(t)}(\mathcal{B}^s(\mathcal{S}))} \stackrel{\tiny def}{=} \; \sum_{q\in\ZZ} 2^{qs}\left( \int_0^t \delta(t') \norm{\DD^h_q f(t')}_{L^2}^p dt'\right)^{\frac{1}{p}}.
	\end{equation*}
\end{defin}

The following Bernstein lemma gives important properties of a distribution $u$ when its Fourier transform is well localized. We refer the reader to \cite{C1995book} for the proof of this lemma.
\begin{lemma}
	\label{lem:Bernsteinbis1}
	Let $k\in\NN$, $d \in \NN^*$ and $r_1, r_2 \in \RR$ satisfy $0 < r_1 < r_2$. There exists a constant $C > 0$ such that, for any $a, b \in \RR$, $1 \leq a \leq b \leq +\infty$, for any $\lambda > 0$ and for any $f \in L^a(\RR^d)$, we have
	\begin{equation*} 
	\supp\pare{\widehat{f}} \subset \set{\xi \in \RR^d \;\vert\; \abs{\xi} \leq r_1\lambda} \quad \Longrightarrow \quad \sup_{\abs{\alpha} = k} \norm{\dd^\alpha f}_{L^b} \leq C^k\lambda^{k+ d \pare{\frac{1}{a}-\frac{1}{b}}} \norm{f}_{L^a},
	\end{equation*}
	and
	\begin{equation*} 
	\supp\pare{\widehat{f}} \subset \set{\xi \in \RR^d \;\vert\; r_1\lambda \leq \abs{\xi} \leq r_2\lambda} \quad \Longrightarrow \quad C^{-k} \lambda^k\norm{f}_{L^a} \leq \sup_{\abs{\alpha} = k} \norm{\dd^\alpha f}_{L^a} \leq C^k \lambda^k\norm{f}_{L^a}.
	\end{equation*}
\end{lemma}

Finally to deal with the estimate concerning the product of two distributions, we shall frequently use on Bony's decomposition (see \cite{B1981} ) in the horizontal variable ( $x$ variable ) that for $f,g$ two tempered distributions  :
\begin{equation} \label{eq:Bonybis1}
	fg = T^h_f g + T^h_g f + R^h(f,g),
	\end{equation}
	where
	$$ T^h_f g = \sum_{q} S^h_{q-1} f \Delta_q^h g, \text{ \ } T^h_g f = \sum_{q} S^h_{q-1} g \Delta_q^h f $$ and the rest term satisfied
	
	$$ R^h(f,g) = \sum_{q} \tilde{\Delta}^h_{q} f \Delta_q^h g \text{ \ with \ } \tilde{\Delta}^h_{q}f = \sum_{|q-q'| \leq 1} \Delta_{q'}^h f. $$
	\subsection{Technical lemmas}
	Our main difficulty relies on finding a way to estimate the nonlinear terms, which allows exploiting the smoothing effect given by the above function spaces. Using the method introduced by Chemin in \cite{C2004} (see also \cite{CGP2011}, \cite{PZ2011} or \cite{PZZ2019}), for any $f \in L^2(\mathcal{S})$, we define the following auxiliary function, which allows to control the analyticity of $u$ in the horizontal variable $x$, 
\begin{equation} 
	\label{eq:Anphichp3}
	\left\{\;
	\begin{aligned}
		&f_{\phi}(t,x,y) = e^{\phi(t,D_x)} f(t,x,y) \stackrel{\tiny def}{=} \mathcal{F}^{-1}_h(e^{\phi(t,\xi)} \mathcal{F}_h(f)(t,\xi,y))\\
		&\phi(t,\xi) = (a- \lambda \theta(t))|\xi|.
	\end{aligned}
	\right.
\end{equation}
where the quantity $\theta(t)$, which describes the evolution of the analytic band of $f$, satisfies
\begin{equation}
\label{eq:AnThetabis1} \forall\, t > 0, \ \dot{\theta}(t) \geq 0 \text{ \ \ and \ \ } \theta(0) = 0. 
\end{equation}
\begin{rem}
In the proofs we need to choose that $\theta(t)$ is satisfying an ordinary differential equation 
$$ \dot{\theta}(t) = \| e^{(a-\lambda\theta(t))|D_x|} \pa_y f\|_{\mathcal{B}^\f12} $$
The local existence of $\theta(t)$ of $t$ between $[0,T[$ is given by the Cauchy-Lipschitz theorem. In all our computation we will be in time interval $[0,\tilde{T}^\star[$ where $\tilde{T}^\star$ is maximal life-time of the solution.  
\end{rem}
In what follows, we shall always assume that $t<T^\star$, with $T^\star$ is determined by 
 \begin{equation} \label{eq:Tstarbis1} 
	T^\star \stackrel{\tiny def}{=} \sup\set{t\in[0,\tilde{T}^\star[\ : \ \norm{f_{\phi}}_{\mathcal{B}^{\f12}} \leq \frac{1}{4C^2} \ \mbox{ and } \ \theta(t) \leq  \frac{a}{\lambda}}.
\end{equation}

By virtue of \eqref{eq:Anphichp3} for any  $t<T^\star$, there holds the following convex inequality
 \begin{equation} \label{eq:phiconvexbis1} 
	\phi(t,\xi) \leq \phi(t,\xi-\eta) + \phi(t,\eta) \text{ \ \ } \forall \xi,\eta \in \mathbb{R}.
\end{equation}

	Before starting the obtained result, we need the following lemma to characterize the product $(fg)_\phi$, indeed this product will be useful in all the rest of the paper. 
\begin{cor}\label{lem:fgphibis1}
Let $f\in L^2_x$, $g\in L^2_x$, we define $f^+ = \mathcal{F}_\xi^{-1}(|\mathcal{F}_x(f)|)$ then, we have 
$$ |\widehat{(fg)_\phi}(\xi)| \leq \widehat{f^+_\phi g^+_\phi }(\xi) \ \ \text{ and } \ \   \|f^+\|_{L^2_x} = \|f\|_{L^2_x}$$
\end{cor} 
\begin{proof}
Let us consider $f$, and $g$ two functions in $L^2_x$, we have 
\begin{align*}
|(\widehat{fg})_\phi(\xi)| &= e^{\phi(\xi)} |\widehat{f}(.)*\widehat{g}(.)(\xi)|\\ &\leq e^{\phi(\xi)} \int |\widehat{f}(\xi - \eta)||\widehat{g}(\eta)|d\eta,
\end{align*}
By virtue of the definition of Function $\phi$ we have $e^{\phi(\xi)} >0$ and $ e^{\phi(\xi)} \leq e^{\phi(\xi-\eta)}e^{\phi(\eta)}$, thus 
\begin{align*}
|(\widehat{fg})_\phi(\xi)| &\leq \int e^{\phi(\xi-\eta)}|\widehat{f}(\xi - \eta)| e^{\phi(\eta)}|\widehat{g}(\eta)|d\eta \\
&\leq \int |\widehat{f}_\phi(\xi - \eta)||\widehat{g}_\phi(\eta)|d\eta \\
&\leq |\widehat{f}_\phi| *|\widehat{g}_\phi|(\xi) = \mathcal{F}_{\xi} \big(f^+_\phi g^+_\phi\big) = \widehat{f^+_\phi g^+_\phi }(\xi)
\end{align*}

The second point of the lemma is trivial.
\end{proof}
\begin{lemma}\label{estfg+}
For smooth functions we have 
$$ \| (fg)_\phi \|_{L^2} \leq \| (f_\phi^+) \|_{L^\infty} \| (g^+_\phi)\|_{L^2}. $$
\end{lemma}

\begin{proof}
Let us consider $f$, and $g$ two functions in $L^2_x$, we have by using Plancherel theorem that  
 \begin{align*}
 \| \big( fg \big)_\phi \|_{L^2}^2 &= \int |\mathcal{F}_{\xi} (f g)_\phi(\xi,y)|^2 d\xi \\
 & \leq \int |\int e^{\phi(\xi)} \mathcal{F}_{\xi}(f)(\xi -\eta ) \mathcal{F}_{\xi}(g)(\eta) d\eta |^2 d\xi \\
 &\leq \int \int e^{2\phi(\xi)} |\mathcal{F}_{\xi}(f)(\xi -\eta )|^2 |\mathcal{F}_{\xi}(g)(\eta)|^2 d\eta d\xi.
 \end{align*}
 We recall that the function $\phi$ verify the following inequality $e^{\phi(\xi)} \leq e^{\phi(\xi - \eta)} e^{\phi(\eta)}$, thus
  \begin{align*}
 \| \big( fg \big)_\phi \|_{L^2}^2 &\leq \| f^+_\phi g^+_\phi \|_{L^2}^2\\
 & \leq \| (f^+_\phi) \|_{L^\infty}^2 \| (g^+_\phi) \|_{L^2}^2,
 \end{align*}
\end{proof}

\begin{cor}\label{coro}
For any $f$ and $g$ in $L^2_x$, we have 
$$|\widehat{(T_f g)_\phi}(\xi)| \leq \widehat{ (T_{f^+_\phi} g^+_\phi)}(\xi) \text{ \ and \ } |\widehat{R(f,g)_\phi}(\xi)| \leq \widehat{ R(f^+_\phi,g^+_\phi)}(\xi).$$
\end{cor}

 We next present the weighted energy estimate for the linear heat equations
 \begin{lemma}\label{lem:1}
 Let $f$ and $g$ two smooth enough functions on $ \mathbb{R}\times (0,1)$, satisfying the Dirichlet boundary condition \ref{S1eq7}, then we have
 
 \begin{enumerate}
     \item \begin{align}\label{pa_tfg}
\frac{d}{dt}\psca{\Delta_q^h  f_\phi, \Delta_q^h g_\phi }_{L^2} &= \psca{\Delta_q^h(\pa_t f)_\phi, \Delta_q^h g_\phi} + \psca{\Delta_q^h f_\phi, \Delta_q^h (\pa_t g)_\phi}\nonumber \\ &- 2\lambda \dot{\theta}(t) \psca{\Delta_q^h|D_x|^\frac{1}{2} f_\phi, \Delta_q^h |D_x|^\frac{1}{2}  g_\phi}, 
\end{align}  
 
 In particular if $f = g$, we obtain that 
 
 \begin{align}\label{pa_tff}
\frac{1}{2} \frac{d}{dt}\norm{\Delta_q^h  f_\phi}_{L^2}^2 = \psca{\Delta_q^h(\pa_t f)_\phi, \Delta_q^h f_\phi}_{L^2} - \lambda \dot{\theta}(t) \norm{\Delta_q^h|D_x|^\frac{1}{2} f_\phi}_{L^2}^2, 
\end{align}  

\item \begin{align}\label{pa_ttf}
&\psca{\Delta_q^h \Big((\pa_t^2 f)_\phi\Big), \Delta_q^h (\pa_tf)_\phi }_{L^2} = \frac{1}{2} \frac{d}{dt} \| \Delta_q^h (\pa_t f)_\phi\|_{L^2}^2  \nonumber \\ &\quad\quad\quad\quad\quad  +\lambda \dot{\theta}(t) \|\Delta_q^h |D_x|^\frac{1}{2} (\pa_tf)_\phi\|_{L^2}^2, 
\end{align}
 
\item \begin{align}\label{pa_yyf}
&\psca{\Delta_q^h \Big(- \pa_y^2 f_\phi\Big), \Delta_q^h (\pa_tf)_\phi }_{L^2} = \frac{1}{2} \frac{d}{dt} \|\Delta_q^h  \pa_y f_\phi \|_{L^2}^2 
 \nonumber \\ &\quad\quad\quad\quad\quad  + \lambda \dot{\theta}(t) \|\Delta_q^h |D_x|^\frac{1}{2} \pa_yf_\phi\|_{L^2}^2,
\end{align}

\end{enumerate}
 \end{lemma}
 \begin{proof}
To prove the first assertion, we apply the rules of the derivation of a product, we obtain  
   \begin{align*}
\frac{d}{dt}&\psca{\Delta_q^h  f_\phi, \Delta_q^h g_\phi }_{L^2} = \psca{\Delta_q^h \pa_t f_\phi, \Delta_q^h g_\phi}_{L^2}  + \psca{\Delta_q^h f_\phi, \Delta_q^h \pa_t g_\phi}_{L^2} \\ &= \psca{\Delta_q^h(\pa_t f)_\phi, \Delta_q^h g_\phi}_{L^2} - \lambda \dot{\theta}(t) \psca{\Delta_q^h |D_x| f_\phi, \Delta_q^h  g_\phi}_{L^2} \\ & + \psca{\Delta_q^h f_\phi, \Delta_q^h (\pa_t g)_\phi}_{L^2} - \lambda \dot{\theta}(t) \psca{\Delta_q^h f_\phi, \Delta_q^h |D_x|  g_\phi}_{L^2} \\
& = \psca{\Delta_q^h (\pa_t f)_\phi, \Delta_q^h  g_\phi}_{L^2} + \psca{\Delta_q^h f_\phi, \Delta_q^h (\pa_t g)_\phi}_{L^2} - 2\lambda \dot{\theta}(t) \psca{\Delta_q^h |D_x|^\frac{1}{2} f_\phi, \Delta_q^h |D_x|^\frac{1}{2}  g_\phi}_{L^2}.
\end{align*}

 By using the rules of the derivation of a product and Parseval equality, we find 
 \begin{align*}
 &\psca{\Delta_q^h \Big((\pa_t^2 f)_\phi\Big), \Delta_q^h (\pa_tf)_\phi }_{L^2} \\ &= \int \Delta_q^h (e^{\phi(t,|D_x|)}\pa_t^2 f) \overline{\Delta_q^h (e^{\phi(t,|D_x|)} \pa_t f)}dx \\
 &= \int \widehat{\Delta_q^h (e^{\phi(t,|D_x|)}\pa_t^2 f)} \overline{\widehat{\Delta_q^h (e^{\phi(t,|D_x|)} \pa_t f)}}d\xi_h \\
 & =\frac{1}{2} \frac{d}{dt} \int |\widehat{\Delta_q^h (e^{\phi(t,|D_x|)} \pa_t f)}|^2 d\xi_h \\ & \quad\quad\quad+   \frac{1}{2} \int 2\lambda \dot{\theta}(t)  \varphi(2^{-q}\xi) |\xi_h| e^{\phi(t,|\xi_h|)}\pa_t \widehat{f}(\xi) \overline{\varphi(2^{-q}\xi) e^{\phi(t,|\xi_h|)}\pa_t \widehat{f}(\xi)} d\xi_h \\ &= \frac{1}{2} \frac{d}{dt} \|\Delta_q^h (\pa_t f)_\phi\|_{L^2}^2  + \lambda \dot{\theta}(t) \int |\widehat{\Delta_q^h |D_x|^{\frac{1}{2}}(e^{\phi(t,|D_x|)} \pa_t f)}|^2 d\xi_h\\
 & = \frac{1}{2} \frac{d}{dt} \| \Delta_q^h (\pa_t f)_\phi\|_{L^2}^2 +\lambda \dot{\theta}(t) \|\Delta_q^h |D_x|^\frac{1}{2} (\pa_tf)_\phi\|_{L^2}^2.
  \end{align*}
 
  By using integration by parts and  replacing in \eqref{pa_tff} by $\pa_y f$, we can find the estimate 
  \begin{align*}
  &\psca{\Delta_q^h \Big(- \pa_y^2 f_\phi\Big), \Delta_q^h (\pa_tf)_\phi }_{L^2} = \psca{\Delta_q^h \Big(\pa_y f_\phi\Big), \Delta_q^h (\pa_t \pa_y f)_\phi }_{L^2} \\
  & = \frac{1}{2} \frac{d}{dt} \|\Delta_q^h \pa_y f_\phi \|_{L^2}^2 + \lambda \dot{\theta}(t) \|\Delta_q^h |D_x|^\frac{1}{2} \pa_yf_\phi\|_{L^2}^2.
  \end{align*}
  Note that here we use the fact that $$\Delta_q^h(\pa_x^\alpha f_\phi) = \pa_x^\alpha(\Delta_q^h f)_\phi. $$
  
 \end{proof}
 
After starting the proof of the mains lemmas we recall from Plancherel formula, Fubini's the following inequality 
\begin{cor}
Let $f, g$ and $h$ be a smooth functions on $\mathbb{R}\times (0,1)$, we have 
\begin{align*}
\|\Delta_q^h f_\phi \| \leq C \|f_\phi\|_{L^2}
\end{align*}
and 
\begin{align*}
&\big|\psca{\Delta_q^h \big(fg)_\phi\big), \Delta_q^h h_\phi }_{L^2}\big|
\leq \sum_{|q-q'|\leq 4} \big|\psca{\Delta_q^h \big(S_{q'-1}f\Delta_{q'}g)_\phi\big), \Delta_q^h h_\phi }_{L^2}\big| \\ &+ \sum_{|q-q'|\leq 4}\big|\psca{\Delta_q^h \big(S_{q'-1}g\Delta_{q'}f)_\phi\big), \Delta_q^h h_\phi }_{L^2}\big| + \sum_{q'\geq q- 3} \big|\psca{\Delta_q^h \big(\tilde{\Delta}_{q'}f\Delta_{q'}g)_\phi\big), \Delta_q^h h_\phi }_{L^2}\big| , 
\end{align*}
such as
\begin{align*}
 \big|\psca{\Delta_q^h \big(S_{q'-1}f\Delta_{q'}g)_\phi\big), \Delta_q^h h_\phi }_{L^2}\big| &\leq C \| S_{q'-1}f_\phi^+ \|_{L^\infty} \| \Delta_{q'}g_\phi^+ \|_{L^2} \|\Delta_q^h h_\phi^+ \|_{L^2} \\
 &\leq C \| S_{q'-1}f_\phi^+ \|_{L^\infty_h(L^2_v)} \| \Delta_{q'}g_\phi^+ \|_{L^2_h(L^\infty_v)} \|\Delta_q^h h_\phi^+ \|_{L^2} 
\end{align*}
and 
\begin{align*}
 \big|\psca{\Delta_q^h \big(\tilde{\Delta}_{q'}f\Delta_{q'}g)_\phi\big), \Delta_q^h h_\phi }_{L^2}\big| &\leq C \| \tilde{\Delta}_{q'}f_\phi^+ \|_{L^\infty} \| \Delta_{q'}g_\phi^+ \|_{L^2} \|\Delta_q^h h_\phi^+ \|_{L^2} \\
 &\leq C \| \tilde{\Delta}_{q'}f_\phi^+ \|_{L^\infty_h(L^2_v)} \| \Delta_{q'}g_\phi^+ \|_{L^2_h(L^\infty_v)} \|\Delta_q^h h_\phi^+ \|_{L^2} \\
 &\leq C \| \tilde{\Delta}_{q'}f_\phi^+ \|_{L^2_h(L^\infty_v)} \| \Delta_{q'}g_\phi^+ \|_{L^2} \|\Delta_q^h h_\phi^+ \|_{L^\infty_h(L^2_v)}.
\end{align*}
\end{cor}
\begin{rem}
Using this corollary, we can assume from now without losing the generality with function $f$, that $\hat{f}\geq 0$.
\end{rem}
\begin{lemma}
Let $f$ be a smooth function on $\mathbb{R}\times (0,1)$. If $f = 0$ when $y=0$ then we have 
\begin{align}
\|f\|_{L^\infty_v} \leq C \|f\|^\f12_{L^2_v} \|\pa_y f\|^\f12_{L^2_v}.
\end{align}
\end{lemma}
\begin{proof}
Let $f$ be a smooth function on $\mathbb{R}\times (0,1)$, while $f=0$ on $y=0$ then we can write 
\begin{align*}
(f(x,y))^2 = \int_0^y \pa_y (f(.,y'))^2 dy' = 2 \int_0^y f(.,y') \pa_y f(.,y') dy',
\end{align*} 
by using Cauchy-Schwarz inequality, we obtain 
\begin{align*}
(f(x,y))^2 \leq 2 \|f\|_{L^2} \|\pa_y f\|_{L^2},
\end{align*} 
then,
\begin{align*}
\|f\|_{L^\infty_v} \leq \sqrt{2} \|f\|_{L^2}^\f12 \|\pa_y f\|_{L^2}^\f12,
\end{align*} 
\end{proof}
We also apply the following Poincaré inequality for a smooth function $f$ vanishing on the boundary (for $y=0,1$)
$$ \|f\|_{L^\infty_v} \leq  \|\pa_y f\|_{L^2_v} \text{ \ and \ } \|f\|_{L^2_v} \leq  \|\pa_y f \|_{L^2_v}. $$

Now, let us state our lemmas, which we use to prove our results.

\begin{lemma} \label{lem:ABCbis1}
	Let $s\in ]0,1[$, and $A$, $B$ and $C$ be smooth enough functions on $[0,T^*)\times\mathbb{R}\times (0,1)$, with $A$ vanishing on the boundary, let $\phi$ be defined as in \eqref{eq:Anphichp3} with $\dot{\theta}(t) = \Vert \pa_y A_{\phi}(t) \Vert_{\mathcal{B}^{\f12}}$. For any $t \in[0,T^*[$ and for any $B, C \in \tilde{L}^2_{t,\dot{\theta}(t)}(\mathcal{B}^{s+\frac{1}{2}})$, we have
	\begin{align}
	 \int_0^t &\abs{\psca{e^{\mathcal{R} t'} \Delta_q^h (A\pa_x B)_{\phi}, e^{\mathcal{R} t'} \Delta_q^h C_{\phi}}_{L^2}} dt' \nonumber \\ &\lesssim 2^{-2qs}d^2_q \Vert e^{\mathcal{R} t} B_{\phi} \Vert_{\tilde{L}^2_{t,\dot{\theta}(t)}(\mathcal{B}^{s+\frac{1}{2}})} \Vert e^{\mathcal{R} t} C_{\phi} \Vert_{\tilde{L}^2_{t,\dot{\theta}(t)}(\mathcal{B}^{s+\frac{1}{2}})}.
	\end{align}
	with $d_q\in \ell^1(\mathbb{Z})$, $\sum\limits_{q\in\mathbb{Z}}d_q\leq 1$ and where $\mathcal{R}$ is the Poincaré constant.
\end{lemma}
\begin{rem}
We can note that $\theta(t)$ verifying an ordinary differential equation 
$$ \dot{\theta}(t) = \| e^{(a-\lambda\theta(t))|D_x|} \pa_y A\|_{\mathcal{B}^\f12} $$
The local existence of $\theta(t)$ of $t$ between $[0,\tilde T^*[$ is given by Cauchy Lipschitz. 
\end{rem}
\begin{proof}
As in \cite{BCD2011book}, using Bony's homogeneous decomposition of $A\pa_x B$  into para-products in the horizontal variable, we can write

$$ A\pa_x B = T^{h}_A \pa_{x} B + T^h_{\pa_x B} A + R^{h}(A,\pa_x B) $$
where,
\begin{align*}
	T_A^h \pa_xB = \sum_{q\in\ZZ} S_{q-1}^h A \DD_q^h \pa_x B \quad\mbox{ and }\quad R^h(A,\pa_x B) = \sum_{\abs{q'-q} \leq 1} \DD_q^h A \DD_{q'}^h \pa_x B = \sum_q \tilde{\DD}_q^h A .\DD_{q}^h \pa_x B.
\end{align*}

We have the following bound

\begin{align*}
	\int_0^t \abs{\psca{e^{\mathcal{R} t'} \Delta_q^h (A\pa_x B)_{\phi}, e^{\mathcal{R} t'} \Delta_q^h C_{\phi}}_{L^2}} dt' \leq M_{1,q} + M_{2,q} + M_{3,q}, 
\end{align*}
where
\begin{align*}
	M_{1,q} &= \int_0^t \abs{\psca{ e^{\mathcal{R} t'} \Delta_q^h ( T^h_{A}\pa_x B)_{\phi}, e^{\mathcal{R} t'} \Delta_q^h C_{\phi}}_{L^2}} dt'\\
	M_{2,q} &= \int_0^t \abs{\psca{ e^{\mathcal{R} t'} \Delta_q^h ( T^h_{\pa_xB} A )_{\phi}, e^{\mathcal{R} t'} \Delta_q^h C_{\phi}}_{L^2}} dt'\\
	M_{3,q} &= \int_0^t \abs{\psca{ e^{\mathcal{R} t'} \Delta_q^h (R^h(A,\pa_xB))_{\phi}, e^{\mathcal{R} t'} \Delta_q^h C_{\phi}}_{L^2}} dt'.
\end{align*}

We start by getting the estimate of the first term $M_{1,q}$, for that we need to use the definition of $T^{h}_{A} \pa_x B$ and  the support properties given in  [\cite{B1981}, Proposition 2.10] and Corollary \ref{coro}. We infer 

\begin{equation} \label{eq:normBCEbis1}
	M_{1,q} \leq \sum_{|q-q'| \leq 4} \int_0^t e^{2\mathcal{R} t'} \Vert S^h_{q'-1} A_{\phi} (t') \Vert_{L^{\infty}}\Vert \Delta_{q'}^h \partial_xB_{\phi} (t') \Vert_{L^2} \Vert \Delta_q^h C_{\phi} (t') \Vert_{L^2} dt'.
\end{equation}

Using  Bernstein Lemma \ref{lem:Bernsteinbis1} we have
 $$\|\Delta_{q'}^h A_\phi \|_{L^\infty_h} \lesssim 2^{\frac{q'}{2}} \|\Delta_{q'}^h A_\phi\|_{L^2_h}.$$
 Taking the $L^\infty_v$ norm  and using that $\|f\|_{L^p_vL^q_h}\leq \|f\|_{L^q_hL^p_v}$, for $q\geq p\geq 1$, we obtain
  $$\|\Delta_{q'}^h A_\phi \|_{L^\infty} \lesssim 2^{\frac{q'}{2}} \|\Delta_{q'}^h A_\phi\|_{L^\infty_vL^2_h}\lesssim 2^{\frac{q'}{2}} \|\Delta_{q'}^h A_\phi\|_{L^2_hL^\infty_v}.$$

While using the inclusion $H^1_y\hookrightarrow L^\infty_y$, 
$$\|\Delta_{q'}^h A_\phi\|_{L^\infty_v} \lesssim \|\Delta_{q'}^h A_\phi\|_{L^2_v}^\f12 \|\Delta_{q'}^h \pa_y A_\phi\|_{L^2_v}^\f12, $$
and the Poincaré inequality on the interval $\lbrace 0<y<1 \rbrace$ on $A$ (as we have that $A = 0$ when y = 0,1)
 $$ \|\Delta_{q'}^h A_\phi\|_{L^2} \lesssim \|\Delta_{q'}^h \pa_y A_\phi \|_{L^2},$$ 
 we obtain
\begin{align} \label{eq:normB12bis1}
	\Vert \Delta^h_q A_{\phi}(t') \Vert_{L^{\infty}} &\lesssim 2^{\frac{q}{2}} \Vert \Delta_q^h A_{\phi}(t') \Vert_{L^2_h(L^{\infty}_v)} \nonumber \\ 
	&\lesssim 2^{\frac{q}{2}} \|\Delta_q^h A_\phi(t') \|^\f12_{L^2} \Vert \Delta_q^h \pa_y A_{\phi}(t') \Vert_{L^2}^\f12 \nonumber \\
	&\lesssim 2^{\frac{q}{2}}\Vert \Delta_q^h \pa_y A_{\phi}(t') \Vert_{L^2} \lesssim d_q(A_\phi) \Vert  \pa_y A_{\phi}(t') \Vert_{\mathcal{B}^{\f12}}.
\end{align}

 Here and in all that follows, we always denote $(d_q(A_\phi))_{q\in\ZZ}$ to be a generic element of $\ell^{1}(\ZZ) $ so that $\sum\limits_{q\in\ZZ} d_q(A_\phi) \leq 1$. 
 
 Then,
\begin{align*}
	\Vert S^h_{q'-1} A_{\phi} (t') \Vert_{L^{\infty}} \lesssim  \Vert  \pa_y A_{\phi}(t') \Vert_{\mathcal{B}^{\f12}},
\end{align*}
and replacing this result in our estimate \eqref{eq:normBCEbis1}, and combining with H\"older inequality, imply that 

\begin{align*}
	M_{1,q} &\lesssim \sum_{|q-q'| \leq 4}  \int_0^t \Vert \pa_y A_{\phi}(t') \Vert_{\mathcal{B}^{\f12}} e^{\mathcal{R} t'} \Vert \Delta_{q'}^h \pa_x B_{\phi}(t') \Vert_{L^2} e^{\mathcal{R} t'} \Vert \Delta_q^h C_{\phi}(t') \Vert_{L^2} dt'\\
	&\lesssim \sum_{|q-q'| \leq 4} 2^{q'} \int_0^t \Vert \pa_y A_{\phi}(t') \Vert_{\mathcal{B}^{\f12}}^\frac{1}{2} e^{\mathcal{R} t'} \Vert \Delta_{q'}^h B_{\phi}(t') \Vert_{L^2} \Vert \pa_y A_{\phi}(t') \Vert_{\mathcal{B}^{\f12}}^\frac{1}{2} e^{\mathcal{R} t'} \Vert \Delta_q^h C_{\phi}(t') \Vert_{L^2} dt'\\
	&\lesssim \sum_{|q-q'| \leq 4} 2^{q'} \left(\int_0^t \Vert \pa_y A_{\phi}(t') \Vert_{\mathcal{B}^{\f12}} e^{2\mathcal{R} t'} \Vert \Delta_{q'}^h B_{\phi} \Vert_{L^2}^2 dt'\right)^{\frac{1}{2}} \left(\int_0^t \Vert  \pa_y A_{\phi}(t') \Vert_{\mathcal{B}^{\f12}} e^{2\mathcal{R} t'} \Vert \Delta_q^h C_{\phi} \Vert_{L^2}^2 dt'\right)^{\frac{1}{2}}.
\end{align*}
We recall that $\dot{\theta}(t) = \Vert  \pa_y A_{\phi}(t) \Vert_{\mathcal{B}^{\f12}}$, and using Definition \ref{def:CLweightbis12}, we have 

\begin{align*}
	\left(\int_0^t \Vert \pa_y A_{\phi}(t') \Vert_{\mathcal{B}^{\f12}} e^{2\mathcal{R} t'} \Vert \Delta_q^h C_{\phi} \Vert_{L^2}^2 dt'\right)^{\frac{1}{2}} &= \left(\int_0^t \dot{\theta}(t') e^{2\mathcal{R} t'} \Vert \Delta_q^h C_{\phi} \Vert_{L^2}^2 dt'\right)^{\frac{1}{2}} \\ &\lesssim 2^{-q(s+\frac{1}{2})} d_q(C_\phi) \Vert e^{\mathcal{R}t} C_{\phi} \Vert_{\tilde{L}^{2}_{t,\dot{\theta}(t)}(\mathcal{B}^{s+\frac{1}{2}})}.
\end{align*}
Then,
\begin{align} \label{eq:I1qbis1}
	M_{1,q} \lesssim 2^{-2qs} d_{q} ^2\Vert e^{\mathcal{R} t} B_{\phi} \Vert_{\tilde{L}^{2}_{t,\dot{\theta}(t)}(\mathcal{B}^{s+\frac{1}{2}})} \Vert e^{\mathcal{R} t} C_{\phi} \Vert_{\tilde{L}^{2}_{t,\dot{\theta}(t)}(\mathcal{B}^{s+\frac{1}{2}})},
\end{align}
where 
$$ d^2_{q} = d_q(C_\phi) \left( \sum_{|q-q'| \leq 4} d_{q'}(B_\phi) 2^{(q-q')(s-\frac{1}{2})} \right) .$$

\medskip

Similarly, by Lemma \ref{lem:fgphibis1} and considering the support properties to the Fourier transform given in [\cite{B1981}, Proposition 2.10] of the terms in $T^h_{\pa_xB}A$, we obtain 

\begin{align*}
M_{2,q} (t) &\lesssim \int_0^t \abs{\psca{ e^{\mathcal{R} t'} \Delta_q^h ( T^h_{\pa_xB} A )_{\phi}, e^{\mathcal{R} t'} \Delta_q^h C_{\phi}}_{L^2}} dt' \\
&\lesssim \sum_{|q-q'|\leq 4} \int_0^t e^{2\mathcal{R} t'}\Vert S^h_{q'-1} \pa_x B_\phi \Vert_{L^\infty_h(L^2_v)} \Vert \Delta_{q'}^h A_\phi \Vert_{L^2_h(L^\infty_v)} \Vert \Delta_q^h C_\phi \Vert_{L^2} dt'.
\end{align*}

As in \eqref{eq:normB12bis1}, we can write
\begin{equation*} 
	\Vert \Delta_{q'}^h A_{\phi} \Vert_{L^2_h(L^{\infty}_v)} \lesssim \Vert \Delta_{q'}^h \pa_y A_{\phi} \Vert_{L^2} \lesssim 2^{-\frac{q'}{2}} d_{q'}
	(A_\phi) \Vert  \pa_y A_{\phi}(t') \Vert_{\mathcal{B}^{\f12}},
\end{equation*}
and using Bernstein inequality we have
$$\Vert S^h_{q'-1} \pa_x B_\phi(t') \Vert_{L^\infty_h(L^2_v)} \lesssim \sum_{l\leq q'-2} 2^l 2^{\frac{l}{2}} \Vert \Delta_l^h B_\phi \Vert_{L^2}. $$
By using Hölder inequality, we have 
\begin{align*}
	 M_{2,q} &\lesssim  \sum_{|q-q'|\leq 4} \int_0^t 2^{-\frac{q'}{2}} d_{q'}(A_\phi) e^{\mathcal{R} t'}\Vert S^h_{q'-1} \pa_x B_\phi(t') \Vert_{L^\infty_h(L^2_v)} \Vert  \pa_y A_{\phi}(t')\Vert_{\mathcal{B}^{\f12}}e^{\mathcal{R} t'}\Vert \Delta_q^h C_\phi(t') \Vert_{L^2} dt'\\
	&\lesssim  \sum_{|q-q'|\leq 4} \int_0^t 2^{-\frac{q'}{2}} e^{\mathcal{R} t'} \sum_{l\leq q'-2} 2^l 2^{\frac{l}{2}} \Vert \Delta_l^h B_\phi \Vert_{L^2} \Vert  \pa_y A_{\phi}(t) \Vert_{\mathcal{B}^{\f12}}e^{\mathcal{R} t'}\Vert \Delta_q^h C_\phi \Vert_{L^2} dt'\\
	&\lesssim \sum_{|q-q'|\leq 4}  2^{\frac{-q'}{2}} \sum_{l \leq q'-2} \left(\int_0^t e^{2\mathcal{R} t'}2^{3l}  \Vert \Delta_l^h B_{\phi} \Vert_{L^{2}}^2 \Vert \pa_y A_{\phi} \Vert_{\mathcal{B}^{\f12}} dt'\right)^\frac{1}{2} \\ & \quad\quad\quad\quad\quad\quad \quad\quad\quad\quad\quad \times  \left( \int_0^t e^{2\mathcal{R} t'} \Vert \Delta_{q}^h C_{\phi} \Vert_{L^{2}}^2 \Vert \pa_y A_{\phi} \Vert_{\mathcal{B}^{\f12}}  dt'\right)^\frac{1}{2}.
\end{align*}

Yet we observe from Definition \ref{def:CLweightbis12}, and $s<1$ we have 
 \begin{align}\label{nac}
 & \sum_{l \lesssim q'-2} \left(\int_0^t2^{3l} e^{2\mathcal{R} t'} \Vert \Delta_l^h B_{\phi} \Vert_{L^{2}}^2 \Vert \pa_y A_{\phi} \Vert_{\mathcal{B}^{\f12}} dt'\right)^\frac{1}{2} \nonumber \\ & \quad \lesssim  \sum_{l\leq q'-2} 2^{\frac{3l}{2}} \left(\int_0^t \Vert \pa_y A_{\phi} \Vert_{\mathcal{B}^{\f12}} e^{2\mathcal{R} t'} \Vert \Delta_l^h  B_{\phi}(t') \Vert_{L^{2}}^2 dt'\right)^{\frac{1}{2}}\nonumber \\
 &\lesssim  \sum_{l \lesssim q'-2} 2^{\frac{3l}{2}} 2^{-l(s+\f12)} 2^{l(s+\f12)} \left(\int_0^t  \Vert \pa_y A_{\phi} \Vert_{\mathcal{B}^{\f12}} e^{2\mathcal{R} t'} \Vert \Delta_l^h  B_{\phi}(t') \Vert_{L^{2}}^2 dt'\right)^{\frac{1}{2}} \nonumber \\
 &\lesssim 2^{q'(1-s)} \sum_{l\leq q'-2} \tilde{d}_l(B_\phi) 2^{(q'-l)(s-1)}  \Vert e^{\mathcal{R} t} B_{\phi} \Vert_{\tilde{L}^{2}_{t,\dot{\theta}(t)}(\mathcal{B}^{s+\frac{1}{2}})} \nonumber \\
  & \lesssim 2^{q'(1-s)} d_{q'}(B_\phi)  \Vert e^{\mathcal{R} t} B_{\phi} \Vert_{\tilde{L}^{2}_{t,\dot{\theta}(t)}(\mathcal{B}^{s+\frac{1}{2}})}.
 \end{align}
 where $$ d_{q'}(B_\phi) = \sum_{l\leq q'-2} 2^{-(q'-l)(1-s)}\tilde{d}_l(B_\phi) \in \ell^1(\ZZ). $$ 
 We remark that $d_{q'}(B_\phi) \in \ell^1(\ZZ)$ because $s<1$ and  then $d_{q'}(B_\phi)$ is a convolution between summable sequences.
 So that it comes out
 \begin{align}\label{eq:I2qbis1}
  M_{2,q} \lesssim d_{q}^2 2^{-2qs} \Vert e^{\mathcal{R} t} B_{\phi} \Vert_{\tilde{L}^2_{t,\dot{\theta}(t)}(\mathcal{B}^{s+\frac{1}{2}})}  \Vert e^{\mathcal{R} t} C_{\phi} \Vert_{\tilde{L}^2_{t,\dot{\theta}(t)}(\mathcal{B}^{s+\frac{1}{2}})}
 \end{align}
 
where
\begin{align*}
	d_{q}^2 = d_q(C_\phi) \left(\sum_{|q-q'|\leq 4}{\tilde d}_{q'}(C_\phi) 2^{(q -q')(s-\frac{1}{2})}\right),
\end{align*}
with $d_q$ a summable sequence.

To end this proof, it remains to estimate $M_{3,q}$(is the rest term). Using the definition of $R^h(A,\pa_x B)$, the support properties given in [\cite{B1981}, Proposition 2.10], and Bernstein lemma \ref{lem:Bernsteinbis1}, we can write
\begin{align*}
	M_{3,q} &= \int_0^t \abs{\psca{ e^{\mathcal{R} t'} \Delta_q^h (R^h(A,\pa_xB))_{\phi}, e^{\mathcal{R} t'} \Delta_q^h C_{\phi}}_{L^2}} dt'\\ &\lesssim \sum_{q'\geq q-3}\int_0^t e^{2\mathcal{R} t'} \Vert \tilde{\DD}^h_{q'} A_{\phi}\Vert_{L^{2}_h(L^{\infty}_v)} \Vert \Delta_{q'}^h \pa_x B_{\phi}  \Vert_{L^{2}} \Vert \Delta_{q}^h  C_{\phi} \Vert_{L^{\infty}_h(L^2_v)}dt' .
\end{align*}
By using Bernstein Lemma \ref{lem:Bernsteinbis1}, we can find that 
\begin{align}\label{estimate:Cphi}
 \Vert \Delta_{q'}^h \pa_x B_{\phi}  \Vert_{L^{2}}\lesssim 2^{q'} \Vert \Delta_{q'}^h  B_{\phi}  \Vert_{L^{2}}\quad\text{and}\quad\| \Delta_{q}^h  C_{\phi} \|_{L^{\infty}_h(L^2_v)} \lesssim 2^{\frac{q}{2}} \|\Delta_q^h C_\phi \|_{L^2_h(L^2_v)},
\end{align} 
and by using Poincaré inequality on $A$ (having that $A=0$ when $y=0,1$), we obtain
\begin{align}\label{estimate:Aphi1}
 \Vert \tilde{\DD}^h_{q'} A_{\phi}\Vert_{L^{2}_h(L^{\infty}_v)} &\lesssim  \Vert \tilde{\DD}^h_{q'} A_{\phi}\Vert_{L^2_h(L^2_v)}^\f12 \Vert \tilde{\DD}^h_{q'} \pa_y A_{\phi}\Vert_{L^2_h(L^2_v)}^\f12  \nonumber\\
 &\lesssim \Vert \tilde{\DD}^h_{q'} \pa_y A_{\phi}\Vert_{L^{2}}^\f12 \Vert \tilde{\DD}^h_{q'} \pa_y A_{\phi}\Vert_{L^{2}}^\f12 =\Vert \tilde{\DD}^h_{q'} \pa_y A_{\phi}\Vert_{L^{2}}\nonumber \\
 &\lesssim d_{q'}(A_\phi) 2^{-\frac{q'}{2}} \|\pa_y A_\phi\|_{\mathcal{B}^\f12}  .
\end{align}
Then, we replace on \eqref{eq:I2qbis1} we achieve 
\begin{align*}
M_{3,q} &\lesssim 2^{\frac{q}{2}} \sum_{q'\geq q-3}\int_0^t e^{2\mathcal{R} t'}   2^{q'(1-\frac{1}{2})}  \Vert \pa_y A_{\phi}\Vert_{\mathcal{B}^{\f12}} \Vert \Delta_{q'}^h B_{\phi}  \Vert_{L^{2}} \Vert \Delta_{q}^h C_{\phi} \Vert_{L^{2}} dt'  \\
	&\lesssim 2^{\frac{q}{2}} \sum_{q'\geq q-3} \int_0^t e^{2\mathcal{R} t'} 2^{\frac{q'}{2}}  \Vert \pa_y A_{\phi}\Vert_{\mathcal{B}^{\f12}} \Vert \Delta_{q'}^h B_{\phi}  \Vert_{L^{2}} \Vert \Delta_{q}^h C_{\phi} \Vert_{L^{2}}dt'.
\end{align*}
Since $s< 1 $, we have by using Hölder inequality 
\begin{align*}
	 M_{3,q} &\lesssim 2^{\frac{q}{2}} \sum_{q'\geq q-3}\int_0^t  2^{\frac{q'}{2}}  \Vert \pa_y A_{\phi}\Vert_{\mathcal{B}^{\f12}} \Vert \Delta_{q'}^h  e^{\mathcal{R} t'} B_{\phi}  \Vert_{L^{2}} \Vert \Delta_{q}^h e^{\mathcal{R} t'} C_{\phi} \Vert_{L^{2}}dt' \\
	&\lesssim  2^{\frac{q}{2}} \sum_{q'\geq q-3} 2^{\frac{q'}{2}} \left(\int_0^t  \Vert e^{\mathcal{R} t'} \Delta_{q'}^h B_{\phi} \Vert_{L^{2}}^2 \Vert \pa_y A_{\phi} \Vert_{\mathcal{B}^{\f12}} dt'\right)^\frac{1}{2} \\ & \quad\quad\quad\quad\quad\quad \quad\quad\quad\quad\quad \times  \left( \int_0^t \Vert e^{\mathcal{R} t'} \Delta_{q}^h C_{\phi} \Vert_{L^{2}}^2 \Vert \pa_y A_{\phi} \Vert_{\mathcal{B}^{\f12}}  dt'\right)^\frac{1}{2}  \\
	&\lesssim d_q(C_\phi) 2^{-2qs} \Vert e^{\mathcal{R} t} B_\phi \Vert_{\tilde{L}^2_{t,\dot{\theta}(t)}(\mathcal{B}^{s+\frac{1}{2}})}  \Vert e^{\mathcal{R} t} C_\phi \Vert_{\tilde{L}^2_{t,\dot{\theta}(t)}(\mathcal{B}^{s+\frac{1}{2}})} \left(  \sum_{q'\geq q-3} d_{q'}(B_\phi) 2^{(q-q')s} \right)dt'.
\end{align*}
 Then,
\begin{align} \label{eq:I3qbis1}
  M_{3,q} \lesssim d_{q}^2 2^{-2qs} \Vert e^{\mathcal{R} t} B_\phi \Vert_{\tilde{L}^2_{t,\dot{\theta}(t)}(\mathcal{B}^{s+\frac{1}{2}})}  \Vert e^{\mathcal{R} t} C_\phi \Vert_{\tilde{L}^2_{t,\dot{\theta}(t)}(\mathcal{B}^{s+\frac{1}{2}})}.
\end{align}
where
\begin{align*}
d_{q}^2 = d_q(C_\phi) \left(  \sum_{q'\geq k-3} d_{q'}(B_\phi) 2^{(q-q')s} \right)
\end{align*}
with $d_q$ a summable  sequence of positive numbers. Lemma \ref{lem:ABCbis1} is then proved by summing Estimates \eqref{eq:I1qbis1}, \eqref{eq:I2qbis1} and \eqref{eq:I3qbis1}. \hfill

\end{proof}

\begin{lemma} \label{lem:Apa_yBCbis1}
Let A,B and C be smooth functions on $[0,T^*) \times \mathbb{R}\times (0,1)$, with $A$ vanishing on the boundary, $s\in ]0,1[$, and $\phi$ be defined as in \eqref{eq:Anphichp3}, with $\dot{\theta} (t) = \Vert \pa_y A_{\phi}(t) \Vert_{\mathcal{B}^{\f12}} $. For any $t \in[0,T^*)$ and for any $B,C \in \tilde{L}^2_{t,\dot{\theta}(t)}(\mathcal{B}^{s+\frac{1}{2}})$, we have
	\begin{align}
	 \int_0^t &\abs{\psca{e^{\mathcal{R} t'} \Delta_q^h \Big(\int_0^y \pa_xBdy'\cdot\pa_y A\Big)_{\phi}, e^{\mathcal{R} t'} \Delta_q^h C_{\phi}}_{L^2}} dt' \nonumber \\ &\lesssim {d}^2_q 2^{-2qs} \Vert e^{\mathcal{R} t} B_{\phi} \Vert_{\tilde{L}^2_{t,\dot{\theta}(t)}(\mathcal{B}^{s+\frac{1}{2}})}\Vert e^{\mathcal{R} t} C_{\phi} \Vert_{\tilde{L}^2_{t,\dot{\theta}(t)}(\mathcal{B}^{s+\frac{1}{2}})}.
	\end{align}
\end{lemma}
\begin{proof}
As in \cite{BCD2011book}, using homogeneous Bony's decomposition of $(\int_0^y \pa_xBdy')\cdot\pa_y A$ into para-products  in the horizontal variable and remainders, we can write

$$ \pa_y A \cdot\int_0^y \pa_xBdy' = T^{h}_{\pa_yA} \int_0^y \pa_{x} Bdy' + T^h_{\int_0^y\pa_x Bdy'} \pa_yA + R^{h}(\pa_yA,\int_0^y \pa_x Bdy') $$
where
\begin{align*}
	T_{\pa_yA}^h \int_0^y\pa_xBdy' &= \sum_{q\in\ZZ} S_{q-1}^h \pa_yA\cdot \DD_q^h \int_0^y\pa_x Bdy'  \\ R^h(\pa_yA,\int_0^y\pa_x Bdy') &= \sum_{\abs{q'-q} \leq 1} \DD_q^h \pa_yA \cdot\DD_{q'}^h \int_0^y\pa_x Bdy'\\
	&= \sum_{q} \tilde{\DD}_q^h \pa_yA \cdot\DD_{q}^h \int_0^y\pa_x Bdy'.
\end{align*}

We replace this decomposition and we obtain the following bound of 
\begin{align*}
	\int_0^t \abs{\psca{e^{\mathcal{R} t'} \Delta_q^h (\pa_yA\cdot\int_0^y\pa_x Bdy')_{\phi}, e^{\mathcal{R} t'} \Delta_q^h C_{\phi}}_{L^2}} dt' \leq N_{1,q} + N_{2,q} + N_{3,q}, 
\end{align*}
where
\begin{align*}
	N_{1,q} &= \int_0^t \abs{\psca{ e^{\mathcal{R} t'} \Delta_q^h ( T^{h}_{\pa_yA} \int_0^y \pa_{x} Bdy')_{\phi}, e^{\mathcal{R} t'} \Delta_q^h C_{\phi}}_{L^2}} dt'\\
	N_{2,q} &= \int_0^t \abs{\psca{ e^{\mathcal{R} t'} \Delta_q^h ( T^h_{\int_0^y\pa_x Bdy'} \pa_yA )_{\phi}, e^{\mathcal{R} t'} \Delta_q^h C_{\phi}}_{L^2}} dt'\\
	N_{3,q} &= \int_0^t \abs{\psca{ e^{\mathcal{R} t'} \Delta_q^h (R^{h}(\pa_yA,\int_0^y \pa_x Bdy'))_{\phi}, e^{\mathcal{R} t'} \Delta_q^h C_{\phi}}_{L^2}} dt'.
\end{align*}

We start by getting the estimate of the first term $N_{1,q}$, for that we need to use the definition of $T^{h}_{\pa_yA} \int_0^y\pa_x Bdy'$, the support properties given in  [\cite{B1981}, Proposition 2.10] and  and again thanks to the Corollary \ref{coro} we infer 

\begin{equation} \label{eq:normpa_yABCbis1}
	N_{1,q} \leq \sum_{|q-q'| \leq 4} \int_0^t e^{2\mathcal{R} t'} \Vert S^h_{q'-1} \pa_y A_{\phi} (t') \Vert_{L^\infty_h(L^{2}_v)}\Vert \Delta_{q'}^h \int_0^y\partial_xB_{\phi} (t')dy' \Vert_{L^2_h(L^\infty_v)} \Vert \Delta_q^h C_{\phi} (t') \Vert_{L^2} dt'.
\end{equation}

We have by applying Bernstein Lemma \ref{lem:Bernsteinbis1} that
\begin{align} \label{eq:normB12bisbis1}
	\Vert \Delta^h_q \pa_y A_{\phi}(t') \Vert_{L^\infty_h(L^2_v)} &\lesssim 2^{\frac{q}{2}} \Vert \Delta_q^h \pa_y A_{\phi}(t') \Vert_{L^2} \nonumber \\ & \lesssim d_q(A_\phi) \Vert  \pa_y A_{\phi}(t') \Vert_{\mathcal{B}^{\f12}},
\end{align}
where $(d_q(t))_{q\in\ZZ}$ is a generic element of $\ell^{1}(\ZZ) $ such that  $\sum d_q(t) \leq 1$. Then 

\begin{align*}
	\Vert S^h_{q'-1} \pa_y A_{\phi} (t') \Vert_{L^\infty_h(L^2_v)} \lesssim  \Vert  \pa_y A_{\phi}(t') \Vert_{\mathcal{B}^{\f12}},
\end{align*}
and 
\begin{align*}
    \Vert \Delta_{q'}^h \int_0^y\partial_xB_{\phi} (t')dy' \Vert_{L^2_h(L^\infty_v)} \lesssim 2^{q'} \|\Delta_{q'}^h \int_0^y B_\phi dy'\|_{L^2_h(L^\infty_v)}\lesssim 2^{q'} \|\Delta_{q'}^h B_\phi \|_{L^2},
\end{align*}
where we have used that $\|\int_0^y fdy'\|_{L^\infty_y}\leq \int_0^1|f|dy'\leq \|f\|_{L^2_y}$. As a result, it comes out  

\begin{align*}
	N_{1,q} &\leq \sum_{|q-q'| \leq 4} 2^{q'}\int_0^t e^{2\mathcal{R} t'} \Vert  \pa_y A_{\phi}(t') \Vert_{\mathcal{B}^{\f12}}\|\Delta_{q'}^h B_\phi \|_{L^2} \Vert \Delta_q^h C_{\phi} (t') \Vert_{L^2} dt' \\
	&\lesssim \sum_{|q-q'| \leq 4} 2^{q'}\int_0^t e^{2\mathcal{R} t'} \dot{\theta}(t') \|\Delta_{q'}^h B_\phi \|_{L^2} \Vert \Delta_q^h C_{\phi} (t') \Vert_{L^2} dt' \\
	&\lesssim \sum_{|q-q'| \leq 4} 2^{q'} \left(\int_0^t e^{2\mathcal{R} t'} \dot{\theta}(t') \|\Delta_{q'}^h B_\phi \|_{L^2}^2 dt' \right)^\frac{1}{2} \times \left( \int_0^t e^{2\mathcal{R} t'} \dot{\theta}(t') \|\Delta_{q}^h C_\phi \|_{L^2}^2 dt'\right)^\frac{1}{2}
\end{align*}
where in the last step we used the Hölder inequality, and we note that $\dot{\theta}(t) = \Vert  \pa_y A_{\phi}(t') \Vert_{\mathcal{B}^{\f12}}$, using definition \ref{def:CLweightbis12}, we achieve
\begin{align*}
\left(\int_0^t e^{2\mathcal{R} t'} \dot{\theta}(t') \|\Delta_{q'}^h B_\phi \|_{L^2}^2 dt' \right)^\frac{1}{2} \lesssim d_{q'}(B_\phi) 2^{-q'(s+\f12)} \|e^{\mathcal{R}t} B_\phi \|_{\tilde{L}^2_{t,\dot{\theta}(t)}(\mathcal{B}^{s+\f12})}.
\end{align*}
and
\begin{align*}
\left(\int_0^t e^{2\mathcal{R} t'} \dot{\theta}(t') \|\Delta_{q'}^h C_\phi \|_{L^2}^2 dt' \right)^\frac{1}{2} \lesssim d_{q'}(C_\phi) 2^{-q'(s+\f12)} \|e^{\mathcal{R}t} C_\phi \|_{\tilde{L}^2_{t,\dot{\theta}(t)}(\mathcal{B}^{s+\f12})}.
\end{align*}
Then,
\begin{align} \label{eq:N1qbis1}
	N_{1,q} \lesssim 2^{-2qs} {d}^2_{q} \Vert e^{\mathcal{R} t} B_{\phi} \Vert_{\tilde{L}^{2}_{t,\dot{\theta}(t)}(\mathcal{B}^{s+\frac{1}{2}})} \Vert e^{\mathcal{R} t} C_{\phi} \Vert_{\tilde{L}^{2}_{t,\dot{\theta}(t)}(\mathcal{B}^{s+\frac{1}{2}})},
\end{align}
where 
$$ {d}_{q}^2 = d_q(C_\phi) \left( \sum_{|q-q'| \leq 4} d_{q'}(B_\phi) 2^{(q-q')(s-\frac{1}{2})} \right) $$

\medskip

Along the same way for $s<1$, we obtain 

\begin{align*}
N_{2,q} (t) &\leq \int_0^t \abs{\psca{ e^{\mathcal{R} t'} \Delta_q^h ( T^h_{\int_0^t\pa_xBdy'} \pa_y A )_{\phi}, e^{\mathcal{R} t'} \Delta_q^h C_{\phi}}_{L^2}} dt' \\
&\lesssim \sum_{|q-q'|\leq 4} \int_0^t e^{2\mathcal{R} t'}\Vert S^h_{q'-1} \int_0^y\pa_x B_\phi dy' \Vert_{L^\infty} \Vert \Delta_{q'}^h \pa_yA_\phi \Vert_{L^2} \Vert \Delta_q^h C_\phi \Vert_{L^2} dt' \\
&\lesssim \sum_{|q-q'|\leq 4} 2^{-\frac{q'}{2}} \sum_{l \lesssim q'-2} \left(\int_0^t e^{2\mathcal{R} t'} \dot{\theta}(t') \Vert \Delta_l^h \int_0^y \pa_x B_\phi \Vert_{L^\infty}^2  dt'\right)^\frac{1}{2} \times  \left( \int_0^t e^{2\mathcal{R} t'} \dot{\theta}(t') \Vert \Delta_{q}^h C_{\phi} \Vert_{L^{2}}^2 dt'\right)^\frac{1}{2}.
\end{align*}
By using the fact that $\|\int_0^y fdy'\|_{L^\infty_y}\leq \int_0^1|f|dy'\leq \|f\|_{L^2_y}$, we have
\begin{align*}
    \Vert \Delta_{l}^h \int_0^y\partial_xB_{\phi} (t')dy' \Vert_{L^\infty_h(L^\infty_v)} \lesssim 2^{l} \|\Delta_{l}^h \int_0^y B_\phi dy'\|_{L^\infty_h(L^\infty_v)}\lesssim 2^{l} \|\Delta_{l}^h B_\phi \|_{L^\infty_h(L^2_v)},
\end{align*}
and we combine with Bernstein's lemma \ref{lem:Bernsteinbis1}, we can find
\begin{align*}
    \Vert \Delta_{l}^h \int_0^y\partial_xB_{\phi} (t')dy' \Vert_{L^\infty_h(L^\infty_v)}\lesssim 2^{l} \|\Delta_{l}^h B_\phi \|_{L^\infty_h(L^2_v)} \lesssim 2^{\frac{3l}{2}} \|\Delta_{l}^h B_\phi \|_{L^2}.
\end{align*}
We observe from Definition \ref{def:CLweightbis12}, and we use the same thing as in \eqref{nac}, then for $s<1$ we have 
 \begin{align}\label{eq;BBBbis1}
 &  \sum_{l \lesssim q'-2} \left(\int_0^t e^{2\mathcal{R} t'} \dot{\theta}(t') \Vert \Delta_l^h \int_0^y \pa_x B_\phi \Vert_{L^\infty}^2  dt'\right)^\frac{1}{2} \\
 &\lesssim  \sum_{l \lesssim q'-2} \left(\int_0^t e^{2\mathcal{R} t'} \dot{\theta}(t')2^{3l} \Vert \Delta_l^h B_\phi \Vert_{L^2}^2  dt'\right)^\frac{1}{2} \nonumber \\
 &\lesssim \sum_{l \lesssim q'-2} 2^{\frac{3l}{2}} 2^{-l(s+\f12)} 2^{l(s+\f12)} \left(\int_0^t e^{2\mathcal{R} t'} \dot{\theta}(t') \Vert \Delta_l^h B_\phi \Vert_{L^2}^2  dt'\right)^\frac{1}{2}\nonumber \\
 &\lesssim 2^{q'(1-s)} \sum_{l \lesssim q'-2} \tilde{d}_l(B_\phi) 2^{(q'-l)(s-1)}  \Vert B_\phi \Vert_{\tilde{L}^2_{t,\dot{\theta}(t)}(\mathcal{B}^{s+\f12})} \nonumber \\
 &  \lesssim  2^{q'(1-s)}d_{q'}(B_\phi) \Vert e^{\mathcal{R} t} B_{\phi} \Vert_{\tilde{L}^{2}_{t,\dot{\theta}(t)}(\mathcal{B}^{s+\frac{1}{2}})} \nonumber.
 \end{align}
 where $$ d_{q'}(B_\phi) = \sum_{l\leq q'-2} 2^{-(q'-l)(1-s)}\tilde{d}_l(B_\phi) \in \ell^1(\ZZ). $$
 So that it comes out
 \begin{align}\label{eq:N2qbis1}
  N_{2,q} \leq {d}_{q}^2 2^{-2qs} \Vert e^{\mathcal{R} t} B_{\phi} \Vert_{\tilde{L}^2_{t,\dot{\theta}(t)}(\mathcal{B}^{s+\frac{1}{2}})} \Vert e^{\mathcal{R} t} C_{\phi} \Vert_{\tilde{L}^2_{t,\dot{\theta}(t)}(\mathcal{B}^{s+\frac{1}{2}})}
 \end{align}
 
where
\begin{align*}
	{d}_{q}^2 = d_q(C_\phi) \left(\sum_{|q-q'|\leq 4}d_{q'}(B_\phi) 2^{(q -q')(s-\frac{1}{2})}\right)
\end{align*}
is a suitable sequence of positive constants.

To end this proof, it remains to estimate $N_{3,q}$(is the rest term). Using the support properties given in [\cite{B1981}, Proposition 2.10], the definition of $R^h(\pa_yA,\int_0^y\pa_x Bdy')$ and Bernstein Lemma \ref{lem:Bernsteinbis1}, we can write
\begin{align*}
	N_{3,q} &\leq \sum_{q'\geq q-3}\int_0^t e^{2\mathcal{R} t'} \Vert \tilde{\DD}^h_{q'} \pa_yA_{\phi}\Vert_{L^{2}} \Vert \Delta_{q'}^h \int_0^y\pa_x B_{\phi}dy'  \Vert_{L^{2}_h(L^{\infty}_v)} \Vert \Delta_{q}^h  C_{\phi} \Vert_{L^{\infty}_h(L^2_v)}dt' \\
	&\leq 2^{\frac{q}{2}} \sum_{q'\geq q-3}\int_0^t e^{2\mathcal{R} t'}   2^{q'(1-\frac{1}{2})}  \Vert \pa_y A_{\phi}\Vert_{\mathcal{B}^{\f12}} \Vert \Delta_{q'}^h B_{\phi}  \Vert_{L^{2}} \Vert \Delta_{q}^h C_{\phi} \Vert_{L^{2}} dt'  \\
	&\leq 2^{\frac{q}{2}} \sum_{q'\geq q-3} \int_0^t e^{2\mathcal{R} t'} 2^{\frac{q'}{2}}  \Vert \pa_y A_{\phi}\Vert_{\mathcal{B}^{\f12}} \Vert \Delta_{q'}^h B_{\phi}  \Vert_{L^{2}} \Vert \Delta_{q}^h C_{\phi} \Vert_{L^{2}}dt'.
\end{align*}
Since $s< 1 $, we have by applying Hölder inequality 
\begin{align*}
	 N_{3,q} &\leq 2^{\frac{q}{2}} \sum_{q'\geq q-3}\int_0^t  2^{\frac{q'}{2}}  \Vert \pa_y A_{\phi}\Vert_{\mathcal{B}^{\f12}} \Vert \Delta_{q'}^h  e^{\mathcal{R} t'} B_{\phi}  \Vert_{L^{2}} \Vert \Delta_{q}^h e^{\mathcal{R} t'} C_{\phi} \Vert_{L^{2}}dt' \\
	&\leq  2^{\frac{q}{2}} \sum_{q'\geq q-3} 2^{\frac{q'}{2}} \left(\int_0^t \Vert \pa_y A_{\phi} \Vert_{\mathcal{B}^{\f12}} \Vert e^{\mathcal{R} t'} \Delta_{q'}^h B_{\phi} \Vert_{L^{2}}^2  dt'\right)^\frac{1}{2} \\ & \quad\quad\quad\quad\quad\quad \quad\quad\quad\quad\quad \times  \left( \int_0^t \Vert \pa_y A_{\phi} \Vert_{\mathcal{B}^{\f12}} \Vert e^{\mathcal{R} t'} \Delta_{q}^h C_{\phi} \Vert_{L^{2}}^2 dt'\right)^\frac{1}{2}.
\end{align*}
Definition \ref{def:CLweightbis12} tells us
\begin{align*}
 \left( \int_0^t \Vert \pa_y A_{\phi} \Vert_{\mathcal{B}^{\f12}} \Vert e^{\mathcal{R} t'} \Delta_{q}^h C_{\phi} \Vert_{L^{2}}^2 dt'\right)^\frac{1}{2} \lesssim d_q(C_\phi) 2^{-q(s+\f12)} \Vert e^{\mathcal{R} t} C_\phi \Vert_{\tilde{L}^2_{t,\dot{\theta}(t)}(\mathcal{B}^{s+\frac{1}{2}})}.
\end{align*}
Then we obtain
\begin{align} \label{eq:N3qbis1}
	 N_{3,q} \lesssim {d}_{q}^2 2^{-2qs} \Vert e^{\mathcal{R} t} B_\phi \Vert_{\tilde{L}^2_{t,\dot{\theta}(t)}(\mathcal{B}^{s+\frac{1}{2}})}  \Vert e^{\mathcal{R} t} C_\phi \Vert_{\tilde{L}^2_{t,\dot{\theta}(t)}(\mathcal{B}^{s+\frac{1}{2}})}.
\end{align}
where
\begin{align*}
{d}_{q} ^2= d_q(C_\phi) \left(  \sum_{q'\geq q-3} d_{q'}(B_\phi) 2^{(q-q')s} \right)
\end{align*}
is a suitable sequence of positive constants.

Lemma \ref{lem:Apa_yBCbis1} is then proved by summing Estimates \eqref{eq:N1qbis1}, \eqref{eq:N2qbis1} and \eqref{eq:N3qbis1}. \hfill
\end{proof}

\begin{lemma} \label{lem:BBBbis1}
	For any $s\in ]0,1[$ and $t\leq T^{\ast}$, and $\phi$ be defined as in \eqref{eq:Anphichp3}, with $$\dot{\theta}(t) =  \Vert \pa_y A_{\phi}(t) \Vert_{\mathcal{B}^{\f12}} + \eps \Vert \pa_y B_{\phi}(t) \Vert_{\mathcal{B}^{\f12}}.$$ Then, for any $t \in[ 0, T^*)$ and for any $A \in \tilde{L}^2_{t,\dot{\theta}(t)}(\mathcal{B}^{s+\frac{1}{2}})$ that satisfies  $\pa_x A = -\pa_y B$, $A$ and $B$ vanishing on the boundary (so that $B(t,x,y) = - \int_0^y \pa_x A(t,x,s) ds$), we have
	\begin{multline*}
		\eps^2 \int_0^t \abs{ \psca{e^{\mathcal{R}t'}\Delta_q^h (B\pa_y B)_{\phi}, e^{\mathcal{R}t'}\Delta_q^h C_{\phi}}_{L^2}}dt' \\ \lesssim d_q^2 2^{-2qs} \Big( \| e^{\mathcal{R}t} A_{\phi} \|_{\tilde{L}^2_{t,\dot{\theta(t)}}(\mathcal{B}^{s+\frac{1}{2}})} +\| e^{\mathcal{R}t} \eps B_{\phi} \|_{\tilde{L}^2_{t,\dot{\theta(t)}}(\mathcal{B}^{s+\frac{1}{2}})}\Big) \| e^{\mathcal{R}t} \eps C_{\phi} \|_{\tilde{L}^2_{t,\dot{\theta(t)}}(\mathcal{B}^{s+\frac{1}{2}})}.
	\end{multline*}
\end{lemma}

\begin{proof}
As in \cite{BCD2011book}, using Bony's homogeneous decomposition of the term $B\pa_y B$ into para-products  in the horizontal variable and remainders, we can write

$$ B\pa_yB = T^{h}_{\pa_yB}B + T^h_{B} \pa_yB + R^{h}(\pa_yB, B) $$

We replace, we obtain the following bound of $\int_0^t \abs{\psca{e^{\mathcal{R} t'} \Delta_q^h (B\pa_y B )_{\phi}, e^{\mathcal{R} t'} \Delta_q^h C_{\phi}}_{L^2}} dt'$

\begin{align*}
	\int_0^t \abs{\psca{e^{\mathcal{R} t'} \Delta_q^h (B\pa_y B)_{\phi}, e^{\mathcal{R} t'} \Delta_q^h C_{\phi}}_{L^2}} dt' \leq L_{1,q} + L_{2,q} + L_{3,q}, 
\end{align*}
where
\begin{align*}
	L_{1,q} &= \int_0^t \abs{\psca{ e^{\mathcal{R} t'} \Delta_q^h ( T^{h}_{B} \pa_y B)_{\phi}, e^{\mathcal{R} t'} \Delta_q^h C_{\phi}}_{L^2}} dt'\\
	L_{2,q} &= \int_0^t \abs{\psca{ e^{\mathcal{R} t'} \Delta_q^h ( T^h_{\pa_y B} B )_{\phi}, e^{\mathcal{R} t'} \Delta_q^h C_{\phi}}_{L^2}} dt'\\
	L_{3,q} &= \int_0^t \abs{\psca{ e^{\mathcal{R} t'} \Delta_q^h (R^{h}(B,\pa_yB))_{\phi}, e^{\mathcal{R} t'} \Delta_q^h C_{\phi}}_{L^2}} dt'.
\end{align*}

We start by getting the estimate of the first term $L_{1,q}$. Due to $\pa_y B = -\pa_x A$, one has

\begin{align}\label{esti:L&q}
\eps^2 L_{1,q} &\lesssim \eps^2 \sum_{|q'-q|\leq 4 }\int_0^t e^{2\mathcal{R} t'} \Vert S^h_{q'-1} B_{\phi}(t) \Vert_{L^{\infty}} \Vert \Delta_{q'}^h \pa_x A_{\phi} (t)\Vert_{L^2} \Vert \Delta_q^h C_{\phi}(t) \Vert_{L^2}.
\end{align}
By using Bernstein Lemma \ref{lem:Bernsteinbis1}, we have 
\begin{align*}
\Vert S^h_{q'-1} B_{\phi}(t') \Vert_{L^{\infty}_h} \lesssim \sum_{l\leq q'-2} 2^{\frac{l}{2}}\|\Delta_l^h B_\phi(t')\|_{L^2_h}.
\end{align*}
Taking the $L^\infty_v$ norm and using that $ \|f\|_{L^p_vL^q_h} \leq \|f\|_{L^q_h L^p_v}$, for $p\geq q \geq 1$ we obtain
\begin{align*}
\Vert S^h_{q'-1} B_{\phi}(t') \Vert_{L^{\infty}} \lesssim \sum_{l\leq q'-2} 2^{\frac{l}{2}}\|\Delta_l^h B_\phi(t')\|_{L^\infty_v(L^2_h)} \lesssim \sum_{l\leq q'-2} 2^{\frac{l}{2}}\|\Delta_l^h B_\phi(t')\|_{L^2_h(L^\infty_v)}.
\end{align*} 
While using the inclusion $H^1_y \hookrightarrow L^\infty_v$,
\begin{align*}
\|\Delta_l^h B_\phi(t')\|_{L^\infty_v} \lesssim \|\Delta_l^h B_\phi(t') \|_{L^2_v}^\f12 \|\Delta_l^h \pa_y B_\phi(t') \|_{L^2_v}^\f12,
\end{align*}
and Poincaré inequality on the interval $\lbrace 0 < y< 1 \rbrace$ on $B$ ( we use the fact that $B=0$ when $y=0,1$) gives
\begin{align}\label{estim:222}
\Vert S^h_{q'-1} B_{\phi}(t') \Vert_{L^{\infty}} \lesssim \sum_{l\leq q'-2} 2^{\frac{l}{2}} \|\Delta_l^h \pa_y B_\phi(t') \|_{L^2} \lesssim \|\pa_y B_\phi(t') \|_{\mathcal{B}^\f12}
\end{align} 
Then, we replace in \eqref{esti:L&q}, we get by using Hölder inequality 
\begin{align*}
\eps^2 L_{1,q} &\lesssim \eps^2 \sum_{|q'-q|\leq 4 } \int_0^t e^{2\mathcal{R} t'}  \|\pa_y B_\phi (t')\|_{\mathcal{B}^\f12} 2^{q'} \Vert \Delta_{q'}^h e^{\mathcal{R}t'}A_{\phi} (t)\Vert_{L^2} \Vert \Delta_q^h C_{\phi}(t) \Vert_{L^2} \\
	&\lesssim \eps \sum_{|q'-q|\leq 4 } 2^{q'}\Big(\int_0^t   \eps \|\pa_y B_\phi (t')\|_{\mathcal{B}^\f12} \Vert \Delta_{q'}^h A_{\phi} (t')\Vert_{L^2}^2 dt' \Big)^\frac{1}{2} \\ & \quad\quad\quad\quad\quad\quad \times \Big(\int_0^t \eps\Vert \pa_y B_{\phi} (t')\Vert_{\mathcal{B}^{\f12}} \Vert \Delta_q^h e^{\mathcal{R}t'} C_{\phi}(t') \Vert_{L^2}^2  dt' \Big)^\frac{1}{2}
\end{align*}
Hence we deduce from the Definition \ref{def:wCLspacesbis1} that 
\begin{align*}
\Big(\int_0^t   \eps \|\pa_y B_\phi (t')\|_{\mathcal{B}^\f12} \Vert \Delta_{q'}^h e^{\mathcal{R}t'} A_{\phi} (t')\Vert_{L^2}^2 dt' \Big)^\frac{1}{2} \lesssim d_{q'}(A_\phi) 2^{-q'(s+\f12)} \|e^{\mathcal{R}t} A_\phi \|_{\tilde{L}^2_{t,\dot{\theta}(t)}(\mathcal{B}^{s+\f12})}.
\end{align*}
Then,
\begin{align}\label{eq:L1qbis1}
 L_{1,q} \lesssim d_{q}^2 2^{-2qs} \| e^{\mathcal{R}t} A_{\phi} \|_{\tilde{L}^2_{t,\dot{\theta(t)}}(\mathcal{B}^{s+\frac{1}{2}})}\| e^{\mathcal{R}t} \eps C_{\phi} \|_{\tilde{L}^2_{t,\dot{\theta(t)}}(\mathcal{B}^{s+\frac{1}{2}})}
\end{align}
where 
$$d_q^2 =d_q(C_\phi) \sum_{|q'-q|\leq 4} d_{q'}(A_\phi)2^{(q-q')(s-\f12)} $$
\medskip

Now we move to get the estimate of the second term $L_{2,q}$. Due to $\pa_y B_\phi = -\pa_x A_\phi$, we can achieve 
\begin{align*}
L_{2,q} (t) &\leq \int_0^t \abs{\psca{ e^{\mathcal{R} t'} \Delta_q^h ( T^h_{\pa_yB} B )_{\phi}, e^{\mathcal{R} t'} \Delta_q^h C_{\phi}}_{L^2}} dt' \\
&\lesssim \sum_{|q-q'|\leq 4} \int_0^t e^{2\mathcal{R} t'}\Vert S^h_{q'-1} \pa_x A_\phi \Vert_{L^\infty} \Vert \Delta_{q'}^h B_\phi \Vert_{L^2} \Vert \Delta_q^h C_\phi \Vert_{L^2} dt'.
\end{align*}
Along the same way as we did to estimate $L_{1,q}$, we can obtain by using \eqref{estim:222} that
$$ \Vert S^h_{q'-1} \pa_x A_\phi(t') \Vert_{L^\infty} \lesssim 2^{q'} \|\pa_y A_\phi(t') \|_{\mathcal{B}^\f12}. $$ 
Then,
\begin{align*}
L_{2,q} (t) &\lesssim \sum_{|q-q'|\leq 4} 2^{q'} \int_0^t e^{2\mathcal{R} t'}\Vert \pa_y A_\phi \Vert_{\mathcal{B}^\f12} \Vert \Delta_{q'}^h B_\phi \Vert_{L^2} \Vert \Delta_q^h C_\phi \Vert_{L^2} dt' \\
&\lesssim \sum_{|q-q'|\leq 4} 2^{q'} \Big( \int_0^t \Vert \pa_y A_\phi \Vert_{\mathcal{B}^\f12} \Vert e^{\mathcal{R}t'} \Delta_{q'}^h B_\phi \Vert_{L^2}^2 dt'\Big)^\frac{1}{2} \\
& \quad\quad\quad\quad \quad\quad\quad\quad \times \Big( \int_0^t \Vert \pa_y A_\phi \Vert_{\mathcal{B}^\f12} \Vert e^{\mathcal{R}t'} \Delta_{q'}^h C_\phi \Vert_{L^2}^2 dt'\Big)^\frac{1}{2},
\end{align*}

in the last step we use Hölder inequality. Then thanks to Definition \ref{def:wCLspacesbis1}, we arrive at
\begin{align*}
 \Big( \int_0^t \Vert \pa_y A_\phi \Vert_{\mathcal{B}^\f12} \Vert e^{\mathcal{R}t'} \Delta_{q'}^h B_\phi \Vert_{L^2}^2 dt'\Big)^\frac{1}{2} \lesssim d_{q'}(B_\phi) 2^{-q'(s+\f12)} \| e^{\mathcal{R}t'} B_\phi \|_{\tilde{L}^2_{t,\dot{\theta}(t)}(\mathcal{B}^{s+\f12})}.
\end{align*}
We conclude our estimate by 
\begin{align}\label{eq:L2qbis1}
 \eps^2 L_{2,q} \lesssim d_q^2 2^{-2qs} \| e^{\mathcal{R}t} \eps B_{\phi} \|_{\tilde{L}^2_{t,\dot{\theta(t)}}(\mathcal{B}^{s+\frac{1}{2}})}\| e^{\mathcal{R}t} \eps C_{\phi} \|_{\tilde{L}^2_{t,\dot{\theta(t)}}(\mathcal{B}^{s+\frac{1}{2}})}
\end{align}

where 
$$ d_q^2 = d_{q}(C_\phi) \left( \sum_{|q-q'|\leq 4} d_{q'}(B_\phi) 2^{(q-q')(s-\f12)}\right) $$

To end this proof, it remains to estimate $L_{3,q}$ (is the rest term). Due to $\pa_y B = -\pa_x A$, we get, by applying lemma \ref{lem:Bernsteinbis1} that

\begin{align*}
L_{3,q} &\lesssim \sum_{q'\geq q-3} \int_0^t e^{2\mathcal{R} t'}\Vert \tilde{\Delta}_{q'}^h \pa_x A_\phi \Vert_{L^2_h(L^\infty_v)} \Vert \Delta_{q'}^h B_\phi \Vert_{L^2} \Vert \Delta_q^h C_\phi \Vert_{L^\infty_h(L^2_v)} dt' \\
&\lesssim  \sum_{q'\geq q-3} 2^\frac{q'}{2} \int_0^t e^{2\mathcal{R} t'}\Vert \pa_y A_\phi \Vert_{\mathcal{B}^\f12} \Vert \Delta_{q'}^h B_\phi \Vert_{L^2} 2^\frac{q}{2} \Vert \Delta_q^h C_\phi \Vert_{L^2} dt' \\
&\lesssim 2^\frac{q}{2} \sum_{q'\geq q-3} 2^\frac{q'}{2} \Big( \int_0^t \Vert \pa_y A_\phi \Vert_{\mathcal{B}^\f12} \Vert e^{\mathcal{R}t'} \Delta_{q'}^h B_\phi \Vert_{L^2}^2 dt'\Big)^\frac{1}{2} \\
& \quad\quad\quad\quad \quad\quad\quad\quad \times \Big( \int_0^t \Vert \pa_y A_\phi \Vert_{\mathcal{B}^\f12} \Vert e^{\mathcal{R}t'} \Delta_{q'}^h C_\phi \Vert_{L^2}^2 dt'\Big)^\frac{1}{2}.
\end{align*}

which together with Definition \ref{def:CLweightbis12} and $s <1$ ensures that

\begin{align}\label{eq:L3qbis1}
\eps^2 L_{3,q} \lesssim d_q 2^{-2qs} \| e^{\mathcal{R}t} \eps B_{\phi} \|_{\tilde{L}^2_{t,\dot{\theta(t)}}(\mathcal{B}^{s+\frac{1}{2}})} \| e^{\mathcal{R}t} \eps C_{\phi} \|_{\tilde{L}^2_{t,\dot{\theta(t)}}(\mathcal{B}^{s+\frac{1}{2}})}
\end{align}
where 
$$ d_q^2 = d_q(C_\phi) \left( \sum_{q'\geq q-3} d_{q'}(B_\phi) 2^{(q-q')s} \right) $$
Lemma \ref{lem:BBBbis1} is then proved by summing Estimates \eqref{eq:L1qbis1}, \eqref{eq:L2qbis1} and \eqref{eq:L3qbis1}. \hfill
\end{proof}

\section{Global existence of the perturbed hydrostatic system \eqref{eq:HPNhydrolimit} }\label{sec:3}

The goal of this section is to prove the global well-posedness of the limit system of the Perturbed Navier-Stokes equation, we remark that the local smooth solution of the limit system follows a standard parabolic regularization method similar to the Perturbation NS system, First, we remark that the Dirichlet boundary condition $$ (u,v)_{/y=0} = (u,v)_{/y=1} =0,$$ and the incompressible condition $\pa_x u +\pa_y v =0$ imply that :
\begin{align}\label{vvbis1}
    v(t,x,y) &= \int_0^y \pa_y v(t,x,s) ds = - \int_0^y \pa_x u(t,x,s)ds
\end{align}
Due to the compatibility condition $\pa_x \int_0^1 u_0 dy = 0 $, we deduce from $\pa_x u +\pa_y v =0$ that 
\begin{align}\label{paxubis1}
    \pa_x\int_0^1u(t,x,y)\, dy = -\int_0^1 \dd_y v(t,x,y)\, dy = v(t,x,1) - v(t,x,0) = 0,
\end{align}
which together with the fact: $u(t,x,y) \rightarrow 0$ as $|x| \rightarrow \infty$, ensure that 
\begin{align}\label{incompatofubis1}
    \int_0^1 u(t,x,y) dy = 0.
\end{align}

 Then by integrating the equations $\pa_t^2 u + \pa_tu+u\pa_x u+v\pa_yu-\pa_y^2u+\pa_xp=0$ and, for $y\in[0,1]$ and using the fact that $\pa_y p=0$, we obtain 
 \begin{align}\label{pressureestimatebis1}
 \pa_x p = \pa_y u(t,x,1) - \pa_yu(t,x,0) - \pa_x \int_0^1 (u)^2(t,x,y) dy
 \end{align}

In view of System \eqref{eq:HPNhydrolimit}, we can transform it like a equation of order one in time, so if we define $V = (u,\pa_tu)$, Then $V$ satisfy the following equation
\begin{equation}
	\label{eq;hydroVbis1}
	\left\{\;
	\begin{aligned}
		&\partial_t V + A(D)V =-\begin{pmatrix}
		0 \\
		u\pa_xu + v\pa_y u + \pa_x p
		\end{pmatrix} \\
		&\pa_y p = 0\\
		&\pa_x u + \pa_y v = 0	 \\
		&(u,v)/_{y=0} = (u,v)/_{y=1} = 0	
	\end{aligned}
	\right.
\end{equation}
where 
$$ V = \begin{pmatrix}
		u \\
		\pa_tu
		\end{pmatrix} \text{ \ \ and \ \ } A(D) = \begin{pmatrix}
		0 &  -1 \\
		-\pa_y^2 & 1
		\end{pmatrix} $$ 
		
Then in view of \eqref{eq:Anphichp3} we observe that $V_\phi$ verifies 
		\begin{equation}
	\label{eq:hydroVphibis1}
	\left\{\;
	\begin{aligned}
		&\partial_t V_\phi + \lambda \dot{\theta}(t) |D_x| V_\phi + A(D)V_\phi =-\begin{pmatrix}
		0 \\
		(u\pa_xu)_{\phi} + (v\pa_y u)_{\phi} + \pa_x p_{\phi}
		\end{pmatrix} \\
		&\pa_y p_\phi = 0\\
		&\pa_x u_\phi + \pa_y v_\phi = 0	 \\
		&(u_\phi,v_\phi)/_{y=0} = (u_\phi,v_\phi)/_{y=1} = 0	
	\end{aligned}
	\right.
\end{equation}

Where $|D_x|$ denotes  the Fourier multiplier with symbol $|\xi|$.
 
The main idea of this technique consists in the fact that if we differentiate, with respect to the time variable a function of the type $e^{\phi(t,D_x)} u(t,x,y)$, we obtain an additional ``good term'' which plays the smoothing role. More precisely, we have
\begin{displaymath}
	\frac{d}{dt} \pare{e^{\phi(t,D_x)} V(t,x,y)} = -\dot{\theta}(t) \abs{D_x} e^{\phi(t,D_x)} V(t,x,y) + e^{\phi(t,D_x)} \dd_t V(t,x,y),
\end{displaymath}
where $-\dot{\theta}(t) \abs{D_x} e^{\phi(t,D_x)} u(t,x,y)$ gives a smoothing effect if $\dot{\theta}(t) \geq 0$. This smoothing effect allows to obtain our global existence and stability results in the analytic framework.

\paragraph{\textbf{Proof of global well-posedness of system \eqref{eq:HPNhydrolimit} }}
By applying the dyadic operator in the horizontal variable $\Delta_q^h$ to \eqref{eq:hydroVphibis1} and taking the $L^2$ inner product of the resulting equation with $\Delta_q^h (V_\phi)$ we obtain 
\begin{multline} \label{S3eq1bis1}
	 \psca{\Delta_q^h \pa_t V_\phi, \Delta_q^h V_{\phi}}_{L^2} + \lambda \dot{\theta}(t) \psca{\Delta_q^h |D_x| V_\phi, \Delta_q^h V_{\phi}}_{L^2} + \psca{\Delta_q^h A(D) V_\phi, \Delta_q^h V_{\phi}}_{L^2} \\
	= -\psca{\Delta_q^h (u\pa_x u + v\pa_y u)_{\phi}, \Delta_q^h (\pa_tu)_{\phi}}_{L^2} - \psca{\Delta_q^h \pa_x p_{\phi}, \Delta_q^h (\pa_tu)_{\phi}}_{L^2}.
\end{multline}

In what follows, we shall use the technical lemmas in Section \ref{sec:1}, to handle term by term in Estimate \eqref{S3eq1bis1}.

By applying the result of Lemma \ref{lem:1}, we find that 
\begin{align}\label{esti:Vphi}
&\psca{\Delta_q^h \pa_t V_\phi, \Delta_q^h V_{\phi}}_{L^2} + \lambda \dot{\theta}(t) \psca{\Delta_q^h |D_x| V_\phi, \Delta_q^h V_{\phi}}_{L^2} = \psca{\Delta_q^h \pa_t u_\phi, \Delta_q^h u_{\phi}}_{L^2} \nonumber\\
& + \psca{\Delta_q^h \pa_t (\pa_tu)_\phi, \Delta_q^h (\pa_tu)_{\phi}}_{L^2} + \lambda \dot{\theta}(t) \Big(\psca{\Delta_q^h |D_x| u_\phi, \Delta_q^h u_{\phi}}_{L^2}+ \psca{\Delta_q^h |D_x| (\pa_tu)_\phi, \Delta_q^h (\pa_tu)_{\phi}}_{L^2} \Big) \nonumber \\
& = \frac{1}{2}\frac{d}{dt}\Big(\|\Delta_q^h u_\phi \|_{L^2}^2+ \|\Delta_q^h (\pa_tu)_\phi \|_{L^2}^2  \Big) + \lambda \dot{\theta}(t) \Big(\|\Delta_q^h |D_x|^\frac{1}{2}u_\phi\|_{L^2}^2 + \|\Delta_q^h |D_x|^\frac{1}{2}(\pa_tu)_\phi\|_{L^2}^2 \Big) ,
\end{align}
and by using the fact that $(\pa_t u)_\phi = \pa_tu_\phi + \lambda \dot{\theta}(t)|D_x| u_\phi$, we have 
\begin{align}\label{esti:A(D)}
 \psca{\Delta_q^h A(D) V_\phi, \Delta_q^h V_{\phi}}_{L^2} &=  -\psca{\Delta_q^h (\pa_tu)_\phi, \Delta_q^h u_{\phi}}_{L^2} -  \psca{\Delta_q^h \pa_y^2 u_\phi, \Delta_q^h (\pa_tu)_{\phi}}_{L^2}+ \psca{\Delta_q^h (\pa_t u)_\phi, \Delta_q^h (\pa_tu)_{\phi}}_{L^2} \nonumber \\
 & =  -\frac{1}{2} \frac{d}{dt} \|\Delta_q^h u_\phi \|_{L^2}^2 - \lambda \dot{\theta}(t) \|\Delta_q^h |D_x|^\frac{1}{2}u_\phi\|_{L^2}^2 \nonumber \\ & + \frac{1}{2}\frac{d}{dt}\|\Delta_q^h \pa_yu_\phi \|_{L^2}^2 +  \lambda \dot{\theta}(t) \|\Delta_q^h |D_x|^\frac{1}{2}\pa_y u_\phi\|_{L^2}^2+\|\Delta_q^h (\pa_tu)_\phi\|_{L^2}^2.
\end{align}
 
 We sum the two equality \eqref{esti:Vphi} and \eqref{esti:A(D)}, we obtain 
 \begin{align}\label{S1eq3bis1}
&\psca{\Delta_q^h \pa_t V_\phi, \Delta_q^h V_{\phi}}_{L^2} + \lambda \dot{\theta}(t) \psca{\Delta_q^h |D_x| V_\phi, \Delta_q^h V_{\phi}}_{L^2} + \psca{\Delta_q^h A(D) V_\phi, \Delta_q^h V_{\phi}}_{L^2} =  \frac{1}{2}\frac{d}{dt}\Big( \|\Delta_q^h (\pa_tu)_\phi \|_{L^2}^2  \nonumber \\ &+ \|\Delta_q^h \pa_yu_\phi \|_{L^2}^2 \Big) + \lambda \dot{\theta}(t) \Big( \|\Delta_q^h |D_x|^\frac{1}{2}(\pa_tu)_\phi\|_{L^2}^2 + \|\Delta_q^h |D_x|^\frac{1}{2}\pa_y u_\phi\|_{L^2}^2 \Big) +\|\Delta_q^h (\pa_tu)_\phi\|_{L^2}^2.
 \end{align}
Now, for the pressure term, using the Dirichlet boundary condition $(u,v)\vert_{y=0} =(u,v)\vert_{y=1} = 0$, and the incompressibility condition $\pa_x u +\pa_y v =0$, by performing an integration by parts, we get 
\begin{align*}
	D = \abs{\psca{\Delta_q^h \pa_x p_{\phi}, \Delta_q^h (\pa_tu)_{\phi}}_{L^2}} &= \abs{ \psca{\Delta_q^h  p_{\phi}, \Delta_q^h (\pa_t \pa_xu)_{\phi}}_{L^2}}\\ & = \abs{\psca{\Delta_q^h  p_{\phi}, \Delta_q^h (\pa_t\pa_y v)_{\phi}}_{L^2}} \\ &= \abs{ \psca{\Delta_q^h  \pa_y p_{\phi}, \Delta_q^h  (\pa_tv)_{\phi}}_{L^2}}  = 0,
\end{align*}
In the last step we use the fact that $\pa_y p_\phi =0$. Thus,
\begin{align*}
	D = \abs{\psca{\Delta_q^h \pa_x p_{\phi}, \Delta_q^h (\pa_tu)_{\phi}}_{L^2}} = 0.
\end{align*}

While due to $u/_{y=0} = u/_{y=1} = 0$, by applying Poincar\'e inequality, we have

\begin{align}\label{poincarebis1}
 k \| \Delta^h u_\phi \|_{L^2}^2 \leq  \| \Delta_q^h \pa_y u_\phi \|_{L^2}^2,
\end{align}
where $k$ is the Poincaré constant.
Then, multiplying \eqref{S1eq3bis1} by $e^{2\mathcal{R}t}$ ($\mathcal{R}$ is a constant smaller than Poincaré constant denoted by $k$), we achieve
\begin{multline} \label{S3eq2bis1}
	 \frac{1}{2}\frac{d}{dt}\Big( \|e^{\mathcal{R} t}\Delta_q^h (\pa_tu)_\phi \|_{L^2}^2 + \|e^{\mathcal{R} t}\Delta_q^h \pa_yu_\phi \|_{L^2}^2 \Big)+ \|e^{\mathcal{R} t}\Delta_q^h (\pa_tu)_\phi\|_{L^2}^2 \\ +  \lambda \dot{\theta}(t)  \|e^{\mathcal{R} t}\Delta_q^h |D_x|^\frac{1}{2}(\pa_t u)_\phi\|_{L^2}^2 +\lambda \dot{\theta}(t) \|e^{\mathcal{R} t}\Delta_q^h |D_x|^\frac{1}{2}\pa_y u_\phi\|_{L^2}^2 
\\	\leq \mathcal{R}\left(  \|e^{\mathcal{R} t}\Delta_q^h (\pa_tu)_\phi \|_{L^2}^2 +\|e^{\mathcal{R} t}\Delta_q^h \pa_yu_\phi \|_{L^2}^2 \right) + \abs{\psca{\Delta_q^h (u\pa_x u)_{\phi}, e^{2\mathcal{R} t}\Delta_q^h (\pa_tu)_{\phi}}_{L^2}} \\ + \abs{\psca{\Delta_q^h ( v\pa_y u)_{\phi}, e^{2\mathcal{R} t}\Delta_q^h (\pa_tu)_{\phi}}_{L^2}} .
\end{multline}
In what follows, we shall always assume that $t <T^\star$, where $T^\star$ given by 
\begin{equation} \label{eq:Tstarbisbis11} 
	T^\star \stackrel{\tiny def}{=} \sup\set{t>0,\  \ \ \theta(t) \leq  \frac{a}{\lambda}}.
\end{equation}
Next, we note that 
\begin{equation*} 
	\left\{\;
	\begin{aligned}
		&J^q_1(t) = \abs{\psca{\Delta_q^h (u\pa_x u)_{\phi}, e^{2\mathcal{R} t}\Delta_q^h (\pa_tu)_{\phi}}_{L^2}}\\
		&J^q_2(t) = \abs{\psca{\Delta_q^h (v\pa_y u)_{\phi}, e^{2\mathcal{R} t}\Delta_q^h(\pa_tu)_{\phi}}_{L^2}}.
	\end{aligned}
	\right.
\end{equation*}

To estimate those two non linear terms we need to use Lemmas \ref{lem:ABCbis1} and \ref{lem:Apa_yBCbis1}. In view of Lemma \ref{lem:ABCbis1}, we replace $C_\phi$ by $(\pa_tu)_\phi$ and $ A=B= u$, then we conclude the following estimate of $J^q_1$ for any $t<T^\star$
\begin{align}\label{estim:Jq1bis1}
 \int_0^t J^q_1(t') dt' = & \int_0^t \abs{\psca{e^{\mathcal{R} t'}\Delta_q^h (u\pa_x u)_{\phi}, e^{\mathcal{R} t'}\Delta_q^h(\pa_tu)_{\phi}}_{L^2}} dt' \nonumber \\ &\leq C 2^{-2qs} d_q^2 \Vert e^{\mathcal{R} t} u_{\phi} \Vert_{\tilde{L}^2_{t,\dot{\theta}(t)}(\mathcal{B}^{s+\frac{1}{2}})} \Vert e^{\mathcal{R} t} (\pa_tu)_{\phi} \Vert_{\tilde{L}^2_{t,\dot{\theta}(t)}(\mathcal{B}^{s+\frac{1}{2}})},
\end{align}

where $\dot{\theta}(t) = \|\pa_y u_\phi\|_{\mathcal{B}^\f12}$.

In view of Lemma \ref{lem:Apa_yBCbis1}, we replace $C_\phi$ by $(\pa_tu)_\phi$, $ A= u$ and $B=v$, then we conclude the following estimate of $J^q_2$ for any $t<T^\star$
\begin{align}\label{estim:Jq2bis1}
\int_0^t J^q_2(t') dt'  = & \abs{\psca{e^{\mathcal{R} t'}\Delta_q^h (v\pa_y u)_{\phi}, e^{\mathcal{R} t'}\Delta_q^h (\pa_tu)_{\phi}}_{L^2}} \nonumber \\ &\leq C 2^{-2qs} d_q^2 \Vert e^{\mathcal{R} t} u_{\phi} \Vert_{\tilde{L}^2_{t,\dot{\theta}(t)}(\mathcal{B}^{s+\frac{1}{2}})} \Vert e^{\mathcal{R} t} (\pa_tu)_{\phi} \Vert_{\tilde{L}^2_{t,\dot{\theta}(t)}(\mathcal{B}^{s+\frac{1}{2}})}, 
\end{align}

where $\dot{\theta}(t) = \|\pa_y u_\phi\|_{\mathcal{B}^\f12}$.

Then we deduce from $\eqref{S3eq2bis1}$ by integrating with respect to the time interval, that 
\begin{multline} \label{S3eq3bis11}
	 \int_0^t \frac{1}{2}\frac{d}{dt}\Big( \|e^{\mathcal{R} t'}\Delta_q^h (\pa_tu)_\phi(t') \|_{L^2}^2 + \|e^{\mathcal{R} t}\Delta_q^h \pa_yu_\phi(t') \|_{L^2}^2 \Big) dt'+ \int_0^t \|e^{\mathcal{R} t'}\Delta_q^h (\pa_tu)_\phi(t')\|_{L^2}^2 dt'\\ +  \lambda \int_0^t \dot{\theta}(t')  \|e^{\mathcal{R} t'}\Delta_q^h |D_x|^\frac{1}{2}(\pa_t u)_\phi(t')\|_{L^2}^2 dt' +\lambda \int_0^t \dot{\theta}(t') \|e^{\mathcal{R} t'}\Delta_q^h |D_x|^\frac{1}{2}\pa_y u_\phi(t') \|_{L^2}^2 dt' 
\\	\leq \mathcal{R}\left(  \|e^{\mathcal{R} t}\Delta_q^h (\pa_tu)_\phi \|_{L^2}^2 +\|e^{\mathcal{R} t}\Delta_q^h \pa_yu_\phi \|_{L^2}^2 \right) +  C2^{-2qs} d_q^2 \Vert e^{\mathcal{R} t} u_{\phi} \Vert_{\tilde{L}^2_{t,\dot{\theta}(t)}(\mathcal{B}^{s+\frac{1}{2}})} \Vert e^{\mathcal{R} t} (\pa_tu)_{\phi} \Vert_{\tilde{L}^2_{t,\dot{\theta}(t)}(\mathcal{B}^{s+\frac{1}{2}})}
\end{multline}

We multiply the estimate \eqref{S3eq3bis11} by $2$, we obtain 

\begin{multline} \label{S3eq3bis1}
	 \int_0^t \frac{d}{dt}\Big( \|e^{\mathcal{R} t'}\Delta_q^h (\pa_tu)_\phi(t') \|_{L^2}^2 + \|e^{\mathcal{R} t}\Delta_q^h \pa_yu_\phi(t') \|_{L^2}^2 \Big) dt'+ 2\int_0^t \|e^{\mathcal{R} t'}\Delta_q^h (\pa_tu)_\phi(t')\|_{L^2}^2 dt'\\ +  2\lambda \int_0^t \dot{\theta}(t')  \|e^{\mathcal{R} t'}\Delta_q^h |D_x|^\frac{1}{2}(\pa_t u)_\phi(t')\|_{L^2}^2 dt' +2\lambda \int_0^t \dot{\theta}(t') \|e^{\mathcal{R} t'}\Delta_q^h |D_x|^\frac{1}{2}\pa_y u_\phi(t') \|_{L^2}^2 dt' 
\\	\leq 2\mathcal{R}\left(  \|e^{\mathcal{R} t}\Delta_q^h (\pa_tu)_\phi \|_{L^2}^2 +\|e^{\mathcal{R} t}\Delta_q^h \pa_yu_\phi \|_{L^2}^2 \right) +  2C2^{-2qs} d_q^2 \Vert e^{\mathcal{R} t} u_{\phi} \Vert_{\tilde{L}^2_{t,\dot{\theta}(t)}(\mathcal{B}^{s+\frac{1}{2}})} \Vert e^{\mathcal{R} t} (\pa_tu)_{\phi} \Vert_{\tilde{L}^2_{t,\dot{\theta}(t)}(\mathcal{B}^{s+\frac{1}{2}})}
\end{multline}

Now we still have to get some information of the norm $\|\pa_y u_\phi\|_{\mathcal{B}^\f12}$, for that we need to apply the dyadic operator $\Delta_q^h$ to the equation 
\begin{equation}\label{eq:Hnsbis1}
    e^{\phi(t,|D_x|)} (\pa_t^2 u + \pa_t u + u\pa_xu+v\pa_yu-\pa_y^2u +\pa_x p) = 0,
\end{equation}
and then, we take the $L^2$ inner product of the resulting equation \eqref{eq:Hnsbis1} with $\Delta_q^h u_\phi$, we obtain  
\begin{multline} \label{S3eq1bis1bis1}
	 \psca{\Delta_q^h (\pa_t^2u)_\phi, \Delta_q^h  u_{\phi}}_{L^2} + \psca{\Delta_q^h (\pa_tu)_\phi, \Delta_q^h  u_{\phi}}_{L^2} - \psca{\Delta_q^h \pa_y^2u_\phi, \Delta_q^h u_{\phi}}_{L^2} \\
	= -\psca{\Delta_q^h (u\pa_x u + v\pa_y u)_{\phi}, \Delta_q^h u_{\phi}}_{L^2} - \psca{\Delta_q^h \pa_x p_{\phi}, \Delta_q^h u_{\phi}}_{L^2}.
\end{multline}

In what follows, we shall use again the technical lemmas in Section \ref{sec:1}, to handle term by term in the estimate \eqref{S3eq1bis1bis1}. We start by the term $$I_1 =\psca{\Delta_q^h (\pa_t^2u)_\phi, \Delta_q^h u_{\phi}}_{L^2} \text{ \ and \ } I_2 = \psca{\Delta_q^h (\pa_tu)_\phi, \Delta_q^h u_{\phi}}_{L^2} $$
so by using integration by parts, we find 
\begin{align*}
    I_1 = \frac{d}{dt
    }\int \Delta_q^h (\pa_t u)_\phi &\Delta_q^h  u_\phi dX - \int \Delta_q^h (\pa_t u)_\phi \Delta_q^h (\pa_tu)_\phi dX \\ &+ 2\lambda \dot{\theta}(t)\int \Delta_q^h |D_x| (\pa_tu)_\phi \Delta_q^h u_\phi dX, 
\end{align*}
and 
\begin{align*}
    I_2 = \frac{1}{2}\frac{d}{dt} \|\Delta_q^h u_\phi\|_{L^2}^2 + \lambda \dot{\theta}(t) \|\Delta_q^h |D_x|^\frac{1}{2} u_\phi \|_{L^2}^2.
\end{align*}

Whereas due to the boundary condition, and by integrating by part, we achieve 
\begin{align*}
    \psca{\Delta_q^h (-\pa_y^2u_\phi), \Delta_q^h u_{\phi}}_{L^2} = \psca{\Delta_q^h \pa_yu_\phi, \Delta_q^h \pa_y u_{\phi}}_{L^2} = \|\Delta_q^h  \pa_yu_\phi \|_{L^2}^2.
\end{align*}

Now, by using the Dirichlet boundary condition $(u,v)\vert_{y=0} =(u,v)\vert_{y=1} = 0$, and the incompressibility condition $\pa_x u +\pa_y v =0$ and the relation $\pa_y p_\phi = 0$, we can find by integrating by parts the estimate of the pressure  
\begin{align*}
	 \abs{\psca{\Delta_q^h \pa_x p_{\phi}, \Delta_q^h u_{\phi}}_{L^2}} &= \abs{ \psca{\Delta_q^h  p_{\phi}, \Delta_q^h \pa_xu_{\phi}}_{L^2}}\\ & = \abs{\psca{\Delta_q^h  p_{\phi}, \Delta_q^h \pa_y v_{\phi}}_{L^2}} \\&= \abs{ \psca{\Delta_q^h  \pa_y p_{\phi}, \Delta_q^h  v_{\phi}}_{L^2}}  = 0.
\end{align*}
Then by using Lemma \ref{lem:Bernsteinbis1} and by multiplying \eqref{S3eq1bis1bis1} by $e^{2\mathcal{R}t}$, we achieve
\begin{align} \label{S3eq1bis3bis1}
	 \frac{d}{dt}\int &e^{2\mathcal{R}t}\Delta_q^h (\pa_t u)_\phi \Delta_q^h  u_\phi dX - \int e^{2\mathcal{R}t}\Delta_q^h (\pa_t u)_\phi \Delta_q^h (\pa_tu)_\phi dX \nonumber + 2\lambda \dot{\theta}(t)\int e^{2\mathcal{R}t} \Delta_q^h |D_x| (\pa_tu)_\phi \Delta_q^h u_\phi dX \\  &+ \frac{1}{2}\frac{d}{dt} \|e^{\mathcal{R}t}\Delta_q^h u_\phi\|_{L^2}^2 + \lambda \dot{\theta}(t) \|e^{\mathcal{R}t}\Delta_q^h |D_x|^\frac{1}{2} u_\phi \|_{L^2}^2 + \|e^{\mathcal{R}t}\Delta_q^h  \pa_yu_\phi \|_{L^2}^2 \nonumber \\
	 & = 2\mathcal{R}\int e^{2\mathcal{R}t}\Delta_q^h (\pa_t u)_\phi \Delta_q^h  u_\phi dx + \mathcal{R} \|e^{\mathcal{R}t}\Delta_q^h u_\phi\|_{L^2}^2 \nonumber   \\ & -\psca{\Delta_q^h (u\pa_x u)_{\phi}, e^{2\mathcal{R}t}\Delta_q^h u_{\phi}}_{L^2}-\psca{\Delta_q^h (v\pa_y u)_{\phi}, e^{2\mathcal{R}t}\Delta_q^h u_{\phi}}_{L^2}.       
\end{align}

Next, we note that 
\begin{equation*} 
	\left\{\;
	\begin{aligned}
		&L^q_1(t) = \abs{\psca{\Delta_q^h (u\pa_x u)_{\phi}, e^{2\mathcal{R} t}\Delta_q^h u_{\phi}}_{L^2}}\\
		&L^q_2(t) = \abs{\psca{\Delta_q^h (v\pa_y u)_{\phi}, e^{2\mathcal{R} t}\Delta_q^h u_{\phi}}_{L^2}} \\
		&L^3_q (t)= \abs{2\lambda \dot{\theta}(t)\int e^{2\mathcal{R}t} \Delta_q^h |D_x| (\pa_tu)_\phi \Delta_q^h u_\phi dX} .
	\end{aligned}
	\right.
\end{equation*}

In view, of lemma \ref{lem:ABCbis1}-\ref{lem:Apa_yBCbis1}, we can deduce for $t< T^\star$ that
\begin{align}\label{esti:L_q1bis1}
\int_0^t L^q_1(t') dt'  &= \int_0^t \abs{\psca{\Delta_q^h (u\pa_x u)_{\phi}, e^{2\mathcal{R} t'}\Delta_q^h u_{\phi}}_{L^2}} dt' \nonumber \\
&\leq C2^{-2qs}d_q^2 \Vert e^{\mathcal{R} t} u_{\phi} \Vert_{\tilde{L}^2_{t,\dot{\theta}(t)}(\mathcal{B}^{s+\frac{1}{2}})}^2
\end{align}
and
\begin{align}\label{esti:L_q2bis1}
\int_0^t L^q_2(t') dt'  &= \int_0^t \abs{\psca{\Delta_q^h (v\pa_y u)_{\phi}, e^{2\mathcal{R} t'}\Delta_q^h u_{\phi}}_{L^2}} dt' \nonumber \\
&\leq C2^{-2qs}d_q^2 \Vert e^{\mathcal{R} t} u_{\phi} \Vert_{\tilde{L}^2_{t,\dot{\theta}(t)}(\mathcal{B}^{s+\frac{1}{2}})}^2.
\end{align}

Then we still have to estimate $L^3_q$, therefore by cutting the derivative $|D_x|$ into two half derivative, we achieve the 
\begin{align}\label{esti:L_q3bis1}
  L^3_q(t)&=  \abs{2\lambda \dot{\theta}(t)\int e^{2\mathcal{R}t'} \Delta_q^h |D_x| (\pa_tu)_\phi \Delta_q^h u_\phi dX}  \nonumber \\
 &\leq 2\lambda \dot{\theta}(t)\int \abs{ e^{\mathcal{R}t}\Delta_q^h |D_x|^\frac{1}{2} (\pa_tu)_\phi} \abs{ e^{\mathcal{R}t'}\Delta_q^h |D_x|^\frac{1}{2} u_\phi} dX \nonumber \\
 &\leq  \lambda \dot{\theta}(t) \Vert e^{\mathcal{R} t} \Delta_q^h |D_x|^\frac{1}{2}(\pa_tu)_{\phi} \Vert_{L^2}^2 + {\lambda} \dot{\theta}(t) \Vert e^{\mathcal{R} t} \Delta_q^h |D_x|^\frac{1}{2}  u_{\phi} \Vert_{L^2}^2,
\end{align}

 Next, by taking the integration over time of \eqref{S3eq1bis3bis1}, and we sum it with $\eqref{S3eq3bis1}$, we gather that
\begin{multline} \label{estimate:2222}
\int_0^t \frac{d}{dt}\Big( \frac{1}{2} \|e^{\mathcal{R} t'}\Delta_q^h (u + \pa_tu)_\phi \|_{L^2}^2 + \|e^{\mathcal{R} t'} \Delta_q^h \pa_yu_\phi \|_{L^2}^2 + \frac{1}{2} \|e^{\mathcal{R} t'}\Delta_q^h (\pa_tu)_\phi \|_{L^2}^2\Big)dt' + \int_0^t \|e^{\mathcal{R} t'}\Delta_q^h (\pa_tu)_\phi\|_{L^2}^2 dt' \\  + \int_0^t \lambda \dot{\theta}(t') \|e^{\mathcal{R} t'}\Delta_q^h |D_x|^\f12 u_\phi\|_{L^2}^2 dt'+ \int_0^t \|e^{\mathcal{R} t'}\Delta_q^h\pa_yu_\phi\|_{L^2}^2 dt' +  \int_0^t 2\lambda \dot{\theta}(t') \|e^{\mathcal{R} t'}\Delta_q^h |D_x|^\f12(\pa_t u)_\phi\|_{L^2}^2 dt' \\ + \int_0^t  2\lambda \dot{\theta}(t') \|e^{\mathcal{R} t'} \Delta_q^h |D_x|^\f12\pa_y u_\phi\|_{L^2}^2 dt' 	\leq  \mathcal{R}\int_0^t \Big[\|e^{\mathcal{R} t'}\Delta_q^h (u + \pa_tu)_\phi \|_{L^2}^2 + 2\|e^{\mathcal{R} t'} \Delta_q^h \pa_yu_\phi \|_{L^2}^2 \\
 + \|e^{\mathcal{R} t'}\Delta_q^h (\pa_tu)_\phi \|_{L^2}^2 \Big] dt' +  \lambda \int_0^t \dot{\theta}(t') \Vert e^{\mathcal{R} t'} \Delta_q^h |D_x|^\frac{1}{2}(\pa_tu)_{\phi} \Vert_{L^2}^2 dt'+ {\lambda}\int_0^t \dot{\theta}(t') \Vert e^{\mathcal{R} t'} \Delta_q^h |D_x|^\frac{1}{2} u_{\phi} \Vert_{L^2}^2 dt'
 \\  + 2 C d_q^2 2^{-2qs} \Big( \Vert e^{\mathcal{R} t} (\pa_tu)_{\phi} \Vert_{\tilde{L}^2_{t,\dot{\theta}(t)}(\mathcal{B}^{s+\frac{1}{2}})}^2  + \Vert e^{\mathcal{R} t} u_{\phi} \Vert_{\tilde{L}^2_{t,\dot{\theta}(t)}(\mathcal{B}^{s+\frac{1}{2}})}^2 \Big).
\end{multline} 

We begin with by observing that the term in the square brackets of the right-hand side in \eqref{estimate:2222} can be absorbed by the dissipation $\int_0^t \|e^{\mathcal{R} t'}\Delta_q^h (\pa_tu)_\phi\|_{L^2}^2 dt' + \int_0^t \|e^{\mathcal{R} t'}\Delta_q^h\pa_yu_\phi\|_{L^2}^2 dt'$. Indeed, since the value of $\mathcal{R}$ is smaller than $\min\lbrace \frac{1}{8},\frac{k}{8} \rbrace$, we have that 
\begin{align*}
\mathcal{R} &\Big[\|e^{\mathcal{R} t'}\Delta_q^h (u + \pa_tu)_\phi \|_{L^2}^2 + 2\|e^{\mathcal{R} t'} \Delta_q^h \pa_yu_\phi \|_{L^2}^2 + \|e^{\mathcal{R} t'}\Delta_q^h (\pa_tu)_\phi \|_{L^2}^2 \Big] \\ &\leq \mathcal{R} \Big[\|e^{\mathcal{R} t'}\Delta_q^h u _\phi \|_{L^2}^2 + 2\|e^{\mathcal{R} t'} \Delta_q^h \pa_yu_\phi \|_{L^2}^2 + 2\|e^{\mathcal{R} t'}\Delta_q^h (\pa_tu)_\phi \|_{L^2}^2 \Big] \\
&\leq \frac{k}{8} \|e^{\mathcal{R} t'}\Delta_q^h u _\phi \|_{L^2}^2 + \frac{1}{4} \|e^{\mathcal{R} t'} \Delta_q^h \pa_yu_\phi \|_{L^2}^2 +\frac{1}{4} \|e^{\mathcal{R} t'}\Delta_q^h (\pa_tu)_\phi \|_{L^2}^2.
\end{align*}
To absorb this last term, we shall then invoke the Poincaré inequality in $y\in(0,1)$ : $k \|\Delta_q^h u_\phi \|_{L^2} \leq \|\Delta_q^h \pa_y u_\phi\|_{L^2}$. Thus 
\begin{align}\label{esti:Poincaré}
\mathcal{R} &\Big[\|e^{\mathcal{R} t'}\Delta_q^h (u + \pa_tu)_\phi \|_{L^2}^2 + 2\|e^{\mathcal{R} t'} \Delta_q^h \pa_yu_\phi \|_{L^2}^2 + \|e^{\mathcal{R} t'}\Delta_q^h (\pa_tu)_\phi \|_{L^2}^2 \Big] \nonumber \\
&\leq \frac{1}{2}\Big(\|e^{\mathcal{R} t'} \Delta_q^h \pa_yu_\phi \|_{L^2}^2 + \|e^{\mathcal{R} t'}\Delta_q^h (\pa_tu)_\phi \|_{L^2}^2\Big). 
\end{align}
We replace the obtained result \eqref{esti:Poincaré} in \eqref{estimate:2222}, we deduce that 

\begin{multline} \label{estimate:22224}
\int_0^t \frac{d}{dt}\Big( \frac{1}{2} \|e^{\mathcal{R} t'}\Delta_q^h (u + \pa_tu)_\phi \|_{L^2}^2 + \|e^{\mathcal{R} t'} \Delta_q^h \pa_yu_\phi \|_{L^2}^2 + \frac{1}{2} \|e^{\mathcal{R} t'}\Delta_q^h (\pa_tu)_\phi \|_{L^2}^2\Big)dt' \\ + \f12\int_0^t \|e^{\mathcal{R} t'}\Delta_q^h (\pa_tu)_\phi\|_{L^2}^2 dt' + \f12\int_0^t \|e^{\mathcal{R} t'}\Delta_q^h\pa_yu_\phi\|_{L^2}^2 dt'\\ +  \int_0^t \lambda \dot{\theta}(t') \|e^{\mathcal{R} t'}\Delta_q^h |D_x|^\f12(\pa_t u)_\phi\|_{L^2}^2 dt' + \int_0^t 2\lambda \dot{\theta}(t') \|e^{\mathcal{R} t'} \Delta_q^h |D_x|^\f12\pa_y u_\phi\|_{L^2}^2 dt'  \\\leq  2 C d_q^2 2^{-2qs} \Big(\Vert e^{\mathcal{R} t} (\pa_tu)_{\phi} \Vert_{\tilde{L}^2_{t,\dot{\theta}(t)}(\mathcal{B}^{s+\frac{1}{2}})}^2  + \Vert e^{\mathcal{R} t} u_{\phi} \Vert_{\tilde{L}^2_{t,\dot{\theta}(t)}(\mathcal{B}^{s+\frac{1}{2}})}^2 \Big).
\end{multline} 
 
 In the last estimate \eqref{estimate:22224}, we remark that the term $\int_0^t \lambda \dot{\theta}(t') \|e^{\mathcal{R} t'} \Delta_q^h |D_x|^\f12 u_\phi\|_{L^2}^2 dt'$ has disappeared because of the second term that comes out of the estimate \eqref{esti:L_q3bis1}. We can make it appear by using Poincaré inequality on the term $\int_0^t 2\lambda \dot{\theta}(t') \|e^{\mathcal{R} t'} \Delta_q^h |D_x|^\f12\pa_y u_\phi\|_{L^2}^2 dt'$, we have 
 \begin{align*}
 \int_0^t &2\lambda \dot{\theta}(t') \|e^{\mathcal{R} t'} \Delta_q^h |D_x|^\f12\pa_y u_\phi\|_{L^2}^2 dt' \\
 &\geq \lambda\Big(k\int_0^t  \dot{\theta}(t') \|e^{\mathcal{R} t'} \Delta_q^h |D_x|^\f12 u_\phi\|_{L^2}^2 dt' + \int_0^t \dot{\theta}(t') \|e^{\mathcal{R} t'} \Delta_q^h |D_x|^\f12\pa_y u_\phi\|_{L^2}^2 dt'\Big).
\end{align*}

Now we we multiply the obtained result in \eqref{estimate:2222} by $2^{2qs}$ for $s\in]0,1[$ and summing with respect to $q \in \ZZ$, we find that for $t<T^\star$

\begin{multline} \label{S3eq6bis1}
	 \Big( \frac{1}{2} \|e^{\mathcal{R} t}(u + \pa_tu)_\phi \|_{\tilde{L}_t^\infty(\mathcal{B}^{s})} + \|e^{\mathcal{R} t} \pa_yu_\phi \|_{\tilde{L}_t^\infty(\mathcal{B}^{s})} + \frac{1}{2} \|e^{\mathcal{R} t}(\pa_tu)_\phi \|_{\tilde{L}_t^\infty(\mathcal{B}^{s})}\Big) \\+ \frac{1}{2}\Big(\|e^{\mathcal{R} t}(\pa_tu)_\phi\|_{\tilde{L}_t^2(\mathcal{B}^{s})} + \|e^{\mathcal{R} t}\pa_yu_\phi\|_{\tilde{L}_t^2(\mathcal{B}^{s})} \Big) + \sqrt{k\lambda} \|e^{\mathcal{R} t} u_\phi\|_{\tilde{L}^2_{t,\dot{\theta}(t)}(\mathcal{B}^{s+\frac{1}{2}})} +  \sqrt{\lambda}  \|e^{\mathcal{R} t}(\pa_t u)_\phi\|_{\tilde{L}^2_{t,\dot{\theta}(t)}(\mathcal{B}^{s+\frac{1}{2}})}  \\ +  \sqrt{\lambda} \|e^{\mathcal{R} t}\pa_y u_\phi\|_{\tilde{L}^2_{t,\dot{\theta}(t)}(\mathcal{B}^{s+\frac{1}{2}})} 	\leq C\Big( \|e^{a|D_x|}\pa_yu_0 \|_{\mathcal{B}^{s}} + \|e^{a|D_x|}(u_0+u_1) \|_{\mathcal{B}^{s}}+  \|e^{a|D_x|} u_1 \|_{\mathcal{B}^{s}} \Big)\\ + \sqrt{2}C \Vert e^{\mathcal{R} t} (\pa_tu)_{\phi} \Vert_{\tilde{L}^2_{t,\dot{\theta}(t)}(\mathcal{B}^{s+\frac{1}{2}})} + \sqrt{2}C \Vert e^{\mathcal{R} t} u_{\phi} \Vert_{\tilde{L}^2_{t,\dot{\theta}(t)}(\mathcal{B}^{s+\frac{1}{2}})}.
\end{multline}

Taking $\lambda \geq 4C^2$ and $k\lambda \geq 8C^2$ in the above inequality leads to 

\begin{multline} \label{S3eq7bis1}
	\mathcal{E}_{s,\frac{\lambda}{4}}(u)(t) \leq C\|e^{a|D_x|}\pa_yu_0 \|_{\mathcal{B}^{s}} +C\|e^{a|D_x|}(u_0+u_1) \|_{\mathcal{B}^{s}} \\ + C \|e^{a|D_x|} u_1 \|_{\mathcal{B}^{s}}, \text{ \ for \ } t<T^\star .
\end{multline}
where 
\begin{align*}
\mathcal{E}_{s,\frac{\lambda}{4}}(u)(t) &= \Big( \frac{1}{2}\|e^{\mathcal{R} t}(u + \pa_tu)_\phi \|_{\tilde{L}_t^\infty(\mathcal{B}^{s})} + \|e^{\mathcal{R} t} \pa_yu_\phi \|_{\tilde{L}_t^\infty(\mathcal{B}^{s})} \Big) \\
&+ \f12\|e^{\mathcal{R} t}(\pa_tu)_\phi\|_{\tilde{L}_t^2(\mathcal{B}^{s})} + \f12\|e^{\mathcal{R} t}\pa_yu_\phi\|_{\tilde{L}_t^2(\mathcal{B}^{s})} \\
& + \frac{\sqrt{k\lambda}}{2} \|e^{\mathcal{R} t} u_\phi\|_{\tilde{L}^2_{t,\dot{\theta}(t)}(\mathcal{B}^{s+\frac{1}{2}})} +  \frac{\sqrt{\lambda}}{2}  \|e^{\mathcal{R} t}(\pa_t u)_\phi\|_{\tilde{L}^2_{t,\dot{\theta}(t)}(\mathcal{B}^{s+\frac{1}{2}})}+  \frac{\sqrt{\lambda}}{2} \|e^{\mathcal{R} t}\pa_y u_\phi\|_{\tilde{L}^2_{t,\dot{\theta}(t)}(\mathcal{B}^{s+\frac{1}{2}})}
\end{align*}

We recall that we already defined $\dot{\theta}(t) = \Vert \pa_y u_{\phi}(t) \Vert_{\mathcal{B}^{\f12}}$ with $\theta(0) = 0$. Then, for any $0 < t < T^\star$, Inequality \eqref{S3eq7bis1} yields 

\begin{align*}
	\theta (t) &= \int_0^t \Vert \pa_y u_{\phi}(t') \Vert_{\mathcal{B}^\f12} dt' \\ &\leq \int_0^t  e^{-\mathcal{R}t'} \Vert e^{\mathcal{R}t'} \pa_y u_{\phi}(t') \Vert_{\mathcal{B}^\f12} dt' \\
	&\leq \left( \int_0^t  e^{-2\mathcal{R}t'} dt' \right)^{\frac{
	1}{2}} \times \left( \int_0^t \Vert e^{\mathcal{R}t'} \pa_y u_{\phi}(t') \Vert_{\mathcal{B}^\f12}^2  dt' \right)^{\frac{
	1}{2}} \\
	&\leq C \norm{e^{\mathcal{R} t} \pa_y u_{\phi}}_{\tilde{L}^2_t(\mathcal{B}^{\f12})} \\
	&\leq C \left( \|e^{a|D_x|}\pa_yu_0 \|_{\mathcal{B}^{\f12}} +\|e^{a|D_x|}(u_0+u_1) \|_{\mathcal{B}^{\f12}} + \|e^{a|D_x|} u_1 \|_{\mathcal{B}^{\f12}}^2 \right) \\ &< \frac{a}{2\lambda}
\end{align*}
We deduce from the continuity argument that $T^\star = +\infty$ and we have \eqref{S3eq7bis1} is valid for any $t \in \RR_+$.

\section{Propagation of the regularity and of the vorticity of the hyperbolic Prandtl equation \eqref{eq:HPNhydrolimit}} \label{sec:5}
In this section, we present first a proposition states the propagation for any $\mathcal{B}^s$ regularity on the solution of the perturbed hyperbolic Navier-Stokes equations \eqref{eq:HPNhydrolimit}. The second proposition allows us to control two derivatives in the normal direction $\pa_y^2$ in any $\mathcal{B}^s$, despite the difficulties raised by the boundary conditions. Those propositions will be useful in the last section when we prove the global convergence of the Theorem $\ref{th:33}$. In what follows, we shall always assume that $t <T^\star_a$, where $T^\star_a$ given by 
\begin{equation} \label{eq:T0starbisbis11} 
	T^\star_a \stackrel{\tiny def}{=} \sup\set{t>0,\  \ \ \theta(t) \leq  \frac{a}{\lambda}}.
\end{equation}

\begin{prop}\label{prop:1bis1}
We assume that the condition \eqref{eq:condini} is satisfied, then for any $s>0$, there exist a small constant $C$, such that for $$\lambda = C(1 + \| e^{a|D_x|}\pa_y u_0\|_{\mathcal{B}^{\f32}} + \|e^{a|D_x|}(u_0+u_1) \|_{\mathcal{B}^{\f32}} + \|e^{a|D_x|}u_1 \|_{\mathcal{B}^{\f32}}),$$ we have 
\begin{multline} \label{Velocity}
	 \Big( \frac{1}{2}\|e^{\mathcal{R} t} (u + \pa_tu)_\phi \|_{\tilde{L}_t^\infty(\mathcal{B}^{s})} + \frac{1}{2}\|e^{\mathcal{R} t} (\pa_tu)_\phi \|_{\tilde{L}_t^\infty(\mathcal{B}^{s})} + \|e^{\mathcal{R} t} \pa_y u_\phi \|_{\tilde{L}_t^\infty(\mathcal{B}^{s})} \Big)\\ + \f12\|e^{\mathcal{R} t}(\pa_tu)_\phi\|_{\tilde{L}_t^2(\mathcal{B}^{s})} + \f12\|e^{\mathcal{R} t}\pa_y u_\phi\|_{\tilde{L}_t^2(\mathcal{B}^{s})} \\ \leq C\|e^{a|D_x|}\pa_yu_0 \|_{\mathcal{B}^{s}} +C\|e^{a|D_x|} (u_0+u_1) \|_{\mathcal{B}^{s}} + C\|e^{a|D_x|} u_1 \|_{\mathcal{B}^{s}}, \text{ \ for \ } t<T^\star_a .
\end{multline}
\end{prop}

\noindent \emph{Proof of Proposition $\ref{prop:1bis1}$}. We first deduce from Lemma \ref{lem:ABCbis1} that for any $s>0$ we have by replacing $A= B=u$ and $C= \pa_t u$
 \ben \label{vtkbis11}
\int_0^t &|\psca{ \Delta_q^h e^{\mathcal{R}t'}(T^h_u \pa_x u + R^h(u,\pa_x u))_{\phi}, e^{\mathcal{R}t'}\Delta_q^h (\pa_tu)_{\phi}}_{L^2}| dt' \nonumber \\ &\lesssim d_q^2 2^{-2qs} \Vert e^{\mathcal{R}t} u_\phi \Vert_{\tilde{L}^2_{t,\dot{\theta}(t)}(\mathcal{B}^{s+\frac{1}{2}})} \Vert e^{\mathcal{R}t} (\pa_tu)_\phi \Vert_{\tilde{L}^2_{t,\dot{\theta}(t)}(\mathcal{B}^{s+\frac{1}{2}})}.
\een
Then we still have to proof only that 
\begin{align*}
\int_0^t &|\psca{ \Delta_q^h e^{\mathcal{R}t'}(T_{\pa_x u}^hu)_{\phi}, e^{\mathcal{R}t'}\Delta_q^h (\pa_tu)_{\phi}}_{L^2}| dt' \\&\lesssim d_q^2 2^{-2qs} \Vert e^{\mathcal{R}t} u_\phi \Vert_{\tilde{L}^2_{t,\dot{\theta}(t)}(\mathcal{B}^{s+\frac{1}{2}})} \Vert e^{\mathcal{R}t} (\pa_tu)_\phi \Vert_{\tilde{L}^2_{t,\dot{\theta}(t)}(\mathcal{B}^{s+\frac{1}{2}})}, \text{ \ \ for \ \ } s>0.
\end{align*}
Indeed, in view of \eqref{eq:normB12bis1}, we infer
\begin{align*}
    \int_0^t |&\psca{ \Delta_q^h e^{\mathcal{R}t'}(T_{\pa_x u}^hu)_{\phi}, e^{\mathcal{R}t'}\Delta_q^h (\pa_tu)_{\phi}}_{L^2}| dt' \\
    &\lesssim \sum_{|q'-q| \leq4} \int_0^t e^{2\mathcal{R}t'}\Vert S_{q'-1}^h \pa_x u_\phi(t') \Vert_{L^\infty} \Vert \Delta_{q'}^h u_\phi (t') \Vert_{L^2} \Vert \Delta_q^h (\pa_tu)_\phi(t') \Vert_{L^2} dt' \\
    &\lesssim \sum_{|q'-q| \leq4}  2^{q'} \int_0^t e^{2\mathcal{R}t'}\Vert \pa_y u_\phi(t') \Vert_{\mathcal{B}^\f12} \Vert \Delta_{q'}^h u_\phi(t') \Vert_{L^2}\Vert \Delta_q^h (\pa_tu)_\phi(t') \Vert_{L^2} dt' \\
    &\lesssim \sum_{|q'-q| \leq4} 2^{q'} \left( \int_0^t \Vert \pa_y u_\phi(t') \Vert_{\mathcal{B}^\f12} e^{2\mathcal{R}t'}\Vert \Delta_{q'}^h u_\phi(t') \Vert_{L^2}^2 dt' \right)^\frac{1}{2} \\
     & \ \ \  \ \ \ \ \ \ \ \ \ \ \  \times \left(\int_0^t \Vert \pa_y u_\phi(t') \Vert_{\mathcal{B}^\f12} e^{2\mathcal{R}t'}\Vert \Delta_q^h (\pa_tu)_\phi(t') \Vert_{L^2}^2 dt' \right)^\frac{1}{2},
\end{align*}
which leads to 

 \begin{align} \label{vtkbis1}
\int_0^t |\psca{ \Delta_q^h (u\pa_x u)_{\phi}, e^{2\mathcal{R}t'}\Delta_q^h (\pa_tu)_{\phi}}_{L^2}| dt' \lesssim d_q^2 2^{-2qs} \Vert e^{\mathcal{R}t} u_\phi \Vert_{\tilde{L}^2_{t,\dot{\theta}(t)}(\mathcal{B}^{s+\frac{1}{2}})} \Vert e^{\mathcal{R}t} (\pa_tu)_\phi \Vert_{\tilde{L}^2_{t,\dot{\theta}(t)}(\mathcal{B}^{s+\frac{1}{2}})}.
\end{align}

While it follows from the proof of Lemma $\ref{lem:Apa_yBCbis1}$ that
\beno
\int_0^t |\psca{ \Delta_q^h (T^h_{\pa_yu}v + R^h(v,\pa_yu))_{\phi},e^{2\mathcal{R}t'} \Delta_q^h (\pa_tu)_{\phi}}_{L^2}| dt'  \\ \lesssim d_q^2 2^{-2qs} \Vert e^{\mathcal{R}t} u_\phi \Vert_{\tilde{L}^2_{t,\dot{\theta}(t)}(\mathcal{B}^{s+\frac{1}{2}})} \Vert e^{\mathcal{R}t} (\pa_tu)_\phi \Vert_{\tilde{L}^2_{t,\dot{\theta}(t)}(\mathcal{B}^{s+\frac{1}{2}})},
\eeno
so we have yet to determine the estimate of $\int_0^t |\psca{ \Delta_q^h (T^h_{v}\pa_yu)_{\phi}, e^{2\mathcal{R}t'}\Delta_q^h (\pa_t u)_{\phi}}_{L^2}| dt'$, we have 
$$ \Vert \Delta_q^h v_\phi(t) \Vert_{L^\infty} \lesssim d_q(u_\phi) 2^\frac{q}{2} \Vert u_\phi(t) \Vert_{\mathcal{B}^{\f32}}^\frac{1}{2} \Vert \pa_y u_\phi(t) \Vert_{\mathcal{B}^\f12}^\frac{1}{2},$$so that
$$ \Vert S_{q'-1}^h v_\phi(t') \Vert_{L^\infty} \lesssim 2^\frac{q}{2} \Vert u_\phi(t) \Vert_{\mathcal{B}^{\f32}}^\frac{1}{2} \Vert \pa_y u_\phi(t) \Vert_{\mathcal{B}^\f12}^\frac{1}{2}, $$
which implies that
\begin{align*}
    \int_0^t &|\psca{ \Delta_q^h (T^h_{v}\pa_yu)_{\phi}, e^{2\mathcal{R}t'}\Delta_q^h (\pa_tu)_{\phi}}_{L^2}| dt' \\
    &\lesssim \sum_{|q'-q|\leq4} \int_0^t e^{2\mathcal{R}t'}\Vert S_{q'-1}^h v_\phi(t') \Vert_{L^\infty} \Vert \Delta_{q'}^h \pa_y u_\phi(t') \Vert_{L^2} \Vert \Delta_q^h (\pa_tu)_\phi(t') \Vert_{L^2}dt' \\
    &\lesssim \sum_{|q'-q|\leq4} 2^\frac{q'}{2} \Vert u_\phi \Vert_{L^\infty_t(\mathcal{B}^{\f32})}^\frac{1}{2} \Vert e^{\mathcal{R}t} \Delta_{q'}^h \pa_y u_\phi \Vert_{L^2_t(L^2)} \left( \int_0^t \Vert  \pa_y u_\phi(t') \Vert_{\mathcal{B}^\f12} \Vert \Delta_q^h e^{\mathcal{R}t} (\pa_tu)_\phi(t') \Vert_{L^2}^2dt' \right)^\frac{1}{2} \\
    &\lesssim d_q^2 2^{-2qs} \Vert u_\phi \Vert_{\tilde{L}^\infty_t(\mathcal{B}^{\f32})}^\frac{1}{2} \Vert e^{\mathcal{R}t}\pa_y u_\phi \Vert_{\tilde{L}^2_t(\mathcal{B}^s)} \Vert e^{\mathcal{R}t}(\pa_tu)_\phi \Vert_{\tilde{L}^2_{t,\dot{\theta}(t)}(\mathcal{B}^{s+\frac{1}{2}})}.
\end{align*}
By summing all the terms we obtain
\begin{align} \label{4.17bis1}
\int_0^t &|\psca{ \Delta_q^h (v\pa_yu)_{\phi}, \Delta_q^h (\pa_tu)_{\phi}}_{L^2}| dt' \lesssim d_q^2 2^{-2qs} \Vert e^{\mathcal{R}t} (\pa_tu)_\phi \Vert_{\tilde{L}^2_{t,\dot{\theta}(t)}(\mathcal{B}^{s+\frac{1}{2}})} \nonumber \\  &\times \left( \Vert e^{\mathcal{R}t} u_\phi \Vert_{\tilde{L}^2_{t,\dot{\theta}(t)}(\mathcal{B}^{s+\frac{1}{2}})} +  \Vert u_\phi \Vert_{\tilde{L}^\infty_t(\mathcal{B}^{\f32})}^\frac{1}{2} \Vert e^{\mathcal{R}t} \pa_y u_\phi \Vert_{\tilde{L}^2_t(\mathcal{B}^s)} \right).
\end{align}

Along the same way we can obtain 

\ben \label{esti:111bis1}
\int_0^t |\psca{ \Delta_q^h (u\pa_x u)_{\phi}, e^{2\mathcal{R}t'}\Delta_q^h u_{\phi}}_{L^2}| dt' \lesssim d_q^2 2^{-2qs} \Vert e^{\mathcal{R}t} u_\phi \Vert_{\tilde{L}^2_{t,\dot{\theta}(t)}(\mathcal{B}^{s+\frac{1}{2}})}^2.
\een
and
\begin{align} \label{esti:222bis1}
\int_0^t &|\psca{ \Delta_q^h (v\pa_yu)_{\phi},e^{2\mathcal{R}t'} \Delta_q^h u_{\phi}}_{L^2}| dt' \lesssim d_q^2 2^{-2qs} \Vert e^{\mathcal{R}t} u_\phi \Vert_{\tilde{L}^2_{t,\dot{\theta}(t)}(\mathcal{B}^{s+\frac{1}{2}})} \nonumber \\  &\times \left( \Vert e^{\mathcal{R}t} u_\phi \Vert_{\tilde{L}^2_{t,\dot{\theta}(t)}(\mathcal{B}^{s+\frac{1}{2}})} +  \Vert u_\phi \Vert_{\tilde{L}^\infty_t(\mathcal{B}^{\f32})}^\frac{1}{2} \Vert e^{\mathcal{R}t}\pa_y u_\phi \Vert_{\tilde{L}^2_t(\mathcal{B}^s)} \right).
\end{align}
By virtue of $\eqref{vtkbis1}$, $\eqref{4.17bis1}$ $(\ref{esti:111bis1})$ and $(\ref{esti:222bis1})$, we deduce from $(\ref{S3eq2bis1})$ and \eqref{S3eq1bis3bis1} that for $t<T^\star_a$

\begin{multline} \label{S4eq6bis1}
	 \Big( \frac{1}{2}\|e^{\mathcal{R} t}(u + \pa_tu)_\phi \|_{\tilde{L}_t^\infty(\mathcal{B}^{s})} +\frac{1}{2}\|e^{\mathcal{R} t}(\pa_tu)_\phi \|_{\tilde{L}_t^\infty(\mathcal{B}^{s})} + \|e^{\mathcal{R} t} \pa_yu_\phi \|_{\tilde{L}_t^\infty(\mathcal{B}^{s})} \Big)+ \f12\|e^{\mathcal{R} t}(\pa_tu)_\phi\|_{\tilde{L}_t^2(\mathcal{B}^{s})} \\ +\sqrt{k\lambda} \|e^{\mathcal{R} t} u_\phi\|_{\tilde{L}^2_{t,\dot{\theta}(t)}(\mathcal{B}^{s+\frac{1}{2}})} + \f12\|e^{\mathcal{R} t}\pa_yu_\phi\|_{\tilde{L}_t^2(\mathcal{B}^{s})} +  \sqrt{\lambda}  \|e^{\mathcal{R} t}(\pa_t u)_\phi\|_{\tilde{L}^2_{t,\dot{\theta}(t)}(\mathcal{B}^{s+\frac{1}{2}})} +  \sqrt{\lambda} \|e^{\mathcal{R} t}\pa_y u_\phi\|_{\tilde{L}^2_{t,\dot{\theta}(t)}(\mathcal{B}^{s+\frac{1}{2}})} 
\\	\leq C \Big(\|e^{a|D_x|}\pa_yu_0 \|_{\mathcal{B}^{s}} +\|e^{a|D_x|}(u_0+u_1) \|_{\mathcal{B}^{s}} +  \|e^{a|D_x|}u_1 \|_{\mathcal{B}^{s}} \Big)  \\ + \sqrt{2}C \Vert e^{\mathcal{R} t} u_{\phi} \Vert_{\tilde{L}^2_{t,\dot{\theta}(t)}(\mathcal{B}^{s+\frac{1}{2}})} + \sqrt{2}C\Vert e^{\mathcal{R} t} u_{\phi} \Vert_{\tilde{L}^2_{t,\dot{\theta}(t)}(\mathcal{B}^{s+\frac{1}{2}})}^\f12 \Vert e^{\mathcal{R} t} (\pa_tu)_{\phi} \Vert_{\tilde{L}^2_{t,\dot{\theta}(t)}(\mathcal{B}^{s+\frac{1}{2}})}^\f12 \\ 
+ \sqrt{2}C\Vert u_\phi \Vert_{\tilde{L}^2_{t,\dot{\theta}}(\mathcal{B}^{s+\frac{1}{2}})}^\f12\Vert \pa_yu_\phi \Vert_{\tilde{L}^\infty_t(\mathcal{B}^{\f32})}^\frac{1}{4} \Vert \pa_y u_\phi \Vert_{\tilde{L}^2_t(\mathcal{B}^s)}^\f12 + \sqrt{2}C\Vert (\pa_tu)_\phi \Vert_{\tilde{L}^2_{t,\dot{\theta}}(\mathcal{B}^{s+\frac{1}{2}})}^\f12\Vert \pa_yu_\phi \Vert_{\tilde{L}^\infty_t(\mathcal{B}^{\f32})}^\frac{1}{4} \Vert \pa_y u_\phi \Vert_{\tilde{L}^2_t(\mathcal{B}^s)}^\f12.
\end{multline}
Applying Young's inequality yields
\begin{align*}
 C\Vert u_\phi \Vert_{\tilde{L}^\infty_t(\mathcal{B}^{\f32})}^\frac{1}{4} \Vert &e^{\mathcal{R}t} \pa_y u_\phi \Vert_{\tilde{L}^2_t(\mathcal{B}^s)}^\f12 \Vert e^{\mathcal{R}t} u_\phi \Vert_{\tilde{L}^2_{t,\dot{\theta}(t)}(\mathcal{B}^{s+\frac{1}{2}})}^\f12\\
&\leq C \Vert u_\phi \Vert_{\tilde{L}^\infty_t(\mathcal{B}^{\f32})}^\f12 \Vert e^{\mathcal{R}t} u_\phi \Vert_{\tilde{L}^2_{t,\dot{\theta}(t)}(\mathcal{B}^{s+\frac{1}{2}})} + \frac{1}{10} \Vert e^{\mathcal{R}t} \pa_y u_\phi \Vert_{\tilde{L}^2_t(\mathcal{B}^s)} .
\end{align*}

Then we achieve
\begin{multline} \label{S4eq7bis1}
	 \Big( \frac{1}{2}\|e^{\mathcal{R} t}(u + \pa_tu)_\phi \|_{\tilde{L}_t^\infty(\mathcal{B}^{s})}+ \frac{1}{2}\|e^{\mathcal{R} t}( \pa_tu)_\phi \|_{\tilde{L}_t^\infty(\mathcal{B}^{s})} + \|e^{\mathcal{R} t} \pa_yu_\phi \|_{\tilde{L}_t^\infty(\mathcal{B}^{s})} \Big)+ \f12\|e^{\mathcal{R} t}(\pa_tu)_\phi\|_{\tilde{L}_t^2(\mathcal{B}^{s})} \\ +\sqrt{k\lambda} \|e^{\mathcal{R} t} u_\phi\|_{\tilde{L}^2_{t,\dot{\theta}(t)}(\mathcal{B}^{s+\frac{1}{2}})} + c\|e^{\mathcal{R} t}\pa_yu_\phi\|_{\tilde{L}_t^2(\mathcal{B}^{s})} +  \sqrt{\lambda}  \|e^{\mathcal{R} t}(\pa_t u)_\phi\|_{\tilde{L}^2_{t,\dot{\theta}(t)}(\mathcal{B}^{s+\frac{1}{2}})} +  \sqrt{\lambda} \|e^{\mathcal{R} t}\pa_y u_\phi\|_{\tilde{L}^2_{t,\dot{\theta}(t)}(\mathcal{B}^{s+\frac{1}{2}})}
\\	\leq C\|e^{a|D_x|}\pa_yu_0 \|_{\mathcal{B}^{s}} +C\|e^{a|D_x|}(u_0+u_1) \|_{\mathcal{B}^{s}} +C\|e^{a|D_x|}u_1 \|_{\mathcal{B}^{s}} \\ + C (2 + \Vert \pa_yu_\phi \Vert_{\tilde{L}^\infty_t(\mathcal{B}^{\f32})}^\f12 )\Vert e^{\mathcal{R} t} (u,\pa_tu)_{\phi} \Vert_{\tilde{L}^2_{t,\dot{\theta}(t)}(\mathcal{B}^{s+\frac{1}{2}})}.
\end{multline}
Therefore if we take
\begin{align}
 &\lambda \geq 4C^2(2+ \Vert \pa_y u_\phi \Vert_{\tilde{L}^\infty_t(\mathcal{B}^{\f32})}) \nonumber \\
 &k\lambda \geq 8C^2(2+ \Vert \pa_y u_\phi \Vert_{\tilde{L}^\infty_t(\mathcal{B}^{\f32})}),
\end{align}
we obtain
\begin{align*}
     \mathcal{E}_{s,\frac{\lambda}{4}}(u)(t) \leq C\Big( \|e^{a|D_x|}\pa_yu_0 \|_{\mathcal{B}^{s}} +\|e^{a|D_x|}(u_0+u_1) \|_{\mathcal{B}^{s}} + \|e^{a|D_x|}u_1 \|_{\mathcal{B}^{s}} \Big),
\end{align*}
where 
\begin{align*}
\mathcal{E}_{s,\frac{\lambda}{4}}(u)(t) &= \Big( \frac{1}{2}\|e^{\mathcal{R} t}(u + \pa_tu)_\phi \|_{\tilde{L}_t^\infty(\mathcal{B}^{s})} + \|e^{\mathcal{R} t} \pa_yu_\phi \|_{\tilde{L}_t^\infty(\mathcal{B}^{s})} \Big) \\
&+ \f12\|e^{\mathcal{R} t}(\pa_tu)_\phi\|_{\tilde{L}_t^2(\mathcal{B}^{s})} + c\|e^{\mathcal{R} t}\pa_yu_\phi\|_{\tilde{L}_t^2(\mathcal{B}^{s})} \\
& + \frac{\sqrt{k\lambda}}{2} \|e^{\mathcal{R} t} u_\phi\|_{\tilde{L}^2_{t,\dot{\theta}(t)}(\mathcal{B}^{s+\frac{1}{2}})} +  \frac{\sqrt{\lambda}}{2}  \|e^{\mathcal{R} t}(\pa_t u)_\phi\|_{\tilde{L}^2_{t,\dot{\theta}(t)}(\mathcal{B}^{s+\frac{1}{2}})}+  \frac{\sqrt{\lambda}}{2} \|e^{\mathcal{R} t}\pa_y u_\phi\|_{\tilde{L}^2_{t,\dot{\theta}(t)}(\mathcal{B}^{s+\frac{1}{2}})}
\end{align*}
which in particular implies that under the condition \eqref{eq:condini}, there hold 
$$ \Vert \pa_y u_\phi \Vert_{\tilde{L}^\infty_t(\mathcal{B}^{\f32})} \leq \sqrt{C}\Big(\|e^{a|D_x|}\pa_yu_0 \|_{\mathcal{B}^{\f32}} +\|e^{a|D_x|}(u_0+u_1) \|_{\mathcal{B}^{\f32}} +\|e^{a|D_x|}u_1 \|_{\mathcal{B}^{\f32}}\Big). $$
Then by taking $\lambda = C(2 + \Vert e^{a|D_x|}  \pa_y u_0\Vert_{\mathcal{B}^{\f32}} +  \Vert e^{a|D_x|}  (u_0 + u_1)\Vert_{\mathcal{B}^{\f32}} + \|e^{a|D_x|}u_1 \|_{\mathcal{B}^{\f32}} )$, therefore the condition of the proposition is satisfied and then the proposition is proved for any $t>0$.
\begin{prop}\label{prop:2bis1}
We assume that the condition \eqref{eq:condini} is satisfied, then for any $s>0$, there exist a small constant $C$, such that for $$\lambda = C(1 + \| e^{a|D_x|}\pa_y u_0\|_{\mathcal{B}^{\f32}} + \|e^{a|D_x|}(u_0+u_1) \|_{\mathcal{B}^{\f32}} + \|e^{a|D_x|}u_1 \|_{\mathcal{B}^{\f32}}),$$ we have 
\begin{multline} \label{Vorticity}
	 \Big( \frac{1}{2}\|e^{\mathcal{R} t}\pa_y (u + \pa_tu)_\phi \|_{\tilde{L}_t^\infty(\mathcal{B}^{s})} + \frac{1}{2}\|e^{\mathcal{R} t}\pa_y (\pa_tu)_\phi \|_{\tilde{L}_t^\infty(\mathcal{B}^{s})} + \|e^{\mathcal{R} t} \pa_y^2 u_\phi \|_{\tilde{L}_t^\infty(\mathcal{B}^{s})} \Big) \\ + \f12 \|e^{\mathcal{R} t}\pa_y(\pa_tu)_\phi\|_{\tilde{L}_t^2(\mathcal{B}^{s})} + \frac{1}{2}\|e^{\mathcal{R} t}\pa_y^2 u_\phi\|_{\tilde{L}_t^2(\mathcal{B}^{s})}  \leq C\|e^{a|D_x|}\pa_y^2u_0 \|_{\mathcal{B}^{s}} \\ +C\|e^{a|D_x|}\pa_y (u_0+u_1) \|_{\mathcal{B}^{s}} +C\|e^{a|D_x|}\pa_y u_1 \|_{\mathcal{B}^{s}} \\
	 \|e^{a|D_x|} u_0\|_{\mathcal{B}^{s+2}} + \|e^{a|D_x|} u_0\|_{\mathcal{B}^{s+1}} + \|e^{a|D_x|} u_1\|_{\mathcal{B}^{s+1}}, \text{ \ for \ } t<T^\star .
\end{multline}
\end{prop}
\begin{proof} of Proposition \ref{prop:2bis1}.
We start our proof by applying the partial derivative on $y$ ($\pa_y$) to \eqref{eq:HPNhydrolimit}, we obtain 
$$ \pa_t^2 \pa_y u + \pa_t \pa_yu + \pa_y(u\pa_xu) + \pa_y(v\pa_yu) - \pa_y^3 u + \pa_x\pa_yp = 0.$$ 

Due to the free divergence condition of $U$ (it mean that $\pa_xu+\pa_yv=0$), we have 
\begin{align*}
  \pa_y(u\pa_xu) + \pa_y(v\pa_yu) &= \pa_yu\pa_xu+ u\pa_x\pa_yu + \pa_yv\pa_yu + v\pa_y^2u \\
  &= \pa_yu\pa_xu+ u\pa_x\pa_yu - \pa_xu\pa_yu + v\pa_y^2u \\ &= u\pa_x\pa_yu + v\pa_y^2u,
\end{align*}
so our equation becomes
\begin{align}\label{eq:Hns22bis1}
    \pa_t^2 w + \pa_t w + u\pa_xw + v\pa_yw - \pa_y^2w + \pa_xq = 0
\end{align}
where $w = \pa_yu$ and $q = \pa_yp$. from which, using that $- \pa_yw + \pa_xp$ is vanishing on the boundary, we get, by using a similar derivation of \eqref{S3eq2bis1} and \eqref{S3eq1bis3bis1}, it mean that we do the scalar product of our equation with $\Delta_q^h (w + 2\pa_tw)_\phi $, so that
\begin{multline} \label{Vort:S4eq1bis1}
	 \frac{d}{dt}\Big( \|e^{\mathcal{R} t}\Delta_q^h (\pa_tw)_\phi \|_{L^2}^2 + \|e^{\mathcal{R} t}\Delta_q^h \pa_yw_\phi \|_{L^2}^2 \Big)+ 2\|e^{\mathcal{R} t}\Delta_q^h (\pa_tw)_\phi\|_{L^2}^2 \\ +  2\lambda \dot{\theta}(t')  \|e^{\mathcal{R} t}\Delta_q^h |D_x|^\frac{1}{2}(\pa_t w)_\phi\|_{L^2}^2 +2\lambda \dot{\theta}(t') \|e^{\mathcal{R} t}\Delta_q^h |D_x|^\frac{1}{2}\pa_y w_\phi\|_{L^2}^2 
\\	\leq  2 \Big( \abs{\psca{\Delta_q^h (u\pa_x w)_{\phi}, e^{2\mathcal{R} t}\Delta_q^h(\pa_tw)_{\phi}}}_{L^2} + \abs{\psca{\Delta_q^h ( v\pa_y w)_{\phi}, e^{2\mathcal{R} t}\Delta_q^h (\pa_tw)_{\phi}}}_{L^2}\Big) \\
+ 2\mathcal{R}\Big( \|e^{\mathcal{R} t}\Delta_q^h (\pa_tw)_\phi \|_{L^2}^2 + \|e^{\mathcal{R} t}\Delta_q^h \pa_yw_\phi \|_{L^2}^2 \Big) + 2\abs{\int_{\mathbb{R}}\Delta_q^h \pa_xp_{\phi}\cdot e^{2\mathcal{R} t}\Delta_q^h (\pa_tw(t,x,1) - \pa_tw(t,x,0))_{\phi} dx},
\end{multline}
and 
\begin{align} \label{vORT:S4eq2bis1}
	 \frac{d}{dt} \mathcal{A}(w_\phi)(t)  &+ \frac{1}{2}\frac{d}{dt} \|e^{\mathcal{R}t}\Delta_q^h w_\phi\|_{L^2}^2 + \lambda \dot{\theta}(t) \|e^{\mathcal{R}t}\Delta_q^h |D_x|^\frac{1}{2} w_\phi \|_{L^2}^2 + \|e^{\mathcal{R}t}\Delta_q^h  \pa_yw_\phi \|_{L^2}^2 \\ &= -\psca{\Delta_q^h (u\pa_x w)_{\phi}, e^{2\mathcal{R}t}\Delta_q^h w_{\phi})}_{L^2} -\psca{\Delta_q^h (v\pa_y w)_{\phi}, e^{2\mathcal{R}t}\Delta_q^h w_{\phi})}_{L^2} + \mathcal{R}\Big( \|e^{\mathcal{R}t}\Delta_q^h w_\phi\|_{L^2}^2  \nonumber \\ &+ \int e^{2\mathcal{R}t}\Delta_q^h (\pa_t w)_\phi \Delta_q^h  w_\phi dX \Big) -  \int_{\mathbb{R}}\Delta_q^h \pa_xp_{\phi}\cdot e^{2\mathcal{R} t}\Delta_q^h (w(t,x,1)-w(t,x,0))_{\phi} dx\nonumber       
\end{align}
where 
\begin{align*}
\mathcal{A}(w_\phi)(t) &= \int e^{2\mathcal{R}t}\Delta_q^h (\pa_t w)_\phi \Delta_q^h  w_\phi dX - \int e^{2\mathcal{R}t}\Delta_q^h (\pa_t w)_\phi \Delta_q^h (\pa_tw)_\phi dX \\&+ 2\lambda \dot{\theta}(t)\int e^{2\mathcal{R}t} \Delta_q^h |D_x| (\pa_tw)_\phi \Delta_q^h w_\phi dX
\end{align*}
Now, We start to estimate the pressure term, for date we denote 
$$ K_q =  \int_{\mathbb{R}} \Delta_q^h\pa_xp_{\phi}\cdot \Delta_q^h (\pa_tw(t,x,1) - \pa_tw(t,x,0))_{\phi}dx . $$
In view of \eqref{pressureestimatebis1} and Lemma \ref{lem:1}, we write 
\begin{align*}
 K_q &=   \int_{\mathbb{R}} \Delta_q^h\pa_xp_{\phi}\cdot \Delta_q^h (\pa_tw(t,x,1) - \pa_tw(t,x,0))_{\phi}dx \\
 & =   \int_{\mathbb{R}} \Delta_q^h \Big((w_\phi(t,x,1)-w_\phi(t,x,0)\Big) \cdot  \Delta_q^h \Big(\pa_tw(t,x,1) - \pa_tw(t,x,0)\Big)_{\phi} dx  \\
 & - \int_{\mathbb{R}} \Delta_q^h \pa_x (\int_0^1 (u)^2_\phi dy)  \cdot  \Delta_q^h \Big(\pa_tw(t,x,1) - \pa_tw(t,x,0)\Big)_{\phi} dx \\
 & = \frac{1}{2}\frac{d}{dt}\int_{\mathbb{R}}  \Delta_q^h \Big((w(t,x,1)-w(t,x,0)\Big)_\phi^2 dx + \lambda \dot{\theta}(t)\int_{\mathbb{R}}   \Delta_q^h|D_x| \Big((w(t,x,1)-w(t,x,0)\Big)_\phi^2 dx \\
 & -  \int_{\mathbb{R}} \Delta_q^h \pa_x (\int_0^1 (u)^2_\phi dy)  \cdot \Delta_q^h \Big(\pa_tw(t,x,1) - \pa_tw(t,x,0)\Big)_{\phi} dx \\ &-\mathcal{R}\int_{\mathbb{R}}  \Delta_q^h \Big((w(t,x,1)-w(t,x,0)\Big)_\phi^2 dx  
\end{align*}

If we use Lemma \ref{lem:1} to the following quantity $\int_{\mathbb{R}} \Delta_q^h\pa_x (\int_0^1 (u)^2 dy)  \cdot e^{2\phi} \Delta_q^h \Big(\pa_tw(t,x,1) - \pa_tw(t,x,0)\Big)_{\phi} dx $, we can obtain that  
\begin{align*}
    &\int_{\mathbb{R}} \Delta_q^h e^{\phi} \pa_x (\int_0^1 (u)^2 dy)  \cdot  \Delta_q^h \Big(\pa_t w(t,x,1) - \pa_t w(t,x,0)\Big)_{\phi} dx \\& = \frac{d}{dt} \int_{\mathbb{R}} \Delta_q^h e^{\phi}\pa_x (\int_0^1 (u)^2 dy)  \cdot e^{\phi} \Delta_q^h \Big(w(t,x,1) - w(t,x,0)\Big) dx \\
    & - 2\int_{\mathbb{R}} \Delta_q^h e^{\phi} \pa_x (\int_0^1 (u\pa_tu) dy)  \cdot e^{\phi} \Delta_q^h \Big(w(t,x,1) - w(t,x,0)\Big) dx \\
    & + 2\lambda \dot{\theta}(t)\int_{\mathbb{R}} \Delta_q^h e^{\phi} |D_x|\pa_x(\int_0^1 (u)^2 dy)  \cdot e^{\phi} \Delta_q^h \Big(w(t,x,1) - w(t,x,0)\Big) dx .
\end{align*}

So we multiply our equation by $e^{2\mathcal{R}t}$, we achieve that 
\begin{align*}
    &|e^{2\mathcal{R}t}K_q(t)| \leq \frac{1}{2} \frac{d}{dt} \|\Delta_q^h e^{\mathcal{R}t} w_\phi\|_{L^\infty_v(L^2_h)}^2 + \lambda \dot{\theta}(t) \|\Delta_q^h e^{\mathcal{R}t} |D_x|^\frac{1}{2} w_\phi\|_{L^\infty_v(L^2_h)}^2 \\
    & + \frac{d}{dt} |\int_{\mathbb{R}} e^{2\mathcal{R}t}\Delta_q^h e^{\phi} \pa_x (\int_0^1 (u)^2 dy)  \cdot e^{\phi} \Delta_q^h \Big(w(t,x,1) - w(t,x,0)\Big) dx| \\
    & + 2\int_{\mathbb{R}} |e^{2\mathcal{R}t}\Delta_q^h e^{\phi}  \pa_x (\int_0^1 (u\pa_tu) dy)  \cdot e^{\phi} \Delta_q^h \Big(w(t,x,1) - w(t,x,0)\Big)| dx \\
    & + 2\lambda \dot{\theta}(t)\int_{\mathbb{R}} |e^{2\mathcal{R}t}\Delta_q^h e^{\phi}  |D_x| \pa_x(\int_0^1 (u)^2 dy)  \cdot e^{\phi} \Delta_q^h \Big(w(t,x,1) - w(t,x,0)\Big)| dx \\
    &+ 2\mathcal{R}|\int_{\mathbb{R}} e^{2\mathcal{R}t}\Delta_q^h e^{\phi} \pa_x (\int_0^1 (u)^2 dy)  \cdot e^{\phi} \Delta_q^h \Big(w(t,x,1) - w(t,x,0)\Big) dx|\\
    & + \mathcal{R}|\int_{\mathbb{R}}  \Delta_q^h \Big((w(t,x,1)-w(t,x,0)\Big)_\phi^2 dx|.
\end{align*}

from which, we infer 

\begin{align*}
    &|e^{2\mathcal{R}t}K_q(t)| \leq \frac{1}{2} \frac{d}{dt} \Big(\f12 \|\Delta_q^h e^{\mathcal{R}t} w_\phi\|_{L^\infty_v(L^2_h)}^2 +  \|\Delta_q^h e^{\mathcal{R}t} \pa_x (u^2)_\phi\|_{L^1_v(L^2_h)}^2 \Big) \\ &+  \Big({\lambda \dot{\theta}(t)}\|\Delta_q^h e^{\mathcal{R}t} |D_x|^\frac{1}{2} w_\phi\|_{L^\infty_v(L^2_h)}^2 + {\lambda \dot{\theta}(t)} \|\Delta_q^h e^{\mathcal{R}t} |D_x|^\frac{1}{2}\pa_x(u^2)_\phi\|_{L^1_v(L^2_h)}^2\Big) \\
   & + C \Big( \|\Delta_q^h e^{\mathcal{R}t} w_\phi\|_{L^\infty_v(L^2_h)}^2 + \|\Delta_q^h e^{\mathcal{R}t} (u\pa_tu)_\phi\|_{L^1_v(L^2_h)}^2 + \lambda \dot{\theta}(t)\|\Delta_q^h e^{\mathcal{R}t} |D_x|^\frac{1}{2} w_\phi\|_{L^\infty_v(L^2_h)}^2 \Big) \\
   &+ \mathcal{R} \Big( \f32  \|\Delta_q^h e^{\mathcal{R}t} w_\phi\|_{L^\infty_v(L^2_h)}^2 + \|\Delta_q^h e^{\mathcal{R}t}\pa_x(u^2)_\phi\|_{L^1_v(L^2_h)}^2 \Big).
\end{align*}
By applying Bony decomposition, we have for any $s>0$
\begin{align*}
\|e^{\mathcal{R}t} \Delta_q^h\pa_x(u^2)_\phi \|_{L^2_t(L^1_v(L^2_h))} & \lesssim \|e^{\mathcal{R}t} \Delta_q^h \pa_x(u^2)_\phi \|_{L^2_t(L^2)} \\
&\lesssim d_q 2^{-qs} \| u_\phi \|_{L^2_t(L^\infty)} \|e^{\mathcal{R}t}  u_\phi\|_{\tilde{L}^\infty_t(\mathcal{B}^{s+1})} \\
&\lesssim  d_q 2^{-qs} \| \pa_y u_\phi \|_{L^2_t(\mathcal{B}^\f12)} \|e^{\mathcal{R}t} u_\phi\|_{\tilde{L}^\infty_t(\mathcal{B}^{s+1})} 
\end{align*}
and 
\begin{align*}
\|e^{\mathcal{R}t} &\Delta_q^h\pa_x(u\pa_tu)_\phi \|_{L^2_t(L^1_v(L^2_h))} \lesssim \|e^{\mathcal{R}t} \Delta_q^h \pa_x(u\pa_tu)_\phi \|_{L^2_t(L^2)} \\
&\lesssim d_q 2^{-qs} \Big( \| u_\phi \|_{L^2_t(L^\infty)} \|e^{\mathcal{R}t}  \pa_t u_\phi\|_{\tilde{L}^\infty_t(\mathcal{B}^{s+1})} + \| \pa_t u_\phi \|_{L^2_t(L^\infty)} \|e^{\mathbb{R}t} u_\phi\|_{\tilde{L}^\infty_t(\mathcal{B}^{s+1})} \Big) \\
&\lesssim  d_q 2^{-qs} \Big( \| \pa_y u_\phi \|_{L^2_t(\mathcal{B}^\f12)} \|e^{\mathcal{R}t} \pa_t u_\phi\|_{\tilde{L}^\infty_t(\mathcal{B}^{s+1})} + \| \pa_t\pa_y u_\phi \|_{L^2_t(\mathcal{B}^\f12)} \|e^{\mathcal{R}t} u_\phi\|_{\tilde{L}^\infty_t(\mathcal{B}^{s+1})} \Big).
\end{align*}

While notice that 
$$ \int_0^1 \Delta_q^h \pa_y u_\phi(t,x,y) dy = 0, $$
then for any fixed $(t,x)\in \mathbb{R}^+ \times \mathbb{R}$, there exist $Y_0^q(t,x)$ so that $\Delta_q^h \pa_y u_\phi(t,x,Y_0^q(t,x))= 0.$ So we have 
$$ (\Delta_q^h \pa_y u_\phi(t,x,y))^2 \leq 2\| \Delta_q^h \pa_y u_\phi \|_{L^2_v} \| \Delta_q^h \pa_y^2 u_\phi \|_{L^2_v}, $$
which implies that 
$$ \|\Delta_q^h \pa_y u_\phi(t,x,y)) \|_{L^\infty_v(L^2_h)} \leq 2\| \Delta_q^h \pa_y u_\phi \|_{L^2} \| \Delta_q^h \pa_y^2 u_\phi \|_{L^2}.$$

As a result, it comes out
\begin{align}\label{Pressureestiwpa_tubis1}
    &\int_0^t |e^{2\mathcal{R}t}K_q(t')|dt' \leq d_q^2 2^{-2qs} \Big(C\mathtt{E}(0) + C\| e^{\mathcal{R}t} w_\phi\|_{L^\infty_t(\mathcal{B}^s)}^2 + \frac{1}{4}\| e^{\mathcal{R}t} \pa_y w_\phi\|_{L^\infty_t(\mathcal{B}^s)}^2 \\ & +  \frac{1}{2}\|\pa_y u_\phi \|_{L^\infty_t(\mathcal{B}^\f12)}^2\| e^{\mathcal{R}t} u_\phi\|_{L^\infty_t(\mathcal{B}^{s+1})}^2 + \lambda\|\pa_y u_\phi \|_{L^\infty_t(\mathcal{B}^{\f32})}^2\| e^{\mathcal{R}t} \pa_y u_\phi\|_{\tilde{L}^2_{t,\dot{\theta}(t)}(\mathcal{B}^{s+\frac{1}{2}})}^2  + C \| e^{\mathcal{R}t} w_\phi\|_{L^2_t(\mathcal{B}^s)}^2 \nonumber  \\
   & +\frac{1}{4}\| e^{\mathcal{R}t} \pa_yw_\phi\|_{L^2_t(\mathcal{B}^s)}^2 + \| \pa_y u_\phi \|_{\tilde{L}^2_t(\mathcal{B}^\f12)}^2 \|e^{\mathcal{R}t} \pa_t u_\phi\|_{\tilde{L}^\infty_t(\mathcal{B}^{s+1})}^2 + \| \pa_t \pa_y u_\phi \|_{\tilde{L}^2_t(\mathcal{B}^\f12)}^2 \|e^{\mathcal{R}t} u_\phi\|_{\tilde{L}^\infty_t(\mathcal{B}^{s+1})}^2\nonumber \\ &+ C\lambda \| e^{\mathcal{R}t} \pa_y u_\phi\|_{\tilde{L}^2_{t,\dot{\theta}(t)}(\mathcal{B}^{s+\frac{1}{2}})}^2 +\frac{\lambda}{4}\| e^{\mathcal{R}t} \pa_y^2 u_\phi\|_{\tilde{L}^2_{t,\dot{\theta}(t)}(\mathcal{B}^{s+\frac{1}{2}})}^2 \nonumber \\
   &\frac{3\mathcal{R}}{8} \| e^{\mathcal{R}t} w_\phi\|_{\tilde{L}^2_{t}(\mathcal{B}^{s})}^2 +\frac{3\mathcal{R}}{8}\| e^{\mathcal{R}t} \pa_y w_\phi\|_{\tilde{L}^2_{t}(\mathcal{B}^{s})}^2 + \mathcal{R} \| \pa_y u_\phi \|_{\tilde{L}^2_t(\mathcal{B}^\f12)}^2 \|e^{\mathcal{R}t} u_\phi\|_{\tilde{L}^\infty_t(\mathcal{B}^{s+1})}^2 \Big) \nonumber 
\end{align}
where 
\begin{align*}
\mathtt{E}(0) = \| e^{a|D_x|} w_0\|_{\mathcal{B}^s}^2 + \| e^{a|D_x|} \pa_y w_0\|_{\mathcal{B}^s}^2 + \|e^{a|D_x|}\pa_y u_0 \|_{\mathcal{B}^\f12}^2\| e^{a|D_x|} u_0\|_{\mathcal{B}^{s+1}}^2
\end{align*}
Along the same way we obtain 
\begin{align}\label{Pressureesti:wbis1}
    &\int_0^t |e^{2\mathcal{R}t'}\int_{\mathbb{R}} \Delta_q^h\pa_xp_{\phi}\cdot \Delta_q^h (w(t,x,1) - w(t,x,0))_{\phi}dx|dt' \\ &\leq  d_q^2 2^{-2qs} \Big(C\| e^{\mathcal{R}t} w_\phi\|_{\tilde{L}^2_t(\mathcal{B}^s)}^2  +  C\|u_\phi \|_{\tilde{L}^2_t(\mathcal{B}^\f12)}^2\| e^{\mathcal{R}t} u_\phi\|_{L^\infty_t(\mathcal{B}^{s+1})}^2
    +\frac{1}{4}\| e^{\mathcal{R}t} \pa_yw_\phi\|_{\tilde{L}^2_t(\mathcal{B}^s)}^2 \Big) \nonumber.
\end{align}

It follows from the proof of Lemma \ref{lem:ABCbis1}, for any $s>0$
\begin{align*} \int_0^t &\abs{\psca{\Delta_q^h(T^h_u \pa_xw + R^h(u,\pa_xw))_\phi,e^{\mathcal{R}t'}\Delta_q^h (\pa_tw)_\phi}_{L^2}} dt' \\ &\leq C\tilde{d}_q 2^{-2qs} \|e^{\mathcal{R}t} w_\phi \|_{\tilde{L}^2_{t,\dot{\theta}}(\mathcal{B}^{s+\frac{1}{2}})} \|e^{\mathcal{R}t} (\pa_tw)_\phi \|_{\tilde{L}^2_{t,\dot{\theta}}(\mathcal{B}^{s+\frac{1}{2}})} .
\end{align*}
While we deduce from Lemma \ref{lem:Bernsteinbis1} and Definition \ref{def:CLweightbis12}
\begin{align*}
    \int_0^t |&\psca{ \Delta_q^h (T_{\pa_x w}^hu)_{\phi}, \Delta_q^h (\pa_tw)_{\phi}}_{L^2}| dt' \\
    &\lesssim \sum_{|q'-q| \leq4} \int_0^t \Vert S_{q'-1}^h \pa_x w_\phi(t') \Vert_{L^\infty_h(L^2_v)} \Vert \Delta_{q'}^h u_\phi (t') \Vert_{L^2_h(L^\infty_v)} \Vert \Delta_q^h (\pa_tw)_\phi(t') \Vert_{L^2} dt' \\
    &\lesssim \sum_{|q'-q| \leq4}  2^{q'} \int_0^t \Vert w_\phi(t') \Vert_{\mathcal{B}^\f12} \Vert \Delta_{q'}^h \pa_y u_\phi(t') \Vert_{L^2}\Vert \Delta_q^h (\pa_tw)_\phi(t') \Vert_{L^2} dt' \\
    &\lesssim \sum_{|q'-q| \leq4} 2^{q'} \left( \int_0^t \Vert w_\phi(t') \Vert_{\mathcal{B}^\f12} \Vert \Delta_{q'}^h w_\phi(t') \Vert_{L^2}^2 dt' \right)^\frac{1}{2} \\
     & \ \ \  \ \ \ \ \ \ \ \ \ \ \  \times \left(\int_0^t \Vert w_\phi(t') \Vert_{\mathcal{B}^\f12} \Vert \Delta_q^h (\pa_tw)_\phi(t') \Vert_{L^2} dt' \right)^\frac{1}{2} \\
     &\leq C d_q^2 2^{-2qs} \|e^{\mathcal{R}t} w_\phi \|_{\tilde{L}^2_{t,\dot{\theta}}(\mathcal{B}^{s+\frac{1}{2}})} \|e^{\mathcal{R}t} (\pa_tw)_\phi \|_{\tilde{L}^2_{t,\dot{\theta}}(\mathcal{B}^{s+\frac{1}{2}})}.
\end{align*}
Then we conclude for any $s>0$
\begin{align}\label{estim:upa_xwpa_twbis1}
\int_0^t |&\psca{ \Delta_q^h (u\pa_x w)_{\phi}, e^{\mathcal{R}t'}\Delta_q^h (\pa_tw)_{\phi}}_{L^2}| dt' \\
&\leq C d_q^2 2^{-2qs} \|e^{\mathcal{R}t} w_\phi \|_{\tilde{L}^2_{t,\dot{\theta}}(\mathcal{B}^{s+\frac{1}{2}})} \|e^{\mathcal{R}t} (\pa_tw)_\phi \|_{\tilde{L}^2_{t,\dot{\theta}}(\mathcal{B}^{s+\frac{1}{2}})}. \nonumber
\end{align}
In the same way we have 
\begin{align}\label{estim:upa_xwwbis1}
\int_0^t |&\psca{ \Delta_q^h (u\pa_x w)_{\phi}, e^{\mathcal{R}t'}\Delta_q^h w_{\phi}}_{L^2}| dt' \\
&\leq C d_q^2 2^{-2qs} \|e^{\mathcal{R}t} w_\phi \|_{\tilde{L}^2_{t,\dot{\theta}}(\mathcal{B}^{s+\frac{1}{2}})}^2. \nonumber
\end{align}

On the other hand, we deduce from Lemma \ref{lem:Bernsteinbis1} that for any $s>0$
\begin{align*}
 \int_0^t &|\psca{ \Delta_q^h (T^h_{v}\pa_yw)_{\phi}, \Delta_q^h (\pa_tw)_{\phi}}_{L^2}| dt' \\
    &\lesssim \sum_{|q'-q|\leq4} \int_0^t \Vert S_{q'-1}^h v_\phi(t') \Vert_{L^\infty} \Vert \Delta_{q'}^h \pa_y w_\phi(t') \Vert_{L^2} \Vert \Delta_q^h (\pa_tw)_\phi(t') \Vert_{L^2}dt' \\
    &\lesssim \sum_{|q'-q|\leq4} 2^\frac{q'}{2} \Vert u_\phi \Vert_{L^\infty_t(\mathcal{B}^{\f32})}^\frac{1}{2} \Vert \Delta_{q'}^h \pa_y w_\phi \Vert_{L^2_t(L^2)} \left( \int_0^t \Vert \pa_y u_\phi(t(') \Vert_{\mathcal{B}^\f12} \Vert \Delta_q^h (\pa_tw)_\phi(t') \Vert_{L^2}dt' \right)^\frac{1}{2} \\
    &\leq C d_q^2 2^{-2qs} \Vert u_\phi \Vert_{\tilde{L}^\infty_t(\mathcal{B}^{\f32})}^\frac{1}{2} \Vert \pa_y w_\phi \Vert_{\tilde{L}^2_t(\mathcal{B}^s)} \Vert (\pa_tw)_\phi \Vert_{\tilde{L}^2_{t,\dot{\theta}(t)}(\mathcal{B}^{s+\frac{1}{2}})}.
\end{align*} 
In the same way we obtain that 
\begin{align*}
 \int_0^t &|\psca{ \Delta_q^h R^h(v,\pa_yw)_{\phi}, \Delta_q^h (\pa_tw)_{\phi}}_{L^2}| dt' \\
    &\leq C d_q^2 2^{-2qs} \Vert u_\phi \Vert_{\tilde{L}^\infty_t(\mathcal{B}^{\f32})}^\frac{1}{2} \Vert \pa_y w_\phi \Vert_{\tilde{L}^2_t(\mathcal{B}^s)} \Vert (\pa_tw)_\phi \Vert_{\tilde{L}^2_{t,\dot{\theta}(t)}(\mathcal{B}^{s+\frac{1}{2}})}.
\end{align*} 
Finally, we use Lemma \ref{lem:ABCbis1}, we find  
\begin{align*}
 \int_0^t |&\psca{ \Delta_q^h (T^h_{\pa_yw}v)_{\phi}, \Delta_q^h (\pa_tw)_{\phi}}_{L^2}| dt' \\&\lesssim \sum_{|q'-q|\leq4} \int_0^t \Vert S_{q'-1}^h \pa_yw_\phi(t') \Vert_{L^\infty_h(L^2_v)} \Vert \Delta_{q'}^h v_\phi(t') \Vert_{L^2_h(L^\infty_v)} \Vert \Delta_q^h (\pa_tw)_\phi(t') \Vert_{L^2}dt' \\
    &\lesssim \sum_{|q'-q|\leq4} \int_0^t \|\pa_y w_\phi\|_{\mathcal{B}^\f12} \Vert \Delta_{q'}^h \pa_x u_\phi \Vert_{L^2} \|\Delta_q^h (\pa_tw)_\phi\|_{L^2} dt'\\
    &\leq C d_q^2 2^{-2qs} \Vert \pa_y w_\phi \Vert_{\tilde{L}^2_t(\mathcal{B}^\f12)} \Vert  u_\phi \Vert_{L^\infty_t(\mathcal{B}^{s+1})} \| (\pa_tw)_\phi\|_{\tilde{L}^2t(\mathcal{B}^{s)}}.
\end{align*}
Then by summing we deduce that 
\begin{align}\label{eq:S4eq5bis1}
\int_0^t |&\psca{ \Delta_q^h (v\pa_yw)_{\phi}, \Delta_q^h (\pa_tw)_{\phi}}_{L^2}| dt' \leq C d_q^2 2^{-2qs} \Vert \pa_y w_\phi \Vert_{\tilde{L}^2_t(\mathcal{B}^\f12)} \Vert  u_\phi \Vert_{L^\infty_t(\mathcal{B}^{s+1})} \| (\pa_tw)_\phi\|_{\tilde{L}^2_t(\mathcal{B}^{s)}} \nonumber\\
& + C d_q^2 2^{-2qs} \Vert u_\phi \Vert_{\tilde{L}^\infty_t(\mathcal{B}^{\f32})}^\frac{1}{2} \Vert \pa_y w_\phi \Vert_{\tilde{L}^2_t(\mathcal{B}^s)} \Vert (\pa_tw)_\phi \Vert_{\tilde{L}^2_{t,\dot{\theta}(t)}(\mathcal{B}^{s+\frac{1}{2}})} 
\end{align}
Along the same way we can found that 
\begin{align*}
 \int_0^t &|\psca{ \Delta_q^h(T^h_v\pa_yw +  R^h(v,\pa_yw))_{\phi}, \Delta_q^h w_{\phi}}_{L^2}| dt' \\
    &\lesssim d_q^2 2^{-2qs} \Vert u_\phi \Vert_{\tilde{L}^\infty_t(\mathcal{B}^{\f32})}^\frac{1}{2} \Vert \pa_y w_\phi \Vert_{\tilde{L}^2_t(\mathcal{B}^s)} \Vert w_\phi \Vert_{\tilde{L}^2_{t,\dot{\theta}(t)}(\mathcal{B}^{s+\frac{1}{2}})}.
\end{align*} 
Then we still have to estimate $T^h_{\pa_yw}v$, so by integration by part we can obtain 
\begin{align*}
 \int_0^t |\psca{ \Delta_q^h (T^h_{\pa_yw}v)_{\phi}, \Delta_q^h w_{\phi}}_{L^2}&| dt' \leq \int_0^t |\psca{ \Delta_q^h (T^h_{w}\pa_y v)_{\phi}, \Delta_q^h w_{\phi}}_{L^2}| dt' \\
 & + \int_0^t |\psca{ \Delta_q^h (T^h_{w}v)_{\phi}, \Delta_q^h \pa_yw_{\phi}}_{L^2}| dt'.
\end{align*}
Due to the free divergence $\pa_xu + \pa_yv = 0$ we deduce 
\begin{align*}
 \int_0^t &|\psca{ \Delta_q^h (T^h_{w}\pa_yv)_{\phi}, \Delta_q^h w_{\phi}}_{L^2}| dt' \\
    &\lesssim \sum_{|q'-q|\leq4} \int_0^t \Vert S_{q'-1}^h w_\phi(t') \Vert_{L^\infty_h(L^2_v)} \Vert \Delta_{q'}^h \pa_xu_\phi(t') \Vert_{L^2_h(L^\infty_v)} \Vert \Delta_q^h w_\phi(t') \Vert_{L^2}dt' \\
    &\lesssim \sum_{|q'-q|\leq4} 2^{q'}\Vert \pa_y u_\phi \Vert_{\mathcal{B}^\f12} \Vert \Delta_{q'}^h w_\phi \Vert_{L^2} \Vert \Delta_q^h w_\phi(t') \Vert_{L^2}dt'  \\
    &\leq C d_q^2 2^{-2qs} \Vert w_\phi \Vert_{\tilde{L}^2_{t,\dot{\theta}(t)}(\mathcal{B}^{s+\frac{1}{2}})}^2.
\end{align*} 
While we observe that
\begin{align*}
 \int_0^t |&\psca{ \Delta_q^h (T^h_{w}v)_{\phi}, \Delta_q^h \pa_y w_{\phi}}_{L^2}| dt' \\&\lesssim \sum_{|q'-q|\leq4} \int_0^t \Vert S_{q'-1}^h w_\phi(t') \Vert_{L^\infty_h(L^2_v)} \Vert \Delta_{q'}^h v_\phi(t') \Vert_{L^2_h(L^\infty_v)} \Vert \Delta_q^h \pa_yw_\phi(t') \Vert_{L^2}dt' \\
    &\lesssim \sum_{|q'-q|\leq4} \int_0^t \Vert w_\phi \Vert_{\mathcal{B}^\f12} \Vert \Delta_{q'}^h \pa_x u_\phi \Vert_{L^2} \|\Delta_q^h \pa_yw_\phi\|_{L^2} dt'\\
    &\leq C d_q^2 2^{-2qs} \Vert w_\phi \Vert_{\tilde{L}^2_t(\mathcal{B}^\f12)}\Vert u_\phi \Vert_{\tilde{L}^\infty_t(\mathcal{B}^{s+1})} \| \pa_yw_\phi\|_{\tilde{L}^2_t(\mathcal{B}^s)}.
\end{align*}
By summarizing the above estimates, we obtain
\begin{align}\label{eq:S4eq6bis1}
\int_0^t |&\psca{ \Delta_q^h (v\pa_yw)_{\phi}, \Delta_q^h w_{\phi}}_{L^2}| dt' \leq C d_q^2 2^{-2qs} \Vert u_\phi \Vert_{\tilde{L}^\infty_t(\mathcal{B}^{\f32})}^\frac{1}{2} \Vert \pa_y w_\phi \Vert_{\tilde{L}^2_t(\mathcal{B}^s)} \Vert w_\phi \Vert_{\tilde{L}^2_{t,\dot{\theta}(t)}(\mathcal{B}^{s+\frac{1}{2}})} \\
&+  C d_q^2 2^{-2qs}\Big( \Vert w_\phi \Vert_{\tilde{L}^2_t(\mathcal{B}^\f12)}\Vert u_\phi \Vert_{\tilde{L}^\infty_t(\mathcal{B}^{s+1})} \| \pa_yw_\phi\|_{\tilde{L}^2_t(\mathcal{B}^s)} + \Vert w_\phi \Vert_{\tilde{L}^2_{t,\dot{\theta}(t)}(\mathcal{B}^{s+\frac{1}{2}})}^2 \Big) \nonumber
\end{align}

By inserting the resulting estimates \eqref{Pressureestiwpa_tubis1}-\eqref{eq:S4eq6bis1} in \eqref{Vort:S4eq1bis1} + \eqref{vORT:S4eq2bis1} and multiplying by $2^{2qs}$ for $s>0$, and then integrating over time, and summing with respect to $q \in \ZZ$, we find that for $t<T^\star$

\begin{multline} \label{S4eq6bis1bis1}
	 \Big( \frac{1}{2}\|e^{\mathcal{R} t}(w + \pa_tw)_\phi \|_{\tilde{L}_t^\infty(\mathcal{B}^{s})} +\frac{1}{2}\|e^{\mathcal{R} t}(\pa_tw)_\phi \|_{\tilde{L}_t^\infty(\mathcal{B}^{s})} + \|e^{\mathcal{R} t} \pa_yw_\phi \|_{\tilde{L}_t^\infty(\mathcal{B}^{s})} \Big)+ \f12\|e^{\mathcal{R} t}(\pa_tw)_\phi\|_{\tilde{L}_t^2(\mathcal{B}^{s})} \\ +\sqrt{k\lambda} \|e^{\mathcal{R} t} w_\phi\|_{\tilde{L}^2_{t,\dot{\theta}(t)}(\mathcal{B}^{s+\frac{1}{2}})} + \f12\|e^{\mathcal{R} t}\pa_yw_\phi\|_{\tilde{L}_t^2(\mathcal{B}^{s})} +  \sqrt{\lambda}  \|e^{\mathcal{R} t}(\pa_t w)_\phi\|_{\tilde{L}^2_{t,\dot{\theta}(t)}(\mathcal{B}^{s+\frac{1}{2}})} +  \sqrt{\lambda} \|e^{\mathcal{R} t}\pa_y w_\phi\|_{\tilde{L}^2_{t,\dot{\theta}(t)}(\mathcal{B}^{s+\frac{1}{2}})} 
\\	\leq C\Big(\mathcal{L}(0)+   \Vert e^{\mathcal{R} t} w_{\phi} \Vert_{\tilde{L}^2_{t,\dot{\theta}(t)}(\mathcal{B}^{s+\frac{1}{2}})} +  \Vert e^{\mathcal{R} t} w_{\phi} \Vert_{\tilde{L}^2_{t,\dot{\theta}(t)}(\mathcal{B}^{s+\frac{1}{2}})}^\f12 \Vert e^{\mathcal{R} t} (\pa_tw)_{\phi} \Vert_{\tilde{L}^2_{t,\dot{\theta}(t)}(\mathcal{B}^{s+\frac{1}{2}})}^\f12 \\ 
+ \Vert w_\phi \Vert_{\tilde{L}^2_{t,\dot{\theta}}(\mathcal{B}^{s+\frac{1}{2}})}^\f12\Vert \pa_yu_\phi \Vert_{\tilde{L}^\infty_t(\mathcal{B}^{\f32})}^\frac{1}{4} \Vert \pa_y w_\phi \Vert_{\tilde{L}^2_t(\mathcal{B}^s)}^\f12 + \Vert (\pa_tw)_\phi \Vert_{\tilde{L}^2_{t,\dot{\theta}}(\mathcal{B}^{s+\frac{1}{2}})}^\f12\Vert \pa_yu_\phi \Vert_{\tilde{L}^\infty_t(\mathcal{B}^{\f32})}^\frac{1}{4} \Vert \pa_y w_\phi \Vert_{\tilde{L}^2_t(\mathcal{B}^s)}^\f12 \\
 + \| \pa_tu_\phi \|_{\tilde{L}^2_t(\mathcal{B}^\f12)} \|e^{\mathcal{R}t} u_\phi\|_{\tilde{L}^\infty_t(\mathcal{B}^{s+1})} + \| \pa_y u_\phi \|_{\tilde{L}^2_t(\mathcal{B}^\f12)} \|e^{\mathcal{R}t} \pa_t u_\phi\|_{\tilde{L}^\infty_t(\mathcal{B}^{s+1})} + \| \pa_y u_\phi \|_{\tilde{L}^2_t(\mathcal{B}^\f12)} \|e^{\mathcal{R}t} u_\phi\|_{\tilde{L}^\infty_t(\mathcal{B}^{s+2})}\Big) \\
  {\sqrt{\lambda}} \big( C\|e^{\mathcal{R}t} w_\phi \|_{\tilde{L}^2_{t,\dot{\theta}(t)}(\mathcal{B}^{s+\frac{1}{2}})} + \f12\|e^{\mathcal{R}t} \pa_y w_\phi \|_{\tilde{L}^2_{t,\dot{\theta}(t)}(\mathcal{B}^{s+\frac{1}{2}})}\big) + C\sqrt{\lambda} \|\pa_y u_\phi\|_{\tilde{L}^\infty_t(\mathcal{B}^\f32)}  \|e^{\mathcal{R}t} w_\phi \|_{\tilde{L}^2_{t,\dot{\theta}(t)}(\mathcal{B}^{s+\frac{1}{2}})} \\
  + C\| e^{\mathcal{R}t} w_\phi\|_{\tilde{L}^2_t(\mathcal{B}^s)} + \frac{1}{8}\| e^{\mathcal{R}t} \pa_y w_\phi\|_{\tilde{L}^2_t(\mathcal{B}^s)} + C\| e^{\mathcal{R}t} w_\phi\|_{\tilde{L}^\infty_t(\mathcal{B}^s)} + \frac{1}{8}\| e^{\mathcal{R}t} \pa_y w_\phi\|_{\tilde{L}^\infty_t(\mathcal{B}^s)}
\end{multline}
where \begin{align*}
\mathcal{L}(0) = \mathtt{E}(0) + \|e^{a|D_x|}\pa_yw_0 \|_{\mathcal{B}^{s}} +\|e^{a|D_x|}(w_0+w_1) \|_{\mathcal{B}^{s}} + \|e^{a|D_x|}w_1 \|_{\mathcal{B}^{s}}.
\end{align*}
Applying Young's inequality yields
\begin{align*}
 C\Vert \pa_y u_\phi \Vert_{\tilde{L}^\infty_t(\mathcal{B}^{\f32})}^\frac{1}{4} \Vert &e^{\mathcal{R}t'} \pa_y w_\phi \Vert_{\tilde{L}^2_t(\mathcal{B}^s)}^\f12 \Vert e^{\mathcal{R}t'} w_\phi \Vert_{\tilde{L}^2_{t,\dot{\theta}(t)}(\mathcal{B}^{s+\frac{1}{2}})}^\f12\\
&\leq C \Vert \pa_yu_\phi \Vert_{\tilde{L}^\infty_t(\mathcal{B}^{\f32})}^\f12 \Vert e^{\mathcal{R}t'} w_\phi \Vert_{\tilde{L}^2_{t,\dot{\theta}(t)}(\mathcal{B}^{s+\frac{1}{2}})} + \frac{1}{8} \Vert e^{\mathcal{R}t'} \pa_y w_\phi \Vert_{\tilde{L}^2_t(\mathcal{B}^s)} .
\end{align*}

Then we achieve
\begin{multline} \label{S4eq7bis1bis1}
	 \Big( \frac{1}{2} \|e^{\mathcal{R} t}(w + \pa_tw)_\phi \|_{\tilde{L}_t^\infty(\mathcal{B}^{s})} +  \frac{1}{2} \|e^{\mathcal{R} t}(\pa_tw)_\phi \|_{\tilde{L}_t^\infty(\mathcal{B}^{s})} +\f34 \|e^{\mathcal{R} t} \pa_yw_\phi \|_{\tilde{L}_t^\infty(\mathcal{B}^{s})} \Big)+ \f12 \|e^{\mathcal{R} t}(\pa_tw)_\phi\|_{\tilde{L}_t^2(\mathcal{B}^{s})} \\ +\sqrt{k\lambda} \|e^{\mathcal{R} t} w_\phi\|_{\tilde{L}^2_{t,\dot{\theta}(t)}(\mathcal{B}^{s+\frac{1}{2}})} + \f14 \|e^{\mathcal{R} t}\pa_yw_\phi\|_{\tilde{L}_t^2(\mathcal{B}^{s})} +  \sqrt{\lambda}  \|e^{\mathcal{R} t}(\pa_t w)_\phi\|_{\tilde{L}^2_{t,\dot{\theta}(t)}(\mathcal{B}^{s+\frac{1}{2}})}  +  \frac{\sqrt{\lambda}}{2} \|e^{\mathcal{R} t}\pa_y w_\phi\|_{\tilde{L}^2_{t,\dot{\theta}(t)}(\mathcal{B}^{s+\frac{1}{2}})} \\ \leq \sqrt{C}\Big(\mathcal{L}(0) + \Big(\sqrt{2} + \Vert \pa_yu_\phi \Vert_{\tilde{L}^\infty_t(\mathcal{B}^{\f32})}^\f12 \Big) \Vert e^{\mathcal{R} t} w_{\phi} \Vert_{\tilde{L}^2_{t,\dot{\theta}(t)}(\mathcal{B}^{s+\frac{1}{2}})} + \sqrt{\lambda}\big(1 +  \Vert \pa_yu_\phi \Vert_{\tilde{L}^\infty_t(\mathcal{B}^{\f32})}\big) \Vert e^{\mathcal{R} t} w_{\phi} \Vert_{\tilde{L}^2_{t,\dot{\theta}(t)}(\mathcal{B}^{s+\frac{1}{2}})} \\ + \big(\sqrt{2} + \Vert \pa_y u_\phi \Vert_{\tilde{L}^\infty_t(\mathcal{B}^{\f32})}^\f12 \big)\Vert e^{\mathcal{R} t} (\pa_y w,\pa_tw)_{\phi} \Vert_{\tilde{L}^2_{t,\dot{\theta}(t)}(\mathcal{B}^{s+\frac{1}{2}})}\\
+ \|\pa_t u_\phi\|_{\tilde{L}_{t}^2(\mathcal{B}^\f12)}\|e^{\mathcal{R}t} u_\phi\|_{\tilde{L}^\infty_t(\mathcal{B}^{s+1})} + \|e^{\mathcal{R}t} (\pa_tu)_\phi\|_{\tilde{L}^\infty_t(\mathcal{B}^{s+1})}+ \|e^{\mathcal{R}t} u_\phi\|_{\tilde{L}^\infty_t(\mathcal{B}^{s+2})}\Big).
\end{multline}
Therefore if we take
\begin{equation}
	\label{4.19bis1}
	\left\{\;
	\begin{aligned}
		\lambda &\geq 8C \Big(2 + \Vert \pa_yu_\phi \Vert_{\tilde{L}^\infty_t(\mathcal{B}^{\f32})} \Big) \\
		k\lambda & \geq 8C \Big(2 + \Vert \pa_y u_\phi \Vert_{\tilde{L}^\infty_t(\mathcal{B}^{\f32})} ),
	\end{aligned}
	\right.
\end{equation}
 we obtain by using the 
\begin{align*}
    \mathcal{E}_{s,\frac{\lambda}{4}}(w)(t) &\leq C \Big( \mathcal{L}(0)+ \|e^{a|D_x|} u_0\|_{\mathcal{B}^{s+2}} + \|e^{a|D_x|} u_0\|_{\mathcal{B}^{s+1}} + \|e^{a|D_x|} u_1\|_{\mathcal{B}^{s+1}} \\&+ \sqrt{\lambda}\big(1 +  \Vert \pa_yu_\phi \Vert_{\tilde{L}^\infty_t(\mathcal{B}^{\f32})}\big) \Vert e^{\mathcal{R} t} w_{\phi} \Vert_{\tilde{L}^2_{t,\dot{\theta}(t)}(\mathcal{B}^{s+\frac{1}{2}})} \Big)
\end{align*}
where 
\begin{align*}
 \mathcal{E}_{s,\frac{\lambda}{4}}(w)(t) &= \Big( \frac{1}{2} \|e^{\mathcal{R} t}(w + \pa_tw)_\phi \|_{\tilde{L}_t^\infty(\mathcal{B}^{s})} +  \frac{1}{2} \|e^{\mathcal{R} t}(\pa_tw)_\phi \|_{\tilde{L}_t^\infty(\mathcal{B}^{s})} +\f34 \|e^{\mathcal{R} t} \pa_yw_\phi \|_{\tilde{L}_t^\infty(\mathcal{B}^{s})} \Big) \\ &+ \f12 \|e^{\mathcal{R} t}(\pa_tw)_\phi\|_{\tilde{L}_t^2(\mathcal{B}^{s})} +\frac{\sqrt{k\lambda}}{4} \|e^{\mathcal{R} t} w_\phi\|_{\tilde{L}^2_{t,\dot{\theta}(t)}(\mathcal{B}^{s+\frac{1}{2}})} + \f14 \|e^{\mathcal{R} t}\pa_yw_\phi\|_{\tilde{L}_t^2(\mathcal{B}^{s})} \\ &+  \frac{\sqrt{\lambda}}{4}  \|e^{\mathcal{R} t}(\pa_t w)_\phi\|_{\tilde{L}^2_{t,\dot{\theta}(t)}(\mathcal{B}^{s+\frac{1}{2}})}  +  \frac{\sqrt{\lambda}}{4} \|e^{\mathcal{R} t}\pa_y w_\phi\|_{\tilde{L}^2_{t,\dot{\theta}(t)}(\mathcal{B}^{s+\frac{1}{2}})}
\end{align*}

Then, taking $\lambda = 8C(2 + \Vert e^{a|D_x|}  \pa_y u_0\Vert_{\mathcal{B}^{\f32}} +  \Vert e^{a|D_x|}  (u_0 + u_1)\Vert_{\mathcal{B}^{\f32}} + \|e^{a|D_x|} u_1\|_{\mathcal{B}^{\f32}} ),$ therefore the condition of the proposition is satisfied and then the proposition is proved.
\end{proof}

As a matter of fact, it remains to present the estimate of $ \Vert \Delta_q^h (\pa_t^2u)_\phi \Vert_{L^2}$, this estimate will serve us in the proof of the last theorem \ref{th:33}. Indeed by applying $\Delta_q^h$ to \eqref{eq:Hnsbis1} and take the $L^2$ inner product of resulting equation with $\Delta_q^h (\pa_t^2u)_\phi$. That yields 
\begin{align*}
\Vert \Delta_q^h (\pa_t^2u)_\phi \Vert_{L^2}^2 &= \psca{\Delta_q^h \pa_y^2 u_\phi,\Delta_q^h (\pa_t^2u)_\phi }_{L^2} - \psca{\Delta_q^h (\pa_t u)_\phi,\Delta_q^h (\pa_t^2u)_\phi }_{L^2}  -  \psca{\Delta_q^h (u\pa_xu)_\phi,\Delta_q^h (\pa_t^2u)_\phi }_{L^2} \\ & - \psca{ \Delta_q^h (v\pa_yu)_\phi,\Delta_q^h (\pa_t^2u)_\phi }_{L^2} - \psca{ \Delta_q^h \pa_x p_\phi , \Delta_q^h (\pa_t^2u)_\phi }_{L^2}.
\end{align*}
The fact that $(\pa_tu)_\phi = \pa_t u_\phi + \lambda \dot{\theta}(t) |D_x|u_\phi$ implies 
$$  \psca{\Delta_q^h (\pa_t u)_\phi,\Delta_q^h (\pa_t^2u)_\phi }_{L^2} = \frac{1}{2} \frac{d}{dt} \Vert \Delta_q^h (\pa_tu)_\phi \Vert_{L^2}^2 + \lambda \dot{\theta}(t) 2^q \Vert \Delta_q^h (\pa_tu)_\phi \Vert_{L^2}^2, $$
from which, we deduce that 
\begin{align*}
\Vert \Delta_q^h (\pa_t^2u)_\phi \Vert_{L^2}^2 + \frac{1}{2} \frac{d}{dt}\Vert \Delta_q^h (\pa_t u)_\phi \Vert_{L^2}^2 \leq  I_1 + I_2 + I_3 + I_4,
\end{align*}
where
\begin{align*}
I_1 &= \abs{\psca{ \Delta_q^h \pa_y^2u_\phi,\Delta_q^h (\pa_t^2u)_\phi }_{L^2}}\\
I_2 &= \abs{\psca{ \Delta_q^h (u\pa_xu)_\phi,\Delta_q^h (\pa_t^2u)_\phi }_{L^2}}\\
I_3 &=  \abs{\psca{  \Delta_q^h (v\pa_yu)_\phi,\Delta_q^h (\pa_t^2u)_\phi }_{L^2}}\\
I_4 &= \abs{\psca{  \Delta_q^h \pa_x p_\phi , \Delta_q^h (\pa_t^2u)_{\phi} }_{L^2}}.
\end{align*}

Since $\dd_x u + \dd_y v = 0$, using \eqref{vvbis1} and integrations by parts, we find 
\begin{align*}
I_4 = \abs{\psca{  \Delta_q^h \pa_x p_\phi , \Delta_q^h (\pa_t^2u)_\phi }_{L^2}} =0. 
\end{align*}

For $I_1$, $I_2$ and $I_3$ we have 
\begin{align*}
    I_2 &= \abs{\psca{  \Delta_q^h \pa_y^2 u_\phi,\Delta_q^h (\pa_t^2u)_\phi }_{L^2}} \leq  C\Vert \Delta_q^h (\pa_y^2u)_\phi \Vert_{L^2}^2  + \frac{1}{10} \Vert \Delta_q^h (\pa_t^2u)_\phi \Vert_{L^2}^2 \\ 
	I_2 &= \abs{\psca{  \Delta_q^h (u\pa_xu)_\phi,\Delta_q^h (\pa_t^2u)_\phi }_{L^2}} \leq  C\Vert \Delta_q^h (u\pa_xu)_\phi \Vert_{L^2}^2  + \frac{1}{10} \Vert \Delta_q^h (\pa_t^2u)_\phi \Vert_{L^2}^2 \\
	I_3 &= \abs{\psca{  \Delta_q^h (v\pa_yu)_\phi,\Delta_q^h (\pa_t^2u)_\phi }_{L^2}} \leq C\Vert \Delta_q^h (v\pa_yu)_\phi \Vert_{L^2}^2  + \frac{1}{10} \Vert \Delta_q^h (\pa_t^2u)_\phi \Vert_{L^2}^2.
\end{align*}
Then, we deduce that
\begin{align*}
	\Vert \Delta_q^h (\pa_t^2u)_\phi \Vert_{L^2}^2 + \frac{1}{2} \frac{d}{dt}\Vert \Delta_q^h (\pa_t u)_\phi \Vert_{L^2}^2 \leq C\left(\Vert \Delta_q^h (u\pa_xu)_\phi \Vert_{L^2}^2  + \Vert \Delta_q^h (v\pa_yu)_\phi \Vert_{L^2}^2 + \Vert \Delta_q^h \pa_y^2u_\phi \Vert_{L^2}^2 \right).
\end{align*}
Multiplying the result by $e^{2\mathcal{R}t}$ and integrating over $[0,t]$, we get 
\begin{multline*}
	\Vert e^{\mathcal{R}t}\Delta_q^h (\pa_t^2u)_\phi \Vert_{L^2_t(L^2)}^2 + \frac{1}{2} \Vert e^{\mathcal{R}t} \Delta_q^h (\pa_t u)_\phi \Vert_{L^\infty_t(L^2)}^2 \\ 
	\leq C \Big(\Vert \Delta_q^h e^{a|D_x|} u_1 \Vert_{L^2}^2   + \Vert e^{\mathcal{R}t}\Delta_q^h (u\pa_xu)_\phi \Vert_{L^2_t(L^2)}^2+\Vert e^{\mathcal{R}t}\Delta_q^h (v\pa_yu)_\phi \Vert_{L^2_t(L^2)}^2 + \Vert e^{\mathcal{R}t}\Delta_q^h \pa_y^2u_\phi \Vert_{L^2_t(L^2)}^2 \Big).
\end{multline*}
Multiplying the above inequality by $2^{3q}$, then taking the square root of the resulting estimate, and finally summing up the obtained equations with respect to $q \in \mathbb{Z},$ we obtain
\begin{multline} \label{5.14bis1}
	\Vert e^{\mathcal{R}t} (\pa_t^2u)_\phi \Vert_{\tilde{L}^2_t(\mathcal{B}^{\f32})} + \frac{1}{4} \Vert e^{\mathcal{R}t} (\pa_t u)_\phi \Vert_{\tilde{L}^\infty_t(\mathcal{B}^{\f32})}\leq C \Big(\Vert e^{a|D_x|} u_1 \Vert_{\mathcal{B}^{\f32}}  \\ 
	+ \Vert e^{\mathcal{R}t} (u\pa_xu)_\phi \Vert_{\tilde{L}^2_t(\mathcal{B}^{\f32})}+\Vert e^{\mathcal{R}t} (v\pa_yu)_\phi \Vert_{\tilde{L}^2_t(\mathcal{B}^{\f32})} + \Vert e^{\mathcal{R}t} \pa_y^2u_\phi \Vert_{\tilde{L}^2_t(\mathcal{B}^{\f32})} \Big).
\end{multline}
Next, to deal the nonlinear terms in \eqref{5.14bis1}, we need the following lemma that give the estimates of the terms $\Vert e^{\mathcal{R}t} (u\pa_xu)_\phi \Vert_{\tilde{L}^2_t(\mathcal{B}^{\f32})} $ and $\Vert e^{\mathcal{R}t} (v\pa_yu)_\phi \Vert_{\tilde{L}^2_t(\mathcal{B}^{\f32})} $, the proof will be given in the appendix \ref{app2}. 
\begin{lemma}\label{lem:estinonlinéaire}
\begin{align}\label{estimUU1}
	\Vert e^{\mathcal{R}t} (u\pa_xu)_\phi \Vert_{\tilde{L}^2_t(\mathcal{B}^{\f32})} &\leq C \Vert  u_\phi \Vert_{\tilde{L}^\infty(\mathcal{B}^\f12)} \Vert e^{\mathcal{R}t} \pa_y u_\phi \Vert_{\tilde{L}^2_t(\mathcal{B}^{\f52})}; 
\end{align}
and
\begin{align}\label{estimUV1}
	\Vert e^{\mathcal{R}t} (v\pa_yu)_\phi \Vert_{\tilde{L}^2_t(\mathcal{B}^{\f32})} &\leq C \Vert  u_\phi \Vert_{\tilde{L}^\infty(\mathcal{B}^\f12)} \Vert e^{\mathcal{R}t} \pa_y u_\phi \Vert_{\tilde{L}^2_t(\mathcal{B}^{\f52})} + \Vert  u_\phi \Vert_{\tilde{L}^\infty(\mathcal{B}^{\f52})} \Vert e^{\mathcal{R}t} \pa_y u_\phi \Vert_{\tilde{L}^2_t(\mathcal{B}^\f12)}.
\end{align}
\end{lemma}
Inserting the above estimates into \eqref{5.14bis1} and then using the smallness condition $\norm{u_{\phi}}_{\mathcal{B}^{\f12}} \leq \frac{1}{4C^2}$ and propositions \ref{prop:1bis1} and \ref{prop:2bis1}, we finally obtain 
\begin{align*}
\Vert e^{\mathcal{R}t} (\pa_tu)_\phi \Vert_{\tilde{L}^2_t(\mathcal{B}^{\f32})} +  \Vert e^{\mathcal{R}t}  (\pa_t u)_\phi \Vert_{\tilde{L}^\infty_t(\mathcal{B}^{\f32})}\leq C\Big( \Vert e^{a|D_x|} u_1 \Vert_{\mathcal{B}^{\f32}} + \Vert e^{a|D_x|} \pa_y u_0 \Vert_{\mathcal{B}^{\f52}} \\ + \Vert e^{a|D_x|} (u_0 + u_1) \Vert_{\mathcal{B}^{\f52}} + \Vert e^{a|D_x|} \pa_y^2 u_0 \Vert_{\mathcal{B}^{\f32}} + \Vert e^{a|D_x|} \pa_y(u_0 + u_1) \Vert_{\mathcal{B}^{\f32}} \Big).
\end{align*}
\section{Global well-posedness of system \eqref{eq:HPNhydroPE} }\label{sec:4}
The goal of this section is to prove the Theorem and to establish the global well-posedness of the system $(\ref{eq:HPNhydroPE})$ with small analytic data. As in Section 2, for any locally bounded function $\Theta$ on $\mathbb{R}_+ \times \mathbb{R}$ and any $u\in L^2(\mathcal{S})$, we define the analyticity in the horizontal variable $x$ by means of the following auxiliary function
\begin{align}\label{analybis}
	u_{\Theta}^{\eps}(t,x,y) = \mathcal{F}^{-1}_{\xi\rightarrow x}(\eps^{\Theta(t,\xi)}\widehat{u}^{\eps}(t,\xi,y)).
\end{align}
The width of the analyticity band $\Theta$ is defined by
\begin{equation*}
	\Theta(t,\xi) = (a- \lambda \tau(t))|\xi|,
\end{equation*}
where $\lambda > 0$ with be precised later and $\tau(t)$ will be chosen in such a way that $\Theta(t,\xi) > 0$, for any $(t,\xi) \in \RR_+ \times \RR$ and $\dot{\Theta}(t) = \Theta'(t) = -\lambda \dot{\tau}(t) \leq 0$. In our paper, we will choose
\begin{eqnarray}\label{bandbis1}
	\dot{\tau}(t) = \Vert \pa_y u_{\Theta}^{\eps}(t) \Vert_{\mathcal{B}^{\f12}} + \eps \Vert \pa_y v_{\Theta}^{\eps}(t) \Vert_{\mathcal{B}^{\f12}} \text{ \ \ \ \ with \ \ \ \ } \tau (0) = 0. 
\end{eqnarray}
In what follows, for the sake of the simplicity, we will neglect the script $\eps$ and write $(u_{\Theta},v_{\Theta})$ instead of $(u_{\Theta}^\epsilon,v_{\Theta}^\epsilon)$. In view of the system \eqref{eq:HPNhydroPE}, we can transform it like a equation of order one in time, so if we define $U = (u,\pa_tu)$ and $V= (v,\pa_tv)$, Then $U$ and $V$ satisfy the following equation
\begin{equation}
	\label{eq:hydroPEbis1bis1}
	\left\{\;
	\begin{aligned}
		&\partial_t U + A_\eps(D)U =-\begin{pmatrix}
		0 \\
		u\pa_xu + v\pa_y u + \pa_x p
		\end{pmatrix} \\
		&\eps^2 \Big(\partial_t V + B_\eps(D)V \Big) =-\begin{pmatrix}
		0 \\
		\eps^2(u\pa_xv + v\pa_y v) + \pa_y p
		\end{pmatrix} \\
		&\pa_x u + \pa_y v = 0	 \\
		&(u,v)/_{y=0} = (u,v)/_{y=1} = 0	
	\end{aligned}
	\right.
\end{equation}
where 
$$ U = \begin{pmatrix}
		u \\
		\pa_tu
		\end{pmatrix} \text{ \ \ and \ \ } A_\eps(D) = \begin{pmatrix}
		0 &  -1 \\
		-\eps^2 \pa_x^2 -\pa_y^2 & 1
		\end{pmatrix} $$ 
		and 
		$$ V = \begin{pmatrix}
		v \\
		\pa_tv
		\end{pmatrix} \text{ \ \ and \ \ } B_\eps(D) = \begin{pmatrix}
		0 &  -1 \\
		-\eps^2 \pa_x^2 -\pa_y^2 & 1
		\end{pmatrix} $$ 
		
Then in view of \eqref{eq:Anphichp3} we observe that $(U,V)_\Theta$ verifies 
		\begin{equation}
	\label{eq:hydroPEVphiUphibis1}
	\left\{\;
	\begin{aligned}
		&\partial_t U_\Theta + \lambda \dot{\tau}(t) |D_x| U_\Theta + A_\eps(D)U_\Theta =-\begin{pmatrix}
		0 \\
		(u\pa_x u)_{\Theta} + (v\pa_y u)_{\Theta} + \pa_x p_{\Theta}
		\end{pmatrix} \\
		&\eps^2 \Big(\partial_t V_\Theta + \lambda \dot{\tau}(t) |D_x| V_\Theta + B_\eps(D)V_\Theta \Big) =-\begin{pmatrix}
		0 \\
		\eps^2(u\pa_xv + v\pa_y v)_\Theta + \pa_y p_\Theta
		\end{pmatrix}\\
		&\pa_x u_\Theta + \pa_y v_\Theta = 0	 \\
		&(u_\Theta,v_\Theta)/_{y=0} = (u_\Theta,v_\Theta)/_{y=1} = 0	
	\end{aligned}
	\right.
\end{equation}

Where $|D_x|$ denote the Fourier multiplier of the symbol $|\xi|$. In what follows, we recall that we use "C" to denote a generic positive constant which can change from line to line.

By applying the dyadic operator in the horizontal variable $\Delta_q^h$ to \eqref{eq:hydroPEVphiUphibis1} and taking the $L^2$ inner product of the resulting equation with $\Delta_q^h U_\Theta$ and $\Delta_q^h V_\Theta$ we obtain 

\begin{multline} \label{S8eq1bis1}
	 \psca{\Delta_q^h \pa_t U_{\Theta}, \Delta_q^h U_{\Theta}}_{L^2} + \lambda \dot{\tau}(t) \psca{\Delta_q^h |D_x| U_{\Theta}, \Delta_q^h U_{\Theta}}_{L^2} + \psca{\Delta_q^h A_\eps(D) U_{\Theta}, \Delta_q^h U_{\Theta}}_{L^2} \\
	= -\psca{\Delta_q^h (u\pa_x u + v\pa_y u)_{\Theta}, \Delta_q^h (\pa_tu)_{\phi}}_{L^2} - \psca{\Delta_q^h \pa_x p_{\phi}, \Delta_q^h (\pa_tu)_{\phi}}_{L^2}
\end{multline}
and
\begin{multline} \label{S8eq2bis1}
	 \eps^2\psca{\Delta_q^h \pa_t V_\Theta, \Delta_q^h V_{\Theta}}_{L^2} + \eps^2\lambda \dot{\tau}(t) \psca{\Delta_q^h |D_x|V_\Theta, \Delta_q^h V_{\Theta}}_{L^2} + \eps^2\psca{\Delta_q^h B_\eps(D) V_\Theta, \Delta_q^h V_{\Theta}}_{L^2} \\
	= -\eps^2 \psca{\Delta_q^h (u\pa_x v + v\pa_y v)_{\Theta}, \Delta_q^h (\pa_tv)_{\Theta}}_{L^2} -  \psca{\Delta_q^h \pa_y p_{\Theta}, \Delta_q^h (\pa_tv)_{\Theta}}_{L^2}.
\end{multline}

As in \eqref{S1eq3bis1} and by using Lemma \ref{lem:1}, we gather 
\begin{align}\label{estimate:UTheta}
&\psca{\Delta_q^h \pa_t U_{\Theta}, \Delta_q^h U_{\Theta}}_{L^2} + \lambda \dot{\tau}(t) \psca{\Delta_q^h |D_x| U_{\Theta}, \Delta_q^h U_{\Theta}}_{L^2} = \psca{\Delta_q^h \pa_t u_{\Theta}, \Delta_q^h u_{\Theta}}_{L^2} \nonumber\\ &+ \psca{\Delta_q^h \pa_t (\pa_tu)_{\Theta}, \Delta_q^h (\pa_tu)_{\Theta}}_{L^2} + \lambda \dot{\tau}(t)\Big( \psca{\Delta_q^h |D_x| u_{\Theta}, \Delta_q^h u_{\Theta}}_{L^2} + \psca{\Delta_q^h |D_x| (\pa_tu)_{\Theta}, \Delta_q^h (\pa_tu)_{\Theta}}_{L^2} \Big) \nonumber\\
&= \frac{1}{2}\frac{d}{dt} \Big( \|\Delta_q^h u_\Theta \|_{L^2}^2 + \|\Delta_q^h (\pa_tu)_\Theta \|_{L^2}^2 \Big) + \lambda \dot{\tau}(t) \Big( \| \Delta_q^h |D_x|^\f12 u_\Theta \|_{L^2}^2 + \|\Delta_q^h |D_x|^\f12 (\pa_tu)_\Theta \|_{L^2}^2 \Big),
\end{align} 
and 
\begin{align}\label{estimate:VTheta}
&\eps^2 \psca{\Delta_q^h \pa_t V_{\Theta}, \Delta_q^h V_{\Theta}}_{L^2} + \eps^2 \lambda \dot{\tau}(t) \psca{\Delta_q^h |D_x| V_{\Theta}, \Delta_q^h V_{\Theta}}_{L^2} = \eps^2\psca{\Delta_q^h \pa_t v_{\Theta}, \Delta_q^h v_{\Theta}}_{L^2} \nonumber\\ &+ \eps^2\psca{\Delta_q^h \pa_t (\pa_tv)_{\Theta}, \Delta_q^h (\pa_tv)_{\Theta}}_{L^2} + \eps^2 \lambda \dot{\tau}(t)\Big( \psca{\Delta_q^h |D_x| v_{\Theta}, \Delta_q^h v_{\Theta}}_{L^2} + \psca{\Delta_q^h |D_x| (\pa_tv)_{\Theta}, \Delta_q^h (\pa_tv)_{\Theta}}_{L^2} \Big)\nonumber \\
&= \frac{1}{2}\frac{d}{dt} \Big( \|\Delta_q^h \eps v_\Theta \|_{L^2}^2 + \|\Delta_q^h \eps (\pa_tv)_\Theta \|_{L^2}^2 \Big) + \lambda \dot{\tau}(t) \Big( \| \Delta_q^h |D_x|^\f12 \eps v_\Theta \|_{L^2}^2 + \|\Delta_q^h |D_x|^\f12 \eps (\pa_tv)_\Theta \|_{L^2}^2 \Big).
\end{align} 
Next, we use that $ (\pa_t u)_\Theta = \pa_t u_\Theta + \lambda \dot{\tau}|D_x| u_\Theta$ and $(\pa_tv)_\Theta = \pa_t v_\Theta + \lambda \dot{\tau}(t) |D_x|v_\Theta $
\begin{align*}
&\psca{\Delta_q^h A_\eps(D) U_{\Theta}, \Delta_q^h U_{\Theta}}_{L^2} = -\psca{\Delta_q^h  (\pa_tu)_{\Theta}, \Delta_q^h u_{\Theta}}_{L^2} -\eps^2\psca{\Delta_q^h \pa_x^2 u_{\Theta}, \Delta_q^h (\pa_tu)_{\Theta}}_{L^2} \\ &- \psca{\Delta_q^h \pa_y^2 u_{\Theta}, \Delta_q^h (\pa_tu)_{\Theta}}_{L^2} + \psca{\Delta_q^h (\pa_tu)_{\Theta}, \Delta_q^h (\pa_tu)_{\Theta}}_{L^2}
\end{align*}
and
\begin{align*}
&\eps^2 \psca{\Delta_q^h B_\eps(D) V_{\Theta}, \Delta_q^h V_{\Theta}}_{L^2} = -\eps^2\psca{\Delta_q^h  (\pa_tv)_{\Theta}, \Delta_q^h v_{\Theta}}_{L^2} -\eps^4\psca{\Delta_q^h \pa_x^2 v_{\Theta}, \Delta_q^h (\pa_tv)_{\Theta}}_{L^2} \\ &- \eps^2\psca{\Delta_q^h \pa_y^2 v_{\Theta}, \Delta_q^h v_{\Theta}}_{L^2} + \eps^2\psca{\Delta_q^h (\pa_tv)_{\Theta}, \Delta_q^h v_{\Theta}}_{L^2}.
\end{align*}
We use again the result of Lemma \ref{lem:1}, we obtain 
\begin{align}\label{estimate:A(D)1}
&\psca{\Delta_q^h A_\eps(D) U_{\Theta}, \Delta_q^h U_{\Theta}}_{L^2} = -\f12 \frac{d}{dt}\|\Delta_q^h u_\Theta \|_{L^2}^2 - \lambda \dot{\tau}(t) \|\Delta_q^h |D_x|^\f12 u_\Theta \|_{L^2}^2   + \frac{1}{2} \frac{d}{dt} \|\Delta_q^h \eps \pa_x u_\Theta \|_{L^2}^2  \nonumber\\ &+ \lambda \dot{\tau}(t) \|\Delta_q^h |D_x|^\f12 \eps\pa_x u_\Theta\|_{L^2}^2 +  \frac{1}{2} \frac{d}{dt} \|\Delta_q^h \pa_y u_\Theta \|_{L^2}^2 + \lambda \dot{\tau}(t) \|\Delta_q^h |D_x|^\f12 \pa_y u_\Theta\|_{L^2}^2 + \|\Delta_q^h (\pa_t u)_\Theta \|_{L^2}^2,
\end{align}
and
\begin{align}\label{estimate:B(D)}
&\eps^2 \psca{\Delta_q^h B_\eps(D) V_{\Theta}, \Delta_q^h V_{\Theta}}_{L^2} = -\f12 \frac{d}{dt}\|\Delta_q^h \eps v_\Theta \|_{L^2}^2 - \lambda \dot{\tau}(t) \|\Delta_q^h |D_x|^\f12 \eps v_\Theta \|_{L^2}^2   + \frac{\eps^2}{2} \frac{d}{dt} \|\Delta_q^h \eps \pa_x v_\Theta \|_{L^2}^2  \nonumber\\ &+ \eps^2 \lambda \dot{\tau}(t) \|\Delta_q^h |D_x|^\f12 \eps\pa_x v_\Theta\|_{L^2}^2 +  \frac{1}{2} \frac{d}{dt} \|\Delta_q^h \eps \pa_y v_\Theta \|_{L^2}^2 + \lambda \dot{\tau}(t) \|\Delta_q^h |D_x|^\f12 \eps \pa_y v_\Theta\|_{L^2}^2 + \|\Delta_q^h \eps (\pa_t v)_\Theta \|_{L^2}^2.
\end{align}
We sum \eqref{estimate:UTheta} with \eqref{estimate:A(D)1}, we achieve 
\begin{align}\label{estimate:A(D)1bis}
&\psca{\Delta_q^h \pa_t U_{\Theta}, \Delta_q^h U_{\Theta}}_{L^2} + \lambda \dot{\tau}(t) \psca{\Delta_q^h |D_x| U_{\Theta}, \Delta_q^h U_{\Theta}}_{L^2} + \psca{\Delta_q^h A_\eps(D) U_{\Theta}, \Delta_q^h U_{\Theta}}_{L^2} \nonumber \\ 
&= \frac{1}{2}\frac{d}{dt} \Big(\|\Delta_q^h (\pa_tu)_\Theta \|_{L^2}^2 + \|\Delta_q^h \eps \pa_x u_\Theta \|_{L^2}^2 +  \|\Delta_q^h \pa_y u_\Theta \|_{L^2}^2 \Big) + \|\Delta_q^h (\pa_t u)_\Theta \|_{L^2}^2 \nonumber \\
& + \lambda \dot{\tau}(t) \Big( \|\Delta_q^h |D_x|^\f12 (\pa_tu)_\Theta \|_{L^2}^2 + \|\Delta_q^h |D_x|^\f12 \eps\pa_x u_\Theta\|_{L^2}^2 + \|\Delta_q^h |D_x|^\f12 \pa_y u_\Theta\|_{L^2}^2 \Big).
\end{align}
Next we sum \eqref{estimate:VTheta} with \eqref{estimate:B(D)}, we obtain 
\begin{align}\label{estimate:B(D)bis}
&\eps^2 \psca{\Delta_q^h \pa_t V_{\Theta}, \Delta_q^h V_{\Theta}}_{L^2} + \eps^2 \lambda \dot{\tau}(t) \psca{\Delta_q^h |D_x| V_{\Theta}, \Delta_q^h V_{\Theta}}_{L^2} + \eps^2 \psca{\Delta_q^h B_\eps(D) V_{\Theta}, \Delta_q^h V_{\Theta}}_{L^2} \nonumber \\ 
& = \frac{1}{2}\frac{d}{dt} \Big( \|\Delta_q^h \eps (\pa_tv)_\Theta \|_{L^2}^2 \eps^2 \|\Delta_q^h \eps \pa_x v_\Theta \|_{L^2}^2 + \|\Delta_q^h \eps \pa_y v_\Theta \|_{L^2}^2  \Big) + \|\Delta_q^h \eps (\pa_t v)_\Theta \|_{L^2}^2  \nonumber\\
& + \lambda \dot{\tau}(t) \Big( \|\Delta_q^h |D_x|^\f12 \eps (\pa_tv)_\Theta \|_{L^2}^2 + \eps^2 \|\Delta_q^h |D_x|^\f12 \eps\pa_x v_\Theta\|_{L^2}^2 + \|\Delta_q^h |D_x|^\f12 \eps \pa_y v_\Theta\|_{L^2}^2 \Big).
\end{align}
By using the Dirichlet boundary condition $(u,v)\vert_{y=0} =(u,v)\vert_{y=1} = 0$, and the incompressibility condition $\pa_x u +\pa_y v =0$ and the relation, we can perform integration by parts, we get 
\begin{align*}
	\abs{\psca{\Delta_q^h \nabla p_{\Theta}, \Delta_q^h (\pa_tu,\pa_tv)_{\Theta}}_{L^2}} = 0.
\end{align*}

We insert the resulting equality \eqref{estimate:A(D)1bis} in \eqref{S8eq1bis1} and \eqref{estimate:B(D)bis} in \eqref{S8eq2bis1}, and then multiplying by $e^{2\mathcal{R}t}$, we achieve
\begin{multline} \label{S8eq3bis1}
	 \frac{1}{2}\frac{d}{dt}\Big( \|e^{\mathcal{R} t}\Delta_q^h (\pa_tu)_\Theta \|_{L^2}^2 + \|e^{\mathcal{R} t}\Delta_q^h \pa_yu_\Theta \|_{L^2}^2 + \|e^{\mathcal{R} t}\Delta_q^h \eps(\pa_tv)_\Theta \|_{L^2}^2 + \|e^{\mathcal{R} t}\Delta_q^h \eps\pa_yv_\Theta \|_{L^2}^2 \\+ \eps^4\|e^{\mathcal{R} t}\Delta_q^h \pa_xv_\Theta \|_{L^2}^2 + \eps^2\|e^{\mathcal{R} t}\Delta_q^h \pa_xu_\Theta \|_{L^2}^2  \Big) +   \|e^{\mathcal{R} t}\Delta_q^h (\pa_tu)_\Theta\|_{L^2}^2 + \|e^{\mathcal{R} t}\Delta_q^h \eps(\pa_tv)_\Theta\|_{L^2}^2 \\ +  \lambda \dot{\tau}(t')  \|e^{\mathcal{R} t}\Delta_q^h |D_x|^\frac{1}{2}(\pa_t u)_\Theta\|_{L^2}^2  + \eps^2\lambda \dot{\tau}(t')  \|e^{\mathcal{R} t}\Delta_q^h |D_x|^\frac{1}{2}(\pa_t v)_\Theta\|_{L^2}^2 +\eps^2\lambda \dot{\tau}(t') \|e^{\mathcal{R} t}\Delta_q^h |D_x|^\frac{1}{2}\pa_x u_\Theta\|_{L^2}^2 \\  +\lambda \dot{\tau}(t') \|e^{\mathcal{R} t}\Delta_q^h |D_x|^\frac{1}{2}\pa_y u_\Theta\|_{L^2}^2 + \lambda \dot{\tau}(t') \|e^{\mathcal{R} t}\Delta_q^h |D_x|^\frac{1}{2}\pa_y v_\Theta\|_{L^2}^2 + \eps^4\lambda \dot{\tau}(t') \|e^{\mathcal{R} t}\Delta_q^h |D_x|^\frac{1}{2}\pa_x v_\Theta\|_{L^2}^2 
\\	\leq \mathcal{R} \Big( \|e^{\mathcal{R} t}\Delta_q^h (\pa_tu)_\Theta \|_{L^2}^2 + \|e^{\mathcal{R} t}\Delta_q^h \pa_yu_\Theta \|_{L^2}^2 + \|e^{\mathcal{R} t}\Delta_q^h \eps(\pa_tv)_\Theta \|_{L^2}^2 + \|e^{\mathcal{R} t}\Delta_q^h \eps\pa_yv_\Theta \|_{L^2}^2 \\+ \eps^4\|e^{\mathcal{R} t}\Delta_q^h \pa_xv_\Theta \|_{L^2}^2 + \eps^2\|e^{\mathcal{R} t}\Delta_q^h \pa_xu_\Theta \|_{L^2}^2  \Big) +  \abs{\psca{\Delta_q^h (u\pa_x u)_{\Theta}, e^{2\mathcal{R} t}\Delta_q^h (\pa_tu)_{\Theta}}}_{L^2} \\ + \abs{\psca{\Delta_q^h ( v\pa_y u)_{\Theta}, e^{2\mathcal{R} t}\Delta_q^h (\pa_tu)_{\Theta}}}_{L^2}
 + \eps^2\abs{\psca{\Delta_q^h (u\pa_x v)_{\Theta}, e^{2\mathcal{R} t}\Delta_q^h (\pa_tv)_{\Theta}}_{L^2}} \\ +\eps^2\abs{\psca{\Delta_q^h (v\pa_y v)_{\Theta}, e^{2\mathcal{R} t}\Delta_q^h (\pa_tv)_{\Theta}}_{L^2}}.
\end{multline}

In what follows, we shall always assume that $t <T^\star_1$, where $T^\star_1$ given by 
\begin{equation} \label{eq:T1starbisbis11} 
	T^\star_1 \stackrel{\tiny def}{=} \sup\set{t>0,\  \ \ \tau(t) \leq  \frac{a}{\lambda}}.
\end{equation}

Then we deduce from Lemmas \ref{lem:ABCbis1}-\ref{lem:BBBbis1} for any $t<T^\star_1$, and integrating over time \eqref{S8eq3bis1}, that 
\begin{multline} \label{S8eq4bis1bis1}
	 \int_0^t \frac{1}{2}\frac{d}{dt}\Big( \|e^{\mathcal{R} t'}\Delta_q^h (\pa_tu,\eps\pa_tv)_\Theta \|_{L^2}^2 + \|e^{\mathcal{R} t'}\Delta_q^h (\pa_yu,\eps\pa_yv)_\Theta \|_{L^2}^2 + \eps^2\|e^{\mathcal{R} t'}\Delta_q^h (\pa_xu,\eps\pa_xv)_\Theta \|_{L^2}^2 \Big) dt' \\ + \int_0^t \Big(\|e^{\mathcal{R} t'}\Delta_q^h (\pa_tu)_\Theta\|_{L^2}^2  + \|e^{\mathcal{R} t'}\Delta_q^h \eps(\pa_tv)_\Theta\|_{L^2}^2 \Big) dt' + \lambda \int_0^t \dot{\tau}(t')  \|e^{\mathcal{R} t'}\Delta_q^h |D_x|^\frac{1}{2}(\pa_t u,\eps\pa_tv)_\Theta\|_{L^2}^2 dt'\\ +  \lambda \int_0^t \dot{\tau}(t') \|e^{\mathcal{R} t'}\Delta_q^h |D_x|^\frac{1}{2}(\pa_y u,\eps \pa_yv)_\Theta\|_{L^2}^2 dt' +\eps^2 \lambda \int_0^t \dot{\tau}(t') \|e^{\mathcal{R} t'}\Delta_q^h |D_x|^\frac{1}{2}(\pa_x u,\eps\pa_xv)_\Theta\|_{L^2}^2 dt'
\\	\leq \mathcal{R} \Big(  \|e^{\mathcal{R} t'}\Delta_q^h (\pa_tu,\eps\pa_tv)_\Theta \|_{L^2}^2 + \|e^{\mathcal{R} t'}\Delta_q^h (\pa_yu,\eps\pa_yv)_\Theta \|_{L^2}^2 + \eps^2\|e^{\mathcal{R} t'}\Delta_q^h (\pa_xu,\eps\pa_xv)_\Theta \|_{L^2}^2 \Big) \\ +  C2^{-2qs} d_q^2 \Big(\Vert e^{\mathcal{R} t} u_{\Theta} \Vert_{\tilde{L}^2_{t,\dot{\tau}(t)}(\mathcal{B}^{s+\frac{1}{2}})} \Vert e^{\mathcal{R} t} (\pa_tu)_{\Theta} \Vert_{\tilde{L}^2_{t,\dot{\tau}(t)}(\mathcal{B}^{s+\frac{1}{2}})} \\ + \Vert e^{\mathcal{R} t}  u_{\Theta} \Vert_{\tilde{L}^2_{t,\dot{\tau}(t)}(\mathcal{B}^{s+\frac{1}{2}})} \Vert e^{\mathcal{R} t} (\eps\pa_tv)_{\Theta} \Vert_{\tilde{L}^2_{t,\dot{\tau}(t)}(\mathcal{B}^{s+\frac{1}{2}})}  \Big).
\end{multline}

We multiply the estimate \eqref{S8eq4bis1bis1} by $2$, we have  

\begin{multline} \label{S8eq4bis1}
	 \int_0^t \frac{d}{dt}\Big( \|e^{\mathcal{R} t'}\Delta_q^h (\pa_tu,\eps\pa_tv)_\Theta \|_{L^2}^2 + \|e^{\mathcal{R} t'}\Delta_q^h (\pa_yu,\eps\pa_yv)_\Theta \|_{L^2}^2 + \eps^2\|e^{\mathcal{R} t'}\Delta_q^h (\pa_xu,\eps\pa_xv)_\Theta \|_{L^2}^2 \Big) dt' \\ + 2\int_0^t \Big(\|e^{\mathcal{R} t'}\Delta_q^h (\pa_tu)_\Theta\|_{L^2}^2  + \|e^{\mathcal{R} t'}\Delta_q^h \eps(\pa_tv)_\Theta\|_{L^2}^2 \Big) dt' + 2\lambda \int_0^t \dot{\tau}(t')  \|e^{\mathcal{R} t'}\Delta_q^h |D_x|^\frac{1}{2}(\pa_t u,\eps\pa_tv)_\Theta\|_{L^2}^2 dt'\\ +  2\lambda \int_0^t \dot{\tau}(t') \|e^{\mathcal{R} t'}\Delta_q^h |D_x|^\frac{1}{2}(\pa_y u,\eps \pa_yv)_\Theta\|_{L^2}^2 dt' +2\eps^2 \lambda \int_0^t \dot{\tau}(t') \|e^{\mathcal{R} t'}\Delta_q^h |D_x|^\frac{1}{2}(\pa_x u,\eps\pa_xv)_\Theta\|_{L^2}^2 dt'
\\	\leq 2\mathcal{R} \Big(  \|e^{\mathcal{R} t'}\Delta_q^h (\pa_tu,\eps\pa_tv)_\Theta \|_{L^2}^2 + \|e^{\mathcal{R} t'}\Delta_q^h (\pa_yu,\eps\pa_yv)_\Theta \|_{L^2}^2 + \eps^2\|e^{\mathcal{R} t'}\Delta_q^h (\pa_xu,\eps\pa_xv)_\Theta \|_{L^2}^2 \Big) \\ +  2C2^{-2qs} d_q^2 \Big(\Vert e^{\mathcal{R} t} u_{\Theta} \Vert_{\tilde{L}^2_{t,\dot{\tau}(t)}(\mathcal{B}^{s+\frac{1}{2}})} \Vert e^{\mathcal{R} t} (\pa_tu)_{\Theta} \Vert_{\tilde{L}^2_{t,\dot{\tau}(t)}(\mathcal{B}^{s+\frac{1}{2}})} \\ + \Vert e^{\mathcal{R} t}  u_{\Theta} \Vert_{\tilde{L}^2_{t,\dot{\tau}(t)}(\mathcal{B}^{s+\frac{1}{2}})} \Vert e^{\mathcal{R} t} (\eps\pa_tv)_{\Theta} \Vert_{\tilde{L}^2_{t,\dot{\tau}(t)}(\mathcal{B}^{s+\frac{1}{2}})}  \Big).
\end{multline}

Now we still have to get some information of the norm $\|\pa_y u_\phi\|_{\mathcal{B}^\f12}$ and  $\|\pa_x u_\phi\|_{\mathcal{B}^\f12}$, for that we need to apply the dyadic operator $\Delta_q^h$ to the equation 
\begin{align}\label{eq:Hypstandbis1}
    e^{\Theta(t,|D_x|)} \Big(\pa_t^2 u + \pa_t u + u\pa_xu+v\pa_yu-\pa_y^2u - \eps^2 \pa_x^2u +\pa_x p \Big) &= 0 \nonumber\\
    e^{\Theta(t,|D_x|)} \Big(\eps^2(\pa_t^2 v + \pa_t v + u\pa_xv+v\pa_yv-\pa_y^2v-\eps^2\pa_xv) +\pa_y p\Big) &= 0,
\end{align}
and then, we take the $L^2$ inner product of the resulting equation \eqref{eq:Hypstandbis1} with $\Delta_q^h u_\Theta$ and $\Delta_q^h v_\Theta$, we obtain  
\begin{multline} \label{S8eq1bis1bis1}
	 \psca{\Delta_q^h (\pa_t^2u)_\Theta, \Delta_q^h  u_{\Theta}}_{L^2} + \psca{\Delta_q^h (\pa_tu)_\Theta, \Delta_q^h  u_{\Theta}}_{L^2} \\- \psca{\Delta_q^h \pa_y^2u_\Theta, \Delta_q^h u_{\Theta}}_{L^2} -\eps^2 \psca{\Delta_q^h \pa_x^2u_\Theta, \Delta_q^h u_{\Theta}}_{L^2}  
	\\= -\psca{\Delta_q^h (u\pa_x u + v\pa_y u)_{\Theta}, \Delta_q^h u_{\Theta}}_{L^2} - \psca{\Delta_q^h \pa_x p_{\Theta}, \Delta_q^h u_{\Theta}}_{L^2},
\end{multline}
and 
\begin{multline} \label{S8eq1bis2bis1}
	 \psca{\Delta_q^h (\eps \pa_t^2v)_\Theta, \Delta_q^h  \eps v_{\Theta}}_{L^2} + \psca{\Delta_q^h (\eps \pa_tv)_\Theta, \Delta_q^h  \eps v_{\Theta}}_{L^2} \\- \psca{\Delta_q^h \eps \pa_y^2v_\Theta, \Delta_q^h \eps v_{\Theta}}_{L^2}  -\eps^2 \psca{\Delta_q^h \eps \pa_x^2v_\Theta, \Delta_q^h \eps v_{\Theta}}_{L^2} 
	 \\= -\eps^2 \psca{\Delta_q^h (u\pa_x v + v\pa_y v)_{\Theta}, \Delta_q^h v_{\Theta}}_{L^2} - \psca{\Delta_q^h \pa_y p_{\Theta}, \Delta_q^h v_{\Theta}}_{L^2},
\end{multline}

In what follows, we shall use again the technical lemmas in Section \ref{sec:1}, to handle term by term in the estimate \eqref{S8eq1bis1bis1} and \eqref{S8eq1bis2bis1}. We start by the complicate term $ I_1 = \psca{\Delta_q^h (\pa_t^2u)_\Theta, \Delta_q^h  u_{\Theta}}_{L^2}$ and $I_2= \psca{\Delta_q^h (\eps \pa_t^2v)_\Theta, \Delta_q^h  \eps v_{\Theta}}_{L^2}$, so by using integration by parts, we find 
\begin{align*}
    I_1 &= \frac{d}{dt
    }\int \Delta_q^h (\pa_t u)_\Theta \Delta_q^h  u_\Theta dX - \int \Delta_q^h (\pa_t u)_\Theta \Delta_q^h (\pa_tu)_\Theta dX \\ &\quad\quad\quad\quad\quad + 2\lambda \dot{\tau}(t)\int \Delta_q^h |D_x| (\pa_tu)_\Theta \Delta_q^h u_\Theta dX \\
    I_2 &= \frac{d}{dt
    }\int \Delta_q^h (\eps \pa_t v)_\Theta \Delta_q^h  \eps v_\Theta dX - \int \Delta_q^h (\eps \pa_t v)_\Theta \Delta_q^h (\eps\pa_tv)_\Theta dX \\ &\quad\quad\quad\quad\quad + 2\lambda \dot{\tau}(t)\int \Delta_q^h |D_x| (\eps\pa_tv)_\Theta \Delta_q^h \eps v_\Theta dX
\end{align*}

Whereas due to the boundary condition, and by integrating by part, we achieve 
\begin{align*}
    &-\psca{\Delta_q^h (\pa_y^2u_\Theta + \eps^2 \pa_x^2 u_\Theta), \Delta_q^h u_{\Theta}}_{L^2} = \|\Delta_q^h  \pa_yu_\Theta \|_{L^2}^2 + \eps^2\|\Delta_q^h  \pa_xu_\Theta \|_{L^2}^2\\
    &\psca{\Delta_q^h (-\eps \pa_y^2v_\Theta - \eps^3 \pa_x^2 v_\Theta), \eps\Delta_q^h v_{\Theta}}_{L^2} = \|\Delta_q^h  \eps\pa_yv_\Theta \|_{L^2}^2 + \eps^2 \|\Delta_q^h  \eps\pa_xv_\Theta \|_{L^2}^2.
\end{align*}

Now, by using the Dirichlet boundary condition $(u,v)\vert_{y=0} =(u,v)\vert_{y=1} = 0$, and the incompressibility condition $\pa_x u +\pa_y v =0$, we can find by integrating by parts the estimate of the pressure  
\begin{align*}
	 \abs{\psca{\Delta_q^h \nabla p_{\Theta}, \Delta_q^h (u,v)_{\Theta}}_{L^2}} = \abs{ \psca{\Delta_q^h  p_{\Theta}, \Delta_q^h \text{ div } (u,v)_{\Theta}}} = 0.
\end{align*}

Then by multiplying \eqref{S8eq1bis1bis1} and \eqref{S8eq1bis2bis1} by $e^{2\mathcal{R}t}$, and integrating the resulting inequality over time, we achieve
\begin{multline} \label{S8eq1bis3bis1}
	 \int \bigg( \frac{d}{dt} \Big( e^{2\mathcal{R}t}\Delta_q^h (\pa_t u)_\Theta \Delta_q^h  u_\Theta \Big) -  e^{2\mathcal{R}t}\Delta_q^h (\pa_t u)_\Theta \Delta_q^h (\pa_tu)_\Theta  + 2\lambda \dot{\tau}(t) e^{2\mathcal{R}t} \Delta_q^h |D_x| (\pa_tu)_\Theta \Delta_q^h u_\Theta \bigg) dX \\
	 + \int \bigg( \frac{d}{dt} \Big( e^{2\mathcal{R}t}\Delta_q^h (\eps \pa_t v)_\Theta \Delta_q^h  (\eps v)_\Theta \Big) - e^{2\mathcal{R}t}\Delta_q^h (\eps\pa_t v)_\Theta \Delta_q^h (\eps\pa_tv)_\Theta + 2\lambda \dot{\tau}(t) e^{2\mathcal{R}t} \Delta_q^h |D_x| (\eps\pa_tv)_\Theta \Delta_q^h (\eps v)_\Theta \bigg) dX  \\
	   + \frac{1}{2}\frac{d}{dt} \|e^{\mathcal{R}t}\Delta_q^h u_\Theta\|_{L^2}^2 + \frac{1}{2}\frac{d}{dt} \|e^{\mathcal{R}t}\Delta_q^h \eps v_\Theta\|_{L^2}^2 + \lambda \dot{\tau}(t) \|e^{\mathcal{R}t}\Delta_q^h |D_x|^\frac{1}{2} (u_\Theta,\eps v_\Theta) \|_{L^2}^2 + \|e^{\mathcal{R}t}\Delta_q^h  \pa_yu_\Theta \|_{L^2}^2 \\
	    + \|e^{\mathcal{R}t}\Delta_q^h  \pa_y \eps v_\Theta \|_{L^2}^2 + \eps^2\|e^{\mathcal{R}t}\Delta_q^h  \pa_xu_\Theta \|_{L^2}^2+ \eps^2\|e^{\mathcal{R}t}\Delta_q^h  \pa_x \eps v_\Theta \|_{L^2}^2	   \\
	    = \mathcal{R}\Big(2 \int e^{2\mathcal{R}t}\Delta_q^h (\pa_t u)_\Theta \Delta_q^h  u_\Theta dX + 2\int e^{2\mathcal{R}t}\Delta_q^h (\eps\pa_t v)_\Theta \Delta_q^h  (\eps v)_\Theta dX + \|e^{\mathcal{R}t} \Delta_q^h (u_\Theta,\eps v_\Theta) \|_{L^2}^2\Big)
	    \\ -\psca{\Delta_q^h (u\pa_x u + v\pa_y u)_{\Theta}, e^{2\mathcal{R}t}\Delta_q^h u_{\Theta}}_{L^2} - \eps^2 \psca{\Delta_q^h (u\pa_x v + v\pa_y v)_{\Theta}, e^{2\mathcal{R}t}\Delta_q^h v_{\Theta}}_{L^2}.       
\end{multline}

In view, of Lemma \ref{lem:ABCbis1}-\ref{lem:Apa_yBCbis1} for $t<T^\star_1$, and by summing $ \eqref{S8eq4bis1}$ with \eqref{S8eq1bis3bis1}, we obtain

\begin{multline} \label{S8eq5bis1}
	    \mathbf{B}(u,v)(t) \leq \mathcal{R}\int_0^t \Big[ \|e^{\mathcal{R}t'} \Delta_q^h \big(\pa_tu +u,\eps(\pa_t v+v)\big)_\Theta \|_{L^2}^2 + 2 \|e^{\mathcal{R}t'}\Delta_q^h \pa_y(u,\eps v)_\Theta \|_{L^2}^2 \\
	+ 2\eps^2 \|e^{\mathcal{R}t'} \Delta_q^h \pa_x (u,\eps v)_\Theta \|_{L^2}^2 + \|e^{\mathcal{R}t'} \Delta_q^h (\pa_t u,\eps \pa_t v)_\Theta \|_{L^2}^2 \Big] dt' + \lambda \int_0^t \dot{\tau}(t') \Vert e^{\mathcal{R} t}|D_x|^\f12 (u,\eps v)_{\Theta} \Vert_{L^2}^2 dt' \\
	+ \lambda \int_0^t \dot{\tau}(t') \Vert e^{\mathcal{R} t}|D_x|^\f12 (\pa_t u,\eps \pa_t v)_{\Theta} \Vert_{L^2}^2 dt' + 2C2^{-2qs} d_q^2 \bigg(\Vert e^{\mathcal{R} t} u_{\Theta} \Vert_{\tilde{L}^2_{t,\dot{\tau}(t)}(\mathcal{B}^{s+\frac{1}{2}})}\\ \times  \Big( \Vert e^{\mathcal{R} t} (\pa_tu)_{\Theta} \Vert_{\tilde{L}^2_{t,\dot{\tau}(t)}(\mathcal{B}^{s+\frac{1}{2}})}  +   \Vert e^{\mathcal{R} t} (\eps\pa_tv)_{\Theta} \Vert_{\tilde{L}^2_{t,\dot{\tau}(t)}(\mathcal{B}^{s+\frac{1}{2}})} \Big) + \Vert e^{\mathcal{R} t}  (u_{\Theta},\eps v_\Theta) \Vert_{\tilde{L}^2_{t,\dot{\tau}(t)}(\mathcal{B}^{s+\frac{1}{2}})}^2 \Bigg). 
\end{multline}
where 
\begin{align*}
\mathbf{B}(u,v)(t) &= \int_0^t \frac{d}{dt} \Big( \f12 \|e^{\mathcal{R} t'}\Delta_q^h \big(\pa_t u+u,\eps (\pa_t v+v)\big)_\Theta \|^2_{L^2} + \f12 \|e^{\mathcal{R} t'}\Delta_q^h \big(\pa_t u,\eps (\pa_t v)\big)_\Theta \|^2_{L^2} \\ & + \|e^{\mathcal{R} t'}\Delta_q^h \pa_y(u,\eps v)_\Theta \|_{L^2}^2 + \eps^2 \|e^{\mathcal{R} t'}\Delta_q^h \pa_x(u,\eps v)_\Theta \|_{L^2}^2 \Big)dt'  + \int_0^t \Big(\|e^{\mathcal{R} t'}\Delta_q^h (\pa_tu,\eps \pa_t v)_\Theta\|_{L^2}^2 \\ & +  \|e^{\mathcal{R}t'}\Delta_q^h  \pa_y(u,\eps v)_\Theta \|_{L^2}^2 + \eps^2 \|e^{\mathcal{R}t'}\Delta_q^h  \pa_x(u,\eps v)_\Theta \|_{L^2}^2  + \lambda \dot{\tau}(t') \|e^{\mathcal{R}t'}\Delta_q^h |D_x|^\frac{1}{2} (u,\eps v)_\Theta \|_{L^2}^2\\  &  + 2\lambda \big( \dot{\tau}(t')  \|e^{\mathcal{R} t'}\Delta_q^h |D_x|^\frac{1}{2}(\pa_t u, \eps \pa_t v)_\Theta\|_{L^2}^2 +  \dot{\tau}(t') \|e^{\mathcal{R} t'}\Delta_q^h |D_x|^\frac{1}{2}(\pa_y u,\eps \pa_yv)_\Theta\|_{L^2}^2 \big) \\ &+2 \eps^2 \lambda  \dot{\tau}(t') \|e^{\mathcal{R} t'}\Delta_q^h |D_x|^\frac{1}{2}(\pa_x u,\eps\pa_xv)_\Theta\|_{L^2}^2 \Big)dt'
\end{align*}

We begin with by observing that the term in the square brackets in \eqref{S8eq5bis1} can be absorbed by the dissipation $$\int_0^t \|e^{\mathcal{R} t'}\Delta_q^h (\pa_tu,\eps \pa_t v)_\Theta\|_{L^2}^2 dt' + \int_0^t \|e^{\mathcal{R} t'}\Delta_q^h(\pa_yu,\eps\pa_y v)_\Theta\|_{L^2}^2 dt' + \eps^2 \int_0^t \|e^{\mathcal{R} t'}\Delta_q^h(\pa_xu,\eps \pa_x v)_\Theta\|_{L^2}^2 dt' $$. Indeed, since the value of $\mathcal{R}$ is smaller than $\min\lbrace \frac{1}{8},\frac{k}{8} \rbrace$, we have that 
\begin{align*}
\mathcal{R} &\Big[ \|e^{\mathcal{R}t'} \Delta_q^h \big(\pa_tu +u,\eps(\pa_t v+v)\big)_\Theta \|_{L^2}^2 + 2 \|e^{\mathcal{R}t'}\Delta_q^h \pa_y(u,\eps v)_\Theta \|_{L^2}^2 
	+ 2\eps^2 \|e^{\mathcal{R}t'} \Delta_q^h \pa_x (u,\eps v)_\Theta \|_{L^2}^2 \\ & + \|e^{\mathcal{R}t'} \Delta_q^h (\pa_t u,\eps \pa_t v)_\Theta \|_{L^2}^2 \Big] \leq \mathcal{R} \Big[ \|e^{\mathcal{R}t'} \Delta_q^h \big(u,\eps v\big)_\Theta \|_{L^2}^2 + 2 \|e^{\mathcal{R}t'}\Delta_q^h \pa_y(u,\eps v)_\Theta \|_{L^2}^2 
	\\
& + 2\eps^2 \|e^{\mathcal{R}t'} \Delta_q^h \pa_x (u,\eps v)_\Theta \|_{L^2}^2 + 2\|e^{\mathcal{R}t'} \Delta_q^h (\pa_t u,\eps \pa_t v)_\Theta \|_{L^2}^2 \Big] \leq \frac{k}{8}  \|e^{\mathcal{R}t'} \Delta_q^h \big(u,\eps v\big)_\Theta \|_{L^2}^2 \\ & + \frac{1}{4}\|e^{\mathcal{R}t'}\Delta_q^h \pa_y(u,\eps v)_\Theta \|_{L^2}^2 +\frac{\eps^2}{4}  \|e^{\mathcal{R}t'} \Delta_q^h \pa_x (u,\eps v)_\Theta \|_{L^2}^2 + \frac{1}{4}\|e^{\mathcal{R}t'} \Delta_q^h (\pa_t u,\eps \pa_t v)_\Theta \|_{L^2}^2.
\end{align*}
To absorb those terms, we shall then invoke the Poincaré inequality in $y\in(0,1)$ : $k \|\Delta_q^h (u,\eps v)_\Theta \|_{L^2} \leq \|\Delta_q^h \pa_y (u,\eps v)_\Theta\|_{L^2}$. Thus 
\begin{multline}\label{esti:Poincaré1}
\mathcal{R} \Big[ \|e^{\mathcal{R}t'} \Delta_q^h \big(\pa_tu +u,\eps(\pa_t v+v)\big)_\Theta \|_{L^2}^2 + 2 \|e^{\mathcal{R}t'}\Delta_q^h \pa_y(u,\eps v)_\Theta \|_{L^2}^2 
	\\ + 2\eps^2 \|e^{\mathcal{R}t'} \Delta_q^h \pa_x (u,\eps v)_\Theta \|_{L^2}^2  + \|e^{\mathcal{R}t'} \Delta_q^h (\pa_t u,\eps \pa_t v)_\Theta \|_{L^2}^2 \Big] 
\\ \leq \frac{1}{2}\Big(\|e^{\mathcal{R} t'} \Delta_q^h \pa_y(u,\eps v)_\Theta \|_{L^2}^2 + \|e^{\mathcal{R} t'}\Delta_q^h (\pa_tu,\eps \pa_t v)_\Theta \|_{L^2}^2 + \eps^2 \|e^{\mathcal{R} t'}\Delta_q^h(\pa_xu,\eps \pa_x v)_\Theta\|_{L^2}^2 \Big). 
\end{multline}
We replace the obtained result \eqref{esti:Poincaré1} in \eqref{S8eq5bis1}, we deduce that

\begin{multline} \label{S8eq5bis12}
	\mathbf{F}(u,v)(t) \leq 2C2^{-2qs} d_q^2 \bigg(\Vert e^{\mathcal{R} t} u_{\Theta} \Vert_{\tilde{L}^2_{t,\dot{\tau}(t)}(\mathcal{B}^{s+\frac{1}{2}})} \times  \Big( \Vert e^{\mathcal{R} t} (\pa_tu)_{\Theta} \Vert_{\tilde{L}^2_{t,\dot{\tau}(t)}(\mathcal{B}^{s+\frac{1}{2}})} \\ +   \Vert e^{\mathcal{R} t} (\eps\pa_tv)_{\Theta} \Vert_{\tilde{L}^2_{t,\dot{\tau}(t)}(\mathcal{B}^{s+\frac{1}{2}})} \Big) + \Vert e^{\mathcal{R} t}  (u_{\Theta},\eps v_\Theta) \Vert_{\tilde{L}^2_{t,\dot{\tau}(t)}(\mathcal{B}^{s+\frac{1}{2}})}^2 \Bigg),
\end{multline}
where 
\begin{multline}\label{Formule:F}
\mathbf{F}(u,v)(t) = \int_0^t \frac{d}{dt} \Big( \f12 \|e^{\mathcal{R} t'}\Delta_q^h \big(\pa_t u+u,\eps (\pa_t v+v)\big)_\Theta \|^2_{L^2} + \f12 \|e^{\mathcal{R} t'}\Delta_q^h \big(\pa_t u,\eps (\pa_t v)\big)_\Theta \|^2_{L^2} \\  + \|e^{\mathcal{R} t'}\Delta_q^h \pa_y(u,\eps v)_\Theta \|_{L^2}^2 + \eps^2 \|e^{\mathcal{R} t'}\Delta_q^h \pa_x(u,\eps v)_\Theta \|_{L^2}^2 \Big)dt'  +  \int_0^t \Big(\f12 \|e^{\mathcal{R} t'}\Delta_q^h (\pa_tu,\eps \pa_t v)_\Theta\|_{L^2}^2 \\  +  \f12\|e^{\mathcal{R}t'}\Delta_q^h  \pa_y(u,\eps v)_\Theta \|_{L^2}^2 + \frac{\eps^2}{2} \|e^{\mathcal{R}t'}\Delta_q^h  \pa_x(u,\eps v)_\Theta \|_{L^2}^2  + k\lambda \dot{\tau}(t') \|e^{\mathcal{R}t'}\Delta_q^h |D_x|^\frac{1}{2} (u,\eps v)_\Theta \|_{L^2}^2  \\ + \lambda \dot{\tau}(t')  \|e^{\mathcal{R} t'}\Delta_q^h |D_x|^\frac{1}{2}(\pa_t u, \eps \pa_t v)_\Theta\|_{L^2}^2 + \lambda \dot{\tau}(t') \|e^{\mathcal{R} t'}\Delta_q^h |D_x|^\frac{1}{2}(\pa_y u,\eps \pa_yv)_\Theta\|_{L^2}^2  \\+ \eps^2 \lambda \dot{\tau}(t') \|e^{\mathcal{R} t'}\Delta_q^h |D_x|^\frac{1}{2}(\pa_x u,\eps\pa_xv)_\Theta\|_{L^2}^2 \Big)dt'.
\end{multline}
If we look at the third line of the equality $\eqref{Formule:F}$, we remark that we have a new term $$k\lambda \int_0^t \dot{\tau}(t') \|e^{\mathcal{R}t'}\Delta_q^h |D_x|^\frac{1}{2} (u,\eps v)_\Theta \|_{L^2}^2 dt' ,$$ this term handle from the Poincaré inequality applied to the following term $$ \lambda \int_0^t \dot{\tau}(t') \|e^{\mathcal{R} t'}\Delta_q^h |D_x|^\frac{1}{2}(\pa_y u,\eps \pa_yv)_\Theta\|_{L^2}^2 dt'. $$

Multiplying \eqref{S8eq5bis1} by $2^{2qs}$ for $s\in]0,1[$ and then integrating over time, and summing with respect to $q \in \ZZ$, we find that for $t<T^\star_1$

\begin{multline} \label{S8eq6bis1}
	 \mathbf{E}_{s,\lambda,k}(u,v)(t)	\leq C\Big( \|e^{a|D_x|}\pa_y(u_0,\eps v_0) \|_{\mathcal{B}^{s}} + \eps \|e^{a|D_x|}\pa_x(u_0,\eps v_0) \|_{\mathcal{B}^{s}}\\ +\|e^{a|D_x|}(u_1,\eps  v_1) \|_{\mathcal{B}^{s}} +\|e^{a|D_x|}(u_0+u_1,\eps (v_0 + v_1)) \|_{\mathcal{B}^{s}} \\ + \sqrt{2} C\Vert e^{\mathcal{R} t} ((\pa_t u)_{\Theta},\eps (\pa_tv)_{\Theta}) \Vert_{\tilde{L}^2_{t,\dot{\tau}(t)}(\mathcal{B}^{s+\frac{1}{2}})} + \sqrt{2}C \Vert e^{\mathcal{R} t} (u,\eps v)_{\Theta} \Vert_{\tilde{L}^2_{t,\dot{\tau}(t)}(\mathcal{B}^{s+\frac{1}{2}})},
\end{multline}
where 
\begin{align*}
\mathbf{E}_{s,\lambda,k}(u,v)(t) &=  \Big( \frac{1}{2}\|e^{\mathcal{R} t}(u + \pa_tu,\eps (v+\pa_tv) )_\Theta \|_{\tilde{L}_t^\infty(\mathcal{B}^{s})} + \|e^{\mathcal{R} t} \pa_y(u,\eps v)_\Theta \|_{\tilde{L}_t^\infty(\mathcal{B}^{s})} \\ &+ \eps \|e^{\mathcal{R} t} \pa_x(u,\eps v)_\Theta \|_{\tilde{L}_t^\infty(\mathcal{B}^{s})} \Big) + \frac{1}{2}\|e^{\mathcal{R} t}( \pa_tu,\eps\pa_tv)_\Theta \|_{\tilde{L}_t^\infty(\mathcal{B}^{s})} \\&+\f12\|e^{\mathcal{R} t}(\pa_tu,\eps\pa_tv)_\Theta\|_{\tilde{L}_t^2(\mathcal{B}^{s})} +\sqrt{k\lambda} \|e^{\mathcal{R} t} (u,\eps v)_\Theta \|_{\tilde{L}^2_{t,\dot{\tau}(t)}(\mathcal{B}^{s+\frac{1}{2}})}\\& + \f12\|e^{\mathcal{R} t}\pa_y(u,\eps v)_\Theta\|_{\tilde{L}_t^2(\mathcal{B}^{s})}  +\frac{\eps}{2} \|e^{\mathcal{R} t}\pa_x(u,\eps v)_\Theta\|_{\tilde{L}_t^2(\mathcal{B}^{s})}+  \sqrt{\lambda}  \|e^{\mathcal{R} t}(\pa_t u,\eps \pa_tv)_\Theta \|_{\tilde{L}^2_{t,\dot{\tau}(t)}(\mathcal{B}^{s+\frac{1}{2}})} \\ &+  \sqrt{\lambda} \|e^{\mathcal{R} t}\pa_y (u,\eps v)_\Theta \|_{\tilde{L}^2_{t,\dot{\tau}(t)}(\mathcal{B}^{s+\frac{1}{2}})} + \sqrt{\lambda} \|e^{\mathcal{R} t}\pa_x (u,\eps v)_\Theta\|_{\tilde{L}^2_{t,\dot{\tau}(t)}(\mathcal{B}^{s+\frac{1}{2}})}
\end{align*}

Taking $\lambda \geq 2C^2$ in the above inequality leads to 

\begin{multline} \label{S8eq7bis1}
	 \Big( \f12 \|e^{\mathcal{R} t}(u + \pa_tu,\eps(v+ \pa_t v))_\Theta \|_{\tilde{L}_t^\infty(\mathcal{B}^{s})} +  \frac{1}{2}\|e^{\mathcal{R} t}( \pa_tu,\eps\pa_tv)_\Theta \|_{\tilde{L}_t^\infty(\mathcal{B}^{s})} + \|e^{\mathcal{R} t} \pa_y(u,\eps v)_\Theta \|_{\tilde{L}_t^\infty(\mathcal{B}^{s})} \\ + \eps \|e^{\mathcal{R} t} \pa_x(u,\eps v)_\Theta \|_{\tilde{L}_t^\infty(\mathcal{B}^{s})}\Big) + \f12 \|e^{\mathcal{R} t}(\pa_tu,\eps \pa_tv)_\Theta\|_{\tilde{L}_t^2(\mathcal{B}^{s})}  + \f12 \|e^{\mathcal{R} t}\pa_y(u,\eps v)_\Theta\|_{\tilde{L}_t^2(\mathcal{B}^{s})} \\+ \frac{\eps}{2} \|e^{\mathcal{R} t}\pa_x(u,\eps v)_\Theta\|_{\tilde{L}_t^2(\mathcal{B}^{s})}   \leq C\Big(\|e^{a|D_x|}\pa_y(u_0,\eps v_0) \|_{\mathcal{B}^{s}} + \eps \|e^{a|D_x|}\pa_x(u_0,\eps v_0) \|_{\mathcal{B}^{s}} \\ + \|e^{a|D_x|}(u_1,\eps  v_1) \|_{\mathcal{B}^{s}} +\|e^{a|D_x|}(u_0+u_1,\eps (v_0 + v_1)) \|_{\mathcal{B}^{s}}\Big), \text{ \ for \ } t<T^\star_1 .
\end{multline}

We recall that we already defined $\dot{\tau}(t) = \Vert \pa_y u_{\Theta}^{\eps}(t) \Vert_{\mathcal{B}^{\f12}} + \eps \Vert \pa_y v_{\Theta}^{\eps}(t) \Vert_{\mathcal{B}^{\f12}} \text{ with } \tau (0) = 0 $. Then, for any $0 < t < T^\star_1$, Inequality \eqref{S3eq7bis1} yields for $s=\f12$

\begin{align*}
	\tau (t) &= \int_0^t \Vert \pa_y u_{\Theta}(t') \Vert_{\mathcal{B}^\f12} + \eps \Vert \pa_y v_{\Theta}(t') \Vert_{\mathcal{B}^{\f12}} dt' \\ &\leq \int_0^t  e^{-\mathcal{R}t'} \Vert e^{\mathcal{R}t'} \pa_y (u,\eps v)_{\Theta}(t') \Vert_{\mathcal{B}^\f12} dt' \\
	&\leq \left( \int_0^t  e^{-2\mathcal{R}t'} dt' \right)^{\frac{
	1}{2}} \left( \int_0^t \Vert e^{\mathcal{R}t'} \pa_y (u,\eps v)_{\Theta}(t') \Vert_{\mathcal{B}^\f12}^2  dt' \right)^{\frac{
	1}{2}} \\
	&\leq C \norm{e^{\mathcal{R} t} \pa_y (u,\eps v)_{\Theta}}_{\tilde{L}^2_t(\mathcal{B}^{\f12})} \\
	&\leq C \Big( \|e^{a|D_x|}\pa_y(u_0,\eps v_0) \|_{\mathcal{B}^{\frac{1}{2}}} + \eps \|e^{a|D_x|}\pa_x(u_0,\eps v_0)\|_{\mathcal{B}^\f12} + \|e^{a|D_x|}(u_0+u_1,\eps (v_0 + v_1)) \|_{\mathcal{B}^{\f12}} \\ &+ \|e^{a|D_x|}(u_1,\eps v_1) \|_{\mathcal{B}^{\f12}} \Big) < \frac{a}{2\lambda}
\end{align*}
A continuity argument implies that $T^\star_1 = +\infty$ and we have \eqref{S3eq7bis1} is valid for any $t \in \RR_+$.

\section{The convergence to the perturbed hydrostatic Navier-Stokes equations}\label{sec:6}

In this section, we justify the limit from the scaled perturbed anisotropic Navier-Stokes system to the perturbed hydrostatic Navier-Stokes system in a 2D thin domain. As in the sections 3 and 4, the main idea will be to obtain a control of the difference between the two solutions in analytic spaces, by using energy estimates with exponential weights in the Fourier variable. As previously, the exponent of the exponential weight is depending on time but shall take into account now the "loss of the analyticity" for both solutions, of the re-scaled perturbed Navier-Stokes system and respectively of the perturbed hydrostatic Navier-Stokes equations. To this end, we introduce  

\begin{align} \label{eq:psiphiqbis1}
	\left\{
	\begin{aligned}
		R^{1,\eps} = u^\eps - u, \\
		R^{2,\eps} = v^\eps - v, \\
		q^\eps     = p^\eps-p.
	\end{aligned}
		\right.
\end{align}

Then, systems \eqref{eq:HPNhydroPE} and \eqref{eq:HPNhydrolimit} imply that $(R^{1,\eps},R^{2,\eps},q_{\eps})$ verifies 
\begin{equation}\label{S6eq1bis1}
\quad\left\{\begin{array}{l}
\displaystyle \pa_t^2 R^{1,\eps}+ \partial_t R^{1,\eps} - \eps^2\partial_x^2 R^{1,\eps} -\partial_y^2 R^{1,\eps} + \partial_x q^\eps= F^{1,\eps}\ \ \mbox{in} \ \mathcal{S}\times ]0,\infty[,\\
\displaystyle \eps^2\left(\pa_t^2 R^{2,\eps} + \partial_t R^{2,\eps} - \eps^2\partial_x^2 R^{2,\eps}-\partial_y^2 R^{2,\eps} \right)+\partial_y q^\eps = F^{2,\eps},\\
\displaystyle \partial_x R^{1,\eps}+\partial_y R^{2,\eps}=0 \\
\displaystyle \left( R^{1,\eps}, R^{2,\eps}\right)|_{t=0}=\left(u_0^\eps-u_0, v_0^{\eps} - v_0 \right),\\
\displaystyle \pa_t \left( R^{1,\eps}, R^{2,\eps}\right)|_{t=0}=\left(u_1^\eps-u_1, v_1^{\eps} - v_1 \right),\\
\displaystyle \left( R^{1,\eps}, R^{2,\eps}\right)|_{y=0}=\left( R^{1,\eps}, R^{2,\eps} \right)|_{y=1} = 0,
\end{array}\right.
\end{equation}
where the remaining terms $F^{i,\eps}$, with $ i =1,2$, are determined by   
\begin{align} \label{eq:Fibis1}
	\left\{
	\begin{aligned}
		F^{1,\eps} &= \eps^2 \pa_x^2 u -(u^{\eps} \pa_x u^{\eps} - u\pa_xu )- (v^\eps \pa_y u^\eps - v\pa_y u), \\
		F^{2,\eps} &= - \eps^2\left(\pa_t^2 v + \pa_t v -\eps^2 \pa_x^2 v -\pa_y^2 v + u^{\eps} \pa_x v^{\eps} + v^\eps \pa_y v^\eps \right).
	\end{aligned}
	\right.
\end{align}

As $ (R^{1,\eps},R^{2,\eps})$ satisfies the boundary condition and also the free divergence, therefore these two conditions allows us to write

\begin{align}\label{PR2bis1}
    R^{2,\eps}(t,x,y) &= \int_0^y \pa_y R^{2,\eps}(t,x,s) ds = - \int_0^y \pa_x R^{1,\eps}(t,x,s)ds 
\end{align}

If we replace $y$ by $1$ in \eqref{PR2bis1}, we deduce from the incompressibility condition $\partial_x R^{1,\eps}+\partial_y  R^{2,\eps}=0 $ that

\begin{align*}
	\pa_x\int_0^1R^{1,\eps}(t,x,y)\, dy = -\int_0^1 \dd_y R^{2,\eps}(t,x,y)\, dy = R^{2,\eps}(t,x,1) - R^{2,\eps}(t,x,0) = 0.
\end{align*}

In what follows, for simplicity, we shall neglect the subscript $\eps$ in $(R^{1,\eps},R^{2,\eps},q^\eps,)$. In view of the system \eqref{S6eq1bis1}, we can transform it like a equation of order one in time, so if we define $G = (R^1,\pa_tR^1)$ and $H= (R^2,\pa_tR^2)$, Then $G$ and $H$ satisfy the following equation
\begin{equation}
	\label{eq:hydroPEGHBIS1bis1}
	\left\{\;
	\begin{aligned}
		&\partial_t G + A_\eps(D)G =\begin{pmatrix}
		0 \\
		F^1 - \pa_x q
		\end{pmatrix} \\
		&\eps^2 \Big(\partial_t H + B_\eps(D)H \Big) =\begin{pmatrix}
		0 \\
		F^2 - \pa_y q
		\end{pmatrix} \\
		&\pa_x R^1 + \pa_y R^2 = 0	 \\
		&(R^1,R^2)/_{y=0} = (R^1,R^2)/_{y=1} = 0	
	\end{aligned}
	\right.
\end{equation}
where 
$$ G = \begin{pmatrix}
		R^1 \\
		\pa_tR^1
		\end{pmatrix} \text{ \ \ and \ \ } A_\eps(D) = \begin{pmatrix}
		0 &  -1 \\
		-\eps^2 \pa_x^2 -\pa_y^2 & 1
		\end{pmatrix} $$ 
		and 
		$$ H = \begin{pmatrix}
		R^2 \\
		\pa_tR^2
		\end{pmatrix} \text{ \ \ and \ \ } B_\eps(D) = \begin{pmatrix}
		0 &  -1 \\
		-\eps^2 \pa_x^2 -\pa_y^2 & 1
		\end{pmatrix} .$$ 
		
In view of \eqref{eq:Anphichp3}, we define for any suitable function $f$
\begin{eqnarray}\label{analytiqueff**}
f_{\varphi}(t,x,y) = \mathcal{F}^{-1}_{\xi \rightarrow x}\left(e^{\varphi(t,\xi)} \widehat{f}(t,\xi,y)\right) \text{ where \  } \varphi(t,\xi) =\left(a-\mu \eta(t)\right)|\xi|,
\end{eqnarray}
where $\mu \geq \lambda$ will be determined later, and $\eta(t)$ is given by 
$$ \eta(t) = \int_0^t \left( \Vert (\pa_y u^{\eps}_{\Theta}, \eps \pa_x u^{\eps}_{\Theta})(t') \Vert_{\mathcal{B}^{\f12}} + \Vert \pa_y u_{\phi} (t') \Vert_{\mathcal{B}^{\f12}} \right) dt'. $$
We can observe that, if we take $c_0$ and $c_1$ small enough in Theorems \ref{th:11} and \ref{th:22} then $\varphi(t) \geq 0$ and 
$$ 0 \leq \varphi(t,\xi) \leq min\left(\phi(t,\xi),\Theta(t,\xi)\right) .$$ 

Then in view of \eqref{analytiqueff**}, we observe that $(G,H)_\varphi$ verifies 
		\begin{equation}
	\label{eq:hydroPEGHvarphibis1}
	\left\{\;
	\begin{aligned}
		&\partial_t G_\varphi + \mu \dot{\eta}(t) |D_x| G_\varphi + A_\eps(D)G_\varphi =\begin{pmatrix}
		0 \\
		F^1_\varphi + \pa_x q_{\varphi}
		\end{pmatrix} \\
		&\eps^2 \Big(\partial_t H_\varphi + \mu \dot{\eta}(t) |D_x| H_\varphi + B_\eps(D)H_\varphi \Big) =-\begin{pmatrix}
		0 \\
		F^2_\varphi + \pa_y q_\varphi
		\end{pmatrix}\\
		&\pa_x R^1_\varphi + \pa_y R^2_\varphi = 0	 \\
		&(R^1_\varphi,R^2_\varphi)/_{y=0} = (R^1_\varphi,R^2_\varphi)/_{y=1} = 0	
	\end{aligned}
	\right.
\end{equation}

Where $|D_x|$ denote the Fourier multiplier of the symbol $|\xi|$. In what follows, we recall that we use "C" to denote a generic positive constant which can change from line to line. Thanks to Theorems \ref{th:11} and \ref{th:22}, Propositions \ref{prop:1bis1} and \ref{prop:2bis1}, we deduce for $0<s<1$ that 
\begin{multline} \label{5Mbis1}
	\Vert u^\eps_\Theta \Vert_{\tilde{L}^\infty(\mathbb{R}^+;\mathcal{B}^\f12)} + \Vert (u + \pa_t u )_\phi \Vert_{\tilde{L}^\infty(\mathbb{R}^+;\mathcal{B}^\f12 \cap \mathcal{B}^{\f72})} +\Vert \pa_y^2 u_\phi \Vert_{\tilde{L}^2(\mathbb{R}^+;\mathcal{B}^\f12 \cap \mathcal{B}^{\f72})} \\ + \Vert \pa_y u_\phi \Vert_{\tilde{L}^2(\mathbb{R}^+;\mathcal{B}^\f12 \cap \mathcal{B}^{\f72})}  + \Vert \pa_t^2 u_\phi \Vert_{\tilde{L}^2(\mathbb{R}^+;\mathcal{B}^{\f32})} \leq M,
\end{multline}
where $u^\eps_\Theta$ and $u_\phi$ are respectively determined by \eqref{analybis} and \eqref{eq:Anphichp3} and $M \geq 1$ is a constant independent to $\eps$.

\textit{Proof of the theorem \ref{th:33}} We apply the dyadic operator in the horizontal variable $\Delta_q^h$ to \eqref{eq:hydroPEGHvarphibis1} and taking the $L^2$ inner product of the resulting equation with $\Delta_q^h G_\varphi$ and $\Delta_q^h H_\varphi$ we obtain 

\begin{multline} \label{S9eq1bis1}
	 \psca{\Delta_q^h \pa_t G_{\varphi}, \Delta_q^h G_{\varphi}}_{L^2} + \mu \dot{\eta}(t) \psca{\Delta_q^h |D_x| G_{\varphi}, \Delta_q^h G_{\varphi}}_{L^2} + \psca{\Delta_q^h A_\eps(D) G_{\varphi}, \Delta_q^h G_{\varphi}}_{L^2} \\
	= \psca{\Delta_q^h (F^1)_{\varphi}, \Delta_q^h (\pa_tR^1)_{\varphi})}_{L^2} - \psca{\Delta_q^h \pa_x q_{\varphi}, \Delta_q^h (\pa_t R^1)_{\varphi}}_{L^2}
\end{multline}
and
\begin{multline} \label{S9eq2bis1}
	  \psca{\Delta_q^h \pa_t H_{\varphi}, \Delta_q^h H_{\varphi}}_{L^2} + \mu \dot{\eta}(t) \psca{\Delta_q^h |D_x| H_{\varphi}, \Delta_q^h H_{\varphi}}_{L^2} + \psca{\Delta_q^h B_\eps(D) H_{\varphi}, \Delta_q^h H_{\varphi}}_{L^2} \\
	= -\psca{\Delta_q^h (F^2)_{\varphi}, \Delta_q^h (\pa_tR^1)_{\varphi})}_{L^2} - \psca{\Delta_q^h \pa_y q_{\varphi}, \Delta_q^h (\pa_t R^1)_{\varphi}}_{L^2}.
\end{multline}

Due to the free divergence condition, we have  
\begin{align*}
	\abs{\psca{\Delta_q^h \nabla q_{\varphi}, \Delta_q^h (\pa_tR^1,\pa_tR^2)_{\varphi}}_{L^2}} = 0.
\end{align*}

Then by using Lemma \ref{lem:Bernsteinbis1}, we achieve
\begin{multline} \label{S9eq3bis1}
	 \frac{1}{2}\frac{d}{dt}\Big( \| \Delta_q^h (\pa_tR^1)_\varphi \|_{L^2}^2 + \| \Delta_q^h \pa_yR^1_\varphi \|_{L^2}^2 + \| \Delta_q^h \eps(\pa_tR^2)_\varphi \|_{L^2}^2 + \| \Delta_q^h \eps\pa_yR^2_\varphi \|_{L^2}^2 \\+ \eps^4\| \Delta_q^h \pa_xR^2_\varphi \|_{L^2}^2 + \eps^2\| \Delta_q^h \pa_xR^1_\varphi \|_{L^2}^2  \Big) +   \| \Delta_q^h (\pa_tR^1)_\varphi\|_{L^2}^2 + \| \Delta_q^h \eps(\pa_tR^2)_\varphi\|_{L^2}^2 \\ +  \mu \dot{\eta}(t')  \| \Delta_q^h |D_x|^\frac{1}{2}(\pa_t R^1)_\varphi\|_{L^2}^2  + \eps^2\mu \dot{\eta}(t')  \| \Delta_q^h |D_x|^\frac{1}{2}(\pa_t R^2)_\varphi\|_{L^2}^2 +\eps^2\mu \dot{\eta}(t') \| \Delta_q^h |D_x|^\frac{1}{2}\pa_x R^1_\varphi\|_{L^2}^2 \\  +\mu \dot{\eta}(t') \| \Delta_q^h |D_x|^\frac{1}{2}\pa_y R^1_\varphi\|_{L^2}^2 + \mu \dot{\eta}(t') \| \Delta_q^h |D_x|^\frac{1}{2}\pa_y R^2_\varphi\|_{L^2}^2 + \eps^4\mu \dot{\eta}(t') \| \Delta_q^h |D_x|^\frac{1}{2}\pa_x R^2_\varphi\|_{L^2}^2 
\\	\leq  \abs{\psca{\Delta_q^h (F^1)_{\varphi}, \Delta_q^h (\pa_tR^1)_{\varphi})}}_{L^2} + \abs{\psca{\Delta_q^h (F^2)_{\varphi}, \Delta_q^h (\pa_tR^2)_{\varphi}}}_{L^2}.
\end{multline}

Now we still have to take the inner product in $L^2$ with  $\Delta_q^h R^1_\varphi$ and $\Delta_q^h R^2_\varphi$ to the equation 
\begin{align}\label{eq:HypstandRRbis1}
    &e^{\varphi(t,|D_x|)} (\pa_t^2 R^1 + \pa_t R^1 -\pa_y^2R^1 - \eps^2 \pa_x^2 R^1 +\pa_x p -F^1)  = 0 \nonumber\\
    &e^{\varphi(t,|D_x|)} \Big(\eps^2(\pa_t^2 R^2 + \pa_t R^2 -\pa_y^2R^2-\eps^2\pa_xR^2) +\pa_y p - F^2 \Big) = 0,
\end{align}
 we obtain  
\begin{multline} \label{S9eq1bis1bis1}
	 \psca{\Delta_q^h (\pa_t^2R^1)_\varphi, \Delta_q^h  R^1_{\varphi}}_{L^2} + \psca{\Delta_q^h (\pa_tR^1)_\varphi, \Delta_q^h  R^1_{\varphi}}_{L^2} - \psca{\Delta_q^h \pa_y^2R^1_\varphi, \Delta_q^h R^1_{\varphi}}_{L^2} \\-\eps^2 \psca{\Delta_q^h \pa_x^2R^1_\varphi, \Delta_q^h R^1_{\varphi}}_{L^2}  
	= \psca{\Delta_q^h (F^1)_{\varphi}, \Delta_q^h R^1_{\varphi})}_{L^2} - \psca{\Delta_q^h \pa_x q_{\varphi}, \Delta_q^h R^1_{\varphi}}_{L^2},
\end{multline}
and 
\begin{multline} \label{S9eq1bis2bis1}
	 \psca{\Delta_q^h (\eps \pa_t^2R^2)_\varphi, \Delta_q^h  \eps R^2_{\varphi}}_{L^2} + \psca{\Delta_q^h (\eps \pa_t R^2)_\varphi, \Delta_q^h  \eps R^2_{\varphi}}_{L^2} - \psca{\Delta_q^h \eps \pa_y^2R^2_\varphi, \Delta_q^h \eps R^2_{\varphi}}_{L^2}\\  -\eps^2 \psca{\Delta_q^h \eps \pa_x^2R^2_\varphi, \Delta_q^h \eps R^2_{\varphi}}_{L^2} 
	 = \psca{F^2)_{\varphi}, \Delta_q^h R^2_{\varphi})}_{L^2} - \psca{\Delta_q^h \pa_y q_{\varphi}, \Delta_q^h R^2_{\varphi}}_{L^2},
\end{multline}

In what follows, we shall use again the technical lemmas in Section \ref{sec:1}, to handle term by term in the estimate \eqref{S9eq1bis1bis1} and \eqref{S9eq1bis2bis1}. We start by the complicate term $ I_1 = \psca{\Delta_q^h (\pa_t^2R^1)_\varphi, \Delta_q^h  R^1_{\varphi}}_{L^2}$ and $I_2= \psca{\Delta_q^h (\eps \pa_t^2R^2)_\varphi, \Delta_q^h  \eps R^2_{\varphi}}_{L^2}$, so by using integration by parts, we find 
\begin{align*}
    I_1 &= \frac{d}{dt
    }\int \Delta_q^h (\pa_t R^1)_\varphi \Delta_q^h  R^1_\varphi dx - \int \Delta_q^h (\pa_t R^1)_\varphi \Delta_q^h (\pa_tR^1)_\varphi dx \\ &\quad\quad\quad\quad\quad + 2\mu \dot{\eta}(t)\int \Delta_q^h |D_x| (\pa_tR^1)_\varphi \Delta_q^h R^1_\varphi dx \\
    I_2 &= \frac{d}{dt
    }\int \Delta_q^h (\eps \pa_t R^2)_\varphi \Delta_q^h  \eps R^2_\varphi dx - \int \Delta_q^h (\eps \pa_t R^2)_\varphi \Delta_q^h (\eps\pa_tR^2)_\varphi dx \\ &\quad\quad\quad\quad\quad + 2\mu \dot{\eta}(t)\int \Delta_q^h |D_x| (\eps\pa_tR^2)_\varphi \Delta_q^h \eps R^2_\varphi dx
\end{align*}

Then by using Lemma \ref{lem:Bernsteinbis1}, we achieve
\begin{align} \label{S9eq1bis3bis1}
	 &\frac{d}{dt}\int \Delta_q^h (\pa_t R^1)_\Theta \Delta_q^h  R^1_\Theta dx - \int \Delta_q^h (\pa_t R^1)_\Theta \Delta_q^h (\pa_tR^1)_\Theta dx \nonumber \\ &+ 2\mu \dot{\eta}(t)\int  \Delta_q^h |D_x| (\pa_tR^1)_\Theta \Delta_q^h R^1_\Theta dx+ 
	 \frac{d}{dt}\int \Delta_q^h (\pa_t R^2)_\Theta \Delta_q^h  R^2_\Theta dx \nonumber \\
	   &- \int \Delta_q^h (\pa_t R^2)_\Theta \Delta_q^h (\pa_tR^2)_\Theta dx + 2\mu \dot{\eta}(t)\int  \Delta_q^h |D_x| (\pa_tR^2)_\Theta \Delta_q^h R^2_\Theta dx \nonumber \\
	   &+ \frac{1}{2}\frac{d}{dt} \|\Delta_q^h R^1_\Theta\|_{L^2}^2 + \frac{1}{2}\frac{d}{dt} \|\Delta_q^h \eps R^2_\Theta\|_{L^2}^2 + \mu \dot{\tau}(t) \|\Delta_q^h |D_x|^\frac{1}{2} (R^1_\Theta,R^2_\Theta) \|_{L^2}^2 \nonumber\\
	   &+ \|\Delta_q^h  \pa_yR^1_\varphi \|_{L^2}^2  + \|\Delta_q^h  \pa_y \eps R^2_\varphi \|_{L^2}^2 + \eps^2\|\Delta_q^h  \pa_xR^1_\Theta \|_{L^2}^2+ \eps^2\|\Delta_q^h  \pa_x \eps R^2_\Theta \|_{L^2}^2	   
	    \\ &= \psca{\Delta_q^h (F^1)_{\varphi}, \Delta_q^h R^1_{\varphi})}_{L^2} + \psca{\Delta_q^h (F^2)_{\varphi}, \Delta_q^h R^2_{\varphi})}_{L^2}. \nonumber       
\end{align}
Now we claim that
\begin{multline}
\int_0^t \abs{2\mu \dot{\eta}(t)\int  \Big(\Delta_q^h |D_x| (\pa_tR^1)_\Theta \Delta_q^h R^1_\Theta + \Delta_q^h |D_x| (\pa_tR^2)_\Theta \Delta_q^h R^2_\Theta\Big) dx} \\ \leq 2^{-q} d_q^2 \mu\Big( \Vert (\pa_tR^1,\eps\pa_t R^2)_{\Theta} \Vert_{\tilde{L}^2_{t,\dot{\eta}}(\mathcal{B}^{1})} \Vert  \pa_y (R^1,\eps R^2)_{\Theta} \Vert_{\tilde{L}^2_{t,\dot{\eta}}(\mathcal{B}^{1})} \Big),
\end{multline} 
\begin{multline} \label{eq:G1bis1}
  \int_0^t \abs{\psca{\Delta_q^h F^1_{\varphi}, \Delta_q^h  R^1_{\varphi}}_{L^2}} dt' + \int_0^t \abs{\psca{\Delta_q^h F^1_{\varphi}, \Delta_q^h (\pa_tR^1)_{\varphi}}_{L^2}} dt' \\ \lesssim  2^{-q}d_q^2 \bigg(\eps \Vert \pa_y u_{\varphi} \Vert_{\tilde{L}^2_{t}(\mathcal{B}^{\frac{5}{2}})} \Big( \Vert \eps R^1_{\varphi} \Vert_{\tilde{L}^2_{t}(\mathcal{B}^{\f12})} + \Vert \eps (\pa_tR^1)_{\varphi} \Vert_{\tilde{L}^2_{t}(\mathcal{B}^{\f12})}\Big) \\ +\Vert u_{\varphi} \Vert_{\tilde{L}^\infty_{t}(\mathcal{B}^{\f32})}^{\frac{1}{2}} \Vert \pa_y R^1_{\varphi} \Vert_{\tilde{L}^2_{t}\mathcal{B}^{\f12})} \Big( \|R^1_\varphi \|_{\tilde{L}^2_{t,\dot{\eta}}(\mathcal{B}^{1})} \\ + \|(\pa_tR^1)_\varphi \|_{\tilde{L}^2_{t,\dot{\eta}}(\mathcal{B}^{1})} \Big) + \|(R^1,\pa_t R^1)_\varphi \|_{\tilde{L}^2_{t,\dot{\eta}}(\mathcal{B}^{1})}^2 \bigg),
\end{multline}
and
\begin{multline} \label{eq:G2bis1}
 \int_0^t \abs{\psca{\Delta_q^h F^2_{\varphi}, \Delta_q^h R^2_{\varphi}}_{L^2}} dt' + \int_0^t \abs{\psca{\Delta_q^h F^2_{\varphi}, \Delta_q^h (\pa_tR^2)_{\varphi}}_{L^2}} dt'\\
\lesssim 2^{-q} d_q^2 \Big \lbrace \|R^1_\varphi \|_{\tilde{L}^2_{t,\dot{\eta}}(\mathcal{B}^{1})}^2 + \eps^2 \|(R^2_\varphi,(\pa_tR^2)_\varphi) \|_{\tilde{L}^2_{t,\dot{\eta}}(\mathcal{B^1})}^2 + \eps^2 \|(R^2_\varphi,\eps(\pa_tR^2)_\varphi) \|_{\tilde{L}^2_t(\mathcal{B}^\f12)} \\
\times \Big( \|(\pa_t^2 u)_\varphi \|_{\tilde{L}^2_t(\mathcal{B}^{\f32})} + \|(\pa_t u )_\varphi \|_{\tilde{L}^2_t(\mathcal{B}^{\f32})} + \eps \|\pa_y u_\varphi \|_{\tilde{L}^2_t(\mathcal{B}^{\f72})} + \|\pa_y u_\varphi \|_{\tilde{L}^2_t(\mathcal{B}^{\f32})} \Big) \\
 + \eps^2 \|(R^2_\varphi,(\pa_tR^2)_\varphi) \|_{\tilde{L}^2_{t,\dot{\eta}}(\mathcal{B}^{1})} \Big( \|R^2_\varphi \|_{\tilde{L}^2_{t,\dot{\eta}}(\mathcal{B}^{1})} + \|u^\eps_\varphi \|_{\tilde{L}^\infty_t(\mathcal{B}^\f12)}^\frac{1}{2} \|\pa_y u_\varphi \|_{\tilde{L}^2_t(\mathcal{B}^{\frac{3}{2}})} \\
 + \|u_\varphi \|_{\tilde{L}^\infty_t(\mathcal{B}^{\f32})}^\frac{1}{2}( \|\pa_y R^2_\varphi \|_{\tilde{L}^2_t(\mathcal{B}^\f12)} +  \|\pa_y u_\varphi \|_{\tilde{L}^2_t(\mathcal{B}^{\f32})} ) \Big \rbrace,
\end{multline}
the proof of those estimates will be presented later in the last section of the chapter.

By virtue of \eqref{5Mbis1}, \eqref{eq:G1bis1} and \eqref{eq:G2bis1}, we infer
\begin{multline}
\sum_{i=1}^{2} \int_0^t \abs{\psca{\Delta_q^h F^i_{\varphi}, \Delta_q^h R^i_{\varphi}}_{L^2}} dt' + \int_0^t \abs{\psca{\Delta_q^h F^i_{\varphi}, \Delta_q^h (\pa_tR^i)_{\varphi}}_{L^2}} dt'\\
\lesssim d_q^2 2^{-q} \Big( M\eps \|(\eps R^1,\eps\pa_tR^1)_\varphi\|_{\tilde{L}^2_t(\mathcal{B}^\f12)} + M^\frac{1}{2} \Vert \pa_y R^1_{\varphi} \Vert_{\tilde{L}^2_{t}\mathcal{B}^{\f12})}\|(R^1,\pa_tR^1)_\varphi \|_{\tilde{L}^2_{t,\dot{\eta}}(\mathcal{B}^{1})} \\
+ \|(R^1,\pa_t R^1)_\varphi \|_{\tilde{L}^2_{t,\dot{\eta}}(\mathcal{B}^{1})}^2 + \|(\eps R^2,\eps \pa_tR^2)_\varphi \|_{\tilde{L}^2_{t,\dot{\eta}}(\mathcal{B}^{1})}^2 + M\eps \|(\eps R^2,\eps \pa_tR^2)_\varphi \|_{\tilde{L}^2_t(\mathcal{B}^\f12)}  \\
+ M^\frac{3}{2} \eps \|(\eps R^2,\eps \pa_tR^2)_\varphi \|_{\tilde{L}^2_{t,\dot{\eta}}(\mathcal{B}^{1})} + M^\frac{1}{2} \Vert \eps \pa_y R^2_{\varphi} \Vert_{\tilde{L}^2_{t}\mathcal{B}^{\f12})}\|(\eps R^2,\eps \pa_tR^2)_\varphi \|_{\tilde{L}^2_{t,\dot{\eta}}(\mathcal{B}^{1})}\Big).
\end{multline}

Then by summing $2\times \eqref{S9eq3bis1}$ with $\eqref{S9eq1bis3bis1}$ and we multiply the resulting by $2^{q}$ and then integrating over time, and summing with respect to $q \in \ZZ$, we find that for $t<T^\star$

\begin{multline} \label{S9eq7bis1}
	  \mathcal{K}(t) \leq \sqrt{C} \mathcal{E}(0) 
 +\sqrt{C}\Big( \sqrt{M\eps} \|(\eps R^1,\eps\pa_tR^1)_\varphi\|_{\tilde{L}^2_t(\mathcal{B}^\f12)}^\frac{1}{2} + M^\frac{1}{4} \Vert \pa_y R^1_{\varphi} \Vert_{\tilde{L}^2_{t}\mathcal{B}^{\f12})}^\frac{1}{2}\|(R^1,\pa_tR^1)_\varphi \|_{\tilde{L}^2_{t,\dot{\eta}}(\mathcal{B}^{1})}^\frac{1}{2} \\
+ \|(R^1,\pa_t R^1)_\varphi \|_{\tilde{L}^2_{t,\dot{\eta}}(\mathcal{B}^{1})} + \|(\eps R^2,\eps \pa_tR^2)_\varphi \|_{\tilde{L}^2_{t,\dot{\eta}}(\mathcal{B}^{1})} + \sqrt{M\eps} \|(\eps R^2,\eps \pa_tR^2)_\varphi \|_{\tilde{L}^2_t(\mathcal{B}^\f12)}^\frac{1}{2}  \\
+ M^\frac{3}{4} \eps^\frac{1}{2} \|(\eps R^2,\eps \pa_tR^2)_\varphi \|_{\tilde{L}^2_{t,\dot{\eta}}(\mathcal{B}^{1})}^\frac{1}{2} + M^\frac{1}{4} \Vert \eps \pa_y R^2_{\varphi} \Vert_{\tilde{L}^2_{t}\mathcal{B}^{\f12})}^\frac{1}{2}\|(\eps R^2,\eps \pa_tR^2)_\varphi \|_{\tilde{L}^2_{t,\dot{\eta}}(\mathcal{B}^{1})}^\frac{1}{2} \Big),
\end{multline}
where 
\begin{align*}
\mathbf{K}(t) &= \Big( \frac{1}{2}\|(R^1 + \pa_tR^1,\eps (R^2+\pa_tR^2) )_\varphi \|_{\tilde{L}_t^\infty(\mathcal{B}^{\f12})} + \frac{1}{2}\|( \pa_tR^1)_\varphi,\eps (\pa_tR^2)_\varphi \|_{\tilde{L}_t^\infty(\mathcal{B}^{\f12})}  \\ & + \| \pa_y(R^1,\eps R^2)_\varphi \|_{\tilde{L}_t^\infty(\mathcal{B}^{\f12})} + \eps \| \pa_x(R^1,\eps R^2)_\varphi \|_{\tilde{L}_t^\infty(\mathcal{B}^{\f12})} \Big)  +\|(\pa_tR^1,\eps\pa_tR^2)_\varphi\|_{\tilde{L}_t^2(\mathcal{B}^{\f12})}  \\& +\sqrt{\mu} \| (R^1,\eps R^2)_\varphi \|_{\tilde{L}^2_{t,\dot{\eta}(t)}(\mathcal{B}^{1})} + \|\pa_y(R^1,\eps R^2)_\varphi\|_{\tilde{L}_t^2(\mathcal{B}^{\f12})}  +\eps \|\pa_x(R^1,\eps R^2)_\varphi\|_{\tilde{L}_t^2(\mathcal{B}^{\f12})}\\ & +  \sqrt{\mu}  \|(\pa_t R^1,\eps \pa_tR^2)_\varphi \|_{\tilde{L}^2_{t,\dot{\eta}(t)}(\mathcal{B}^{1})}  +  \sqrt{\mu} \| \pa_y (R^1,\eps R^2)_\varphi \|_{\tilde{L}^2_{t,\dot{\eta}(t)}(\mathcal{B}^{1})}^2 + \sqrt{\mu} \| \pa_x (R^1,\eps R^2)_\varphi\|_{\tilde{L}^2_{t,\dot{\eta}(t)}(\mathcal{B}^{1})}.
\end{align*}
and 
\begin{align*}
\mathcal{E}(0) &=  \|e^{a|D_x|}\pa_y(u^\eps_0 - u_0,\eps (v^\eps_0 - v_0)) \|_{\mathcal{B}^{\f12}} + \sqrt{C}\eps \|e^{a|D_x|}\pa_x(u^\eps_0 - u_0,\eps (v^\eps_0 - v_0)) \|_{\mathcal{B}^{\f12}}  \\ &+\sqrt{C}\|e^{a|D_x|}((u^\eps_0-u_0)+(u^\eps_1 -u_1),\eps (v^\eps_0 - v_0) + \eps(v^\eps_1 - v_1)) \|_{\mathcal{B}^{\f12}} \\ &+\sqrt{C}\|e^{a|D_x|}((u^\eps_1 -u_1), \eps(v^\eps_1 - v_1)) \|_{\mathcal{B}^{\f12}}
\end{align*}
Applying Young's inequality gives rise to
\begin{align} \label{S9eq8bis1}
\mathcal{K}(t) \leq \sqrt{C} \Big( \mathcal{E}(0) + M( \eps + \|( R^1, \pa_tR^1)_\varphi \|_{\tilde{L}^2_{t,\dot{\eta}}(\mathcal{B}^{1})} + \|(\eps R^2,\eps \pa_tR^2)_\varphi \|_{\tilde{L}^2_{t,\dot{\eta}}(\mathcal{B}^{1})}) \Big)
\end{align}
 
 Taking $\mu = CM^2$ lead to \eqref{vvbis1}, this completes the proof of the theorem \ref{th:33}.


\section{Proof of the estimates \eqref{eq:G1bis1} and \eqref{eq:G2bis1}}
\subsection{Proof of the estimate \eqref{eq:G1bis1}}
We first observe that
\begin{align*}
	F^1_{\varphi}  = (\eps^2 \pa_x^2u - (u^\eps \pa_x R^1 + R^1 \pa_x u)-(v^\eps \pa_yR^1 + R^2 \pa_y u))_{\varphi},
\end{align*}     
so, we define  
\begin{align*}
	G_1^q & =  \int_0^t \abs{\psca{\Delta_q^h F^1_{\varphi}, \Delta_q^h (\pa_tR^1)_{\varphi}}_{L^2}} dt' \\ &= \int_0^t \abs{\psca{\Delta_q^h (\eps^2 \pa_x^2u - (u^\eps \pa_x R^1 + R^1 \pa_x u)-(v^\eps \pa_yR^1 + R^2 \pa_y u))_{\varphi}, \Delta_q^h (\pa_tR^1)_{\varphi}}_{L^2}} dt' \\
	&\leq I_1^q + I_2^q + I_3^q,
\end{align*}
where
\begin{align*}
	I_1^q &= \int_0^t \abs{\psca{\Delta_q^h (\eps^2 \pa_x^2 u)_{\varphi}, \Delta_q^h (\pa_tR^1)_{\varphi}}_{L^2}} dt'\\
	I_2^q &= \int_0^t \abs{\psca{ \Delta_q^h(u^\eps \pa_x R^1 + R^1 \pa_x u )_{\varphi},\Delta_q^h (\pa_tR^1)_{\varphi}}_{L^2}} dt'\\
	I_3^q &= \int_0^t \abs{\psca{ \Delta_q^h(v^\eps \pa_yR^1 + R^2 \pa_y u)_{\varphi}, \Delta_q^h (\pa_tR^1)_{\varphi}}_{L^2}} dt'.
\end{align*}

We first observe that  
\begin{align}\label{I1qbis1}
	I_1^q \leq C d_q^2 2^{-q} \eps^2 \Vert u_{\varphi} \Vert_{\tilde{L}^2_t(\mathcal{B}^{\f52})} \Vert (\pa_tR^1)_{\varphi} \Vert_{\tilde{L}^2_t(\mathcal{B}^{\f12})}.
\end{align}

For $I_2$, we write
\begin{align*}
	I_2^q = \int_0^t \abs{\psca{\Delta_q^h(u^\eps \pa_x R^1 + R^1 \pa_x u )_{\varphi},\Delta_q^h (\pa_tR^1)_{\varphi}}_{L^2}} dt' \leq I_{21} + I_{22},
\end{align*}
where
\begin{align*}
	I_{21}^q &=  \int_0^t \abs{\psca{\Delta_q^h(u^\eps \pa_x R^1)_{\varphi},\Delta_q^h (\pa_tR^1)_{\varphi}}_{L^2}} dt'\\
	I_{22}^q &=  \int_0^t \abs{\psca{ \Delta_q^h( R^1 \pa_x u )_{\varphi},\Delta_q^h (\pa_tR^1)_{\varphi}}_{L^2}} dt'.
\end{align*}
Lemma $\ref{lem:ABCbis1}$ implies
\begin{align}\label{I21qbis1}
	I_{21}^q  \leq C d_q^2 2^{-q}  \Vert R^1_{\varphi} \Vert_{\tilde{L}^2_{t,\dot{\eta}(t)}(\mathcal{B}^{1})}\Vert (\pa_tR^1)_{\varphi} \Vert_{\tilde{L}^2_{t,\dot{\eta}(t)}(\mathcal{B}^{1})},
\end{align}

For $I_{22}$, using Bony's decomposition for the horizontal variable, we write
$$ R^1\pa_x u = T^h_{\pa_xu}R^1 + T^h_{R^1}\pa_xu + R^h(R^1,\pa_xu), $$
and then, we have the following bound
\begin{align*}
	I_{22}^q = \int_0^t \abs{\psca{ \Delta_q^h  (T^h_{\pa_xu}R^1 + T^h_{R^1}\pa_xu + R^h(R^1,\pa_xu))_{\varphi},\Delta_q^h (\pa_tR^1)_\varphi}_{L^2}} dt' \leq I_{22,1}^q + I_{22,2}^q +I_{22,3}^q
\end{align*}
with
\begin{align*} 
	I_{22,1}^q &= \int_0^t  \abs{\psca{\Delta_q^h  (T^h_{\pa_xu}R^1)_{\varphi},\Delta_q^h (\pa_tR^1)_\varphi}_{L^2}} dt'\\
	I_{22,2}^q &= \int_0^t  \abs{\psca{\Delta_q^h  (T^h_{R^1}\pa_xu)_{\varphi},\Delta_q^h (\pa_tR^1)_\varphi}_{L^2}} dt'\\
	I_{22,3}^q &= \int_0^t  \abs{\psca{\Delta_q^h  ( R^h(R^1,\pa_xu)_{\varphi},\Delta_q^h (\pa_tR^1)_\varphi}_{L^2}} dt'.
\end{align*}
Using the support properties given in [\cite{B1981}, Proposition 2.10] and the definition of $T^h_{R^1}\pa_xu$, we have 
\begin{align*}
	I_{22,2}^q &\leq   \sum_{|q-q'|\leq4} \int_0^t \Vert S_{q'-1}^h R^1_{\varphi} \Vert_{L^\infty} \Vert \Delta_{q'}^h \pa_x u_{\varphi} \Vert_{L^2} \Vert \Delta_q^h (\pa_tR^1)_{\varphi} \Vert_{L^2}dt'.
\end{align*}

As we have $R^1 = u^\eps - u$, so
\begin{align*}
 \Vert S_{q'-1}^h R^1_{\varphi} \Vert_{L^\infty} =  \Vert S_{q'-1}^h (u^\eps - u)_{\varphi} \Vert_{L^\infty} \leq \Vert S_{q'-1}^h u^\eps_{\varphi} \Vert_{L^\infty} + \Vert S_{q'-1}^h u_{\varphi} \Vert_{L^\infty},
\end{align*}
then,
\begin{align*}
	I_{22,2}^q &= \int_0^t \abs{\psca{\Delta_q^h  (T^h_{\pa_xu}R^1)_{\varphi},\Delta_q^h (\pa_tR^1)_\varphi}}\\ & \leq \sum_{|q-q'|\leq4} \int_0^t \Vert S_{q'-1}^h u^\eps_{\varphi} \Vert_{L^\infty} \Vert \Delta_{q'}^h \pa_x u_{\varphi} \Vert_{L^2} \Vert \Delta_q^h (\pa_tR^1)_{\varphi} \Vert_{L^2} dt' \\
	&+ \sum_{|q-q'|\leq4} \int_0^t \Vert S_{q'-1}^h u_{\varphi} \Vert_{L^\infty} \Vert \Delta_{q'}^h \pa_x u_{\varphi} \Vert_{L^2} \Vert \Delta_q^h (\pa_tR^1)_{\varphi} \Vert_{L^2}dt'.
\end{align*}
We note that we get by applying Bernstein lemma \ref{lem:Bernsteinbis1}
\begin{align*}
\Vert \Delta_{q'}^h u^\eps_{\varphi} \Vert_{L^\infty} &\leq 2^{\frac{q'}{2}} \Vert \Delta_{q'}^h u^\eps_{\varphi} \Vert_{L^2_h(L^\infty_v)} \\
&\leq C d_{q'}(u^\eps_\varphi) \Vert u^\eps_{\varphi} \Vert_{\mathcal{B}^\f12}^\f12 \Vert \pa_y u^\eps_{\varphi} \Vert_{\mathcal{B}^\f12}^\f12,
\end{align*}
Here and in all that follows, we always denote $(d_{q'}(u_\varphi^\eps))_{q'\in \ZZ}$ to be a generic element of $\ell^1(\ZZ)$ so that $\sum_{q'\in\ZZ}d_{q'}(u^\eps_\varphi) \leq 1$, then 
\begin{align*}
\Vert S_{q'-1}^h u^\eps_{\varphi} \Vert_{L^\infty} \lesssim \Vert u^\eps_{\varphi} \Vert_{\mathcal{B}^\f12}^\f12 \Vert \pa_y u^\eps_{\varphi} \Vert_{\mathcal{B}^\f12}^\f12 \text{ \ and \ } \Vert S_{q'-1}^h u_{\varphi} \Vert_{L^\infty} \lesssim \Vert u_{\varphi} \Vert_{\mathcal{B}^\f12}^\f12 \Vert \pa_y u_{\varphi} \Vert_{\mathcal{B}^\f12}^\f12.
\end{align*}
So 
\begin{align*}
I_{22,2}^q & \lesssim \sum_{|q-q'|\leq4} \int_0^t \Vert u^\eps_{\varphi} \Vert_{\mathcal{B}^\f12}^\f12 \Vert \pa_y u^\eps_{\varphi} \Vert_{\mathcal{B}^\f12}^\f12 \Vert \Delta_{q'}^h \pa_x u_{\varphi} \Vert_{L^2} \Vert \Delta_q^h (\pa_tR^1)_{\varphi} \Vert_{L^2} dt' \\
	&+ \sum_{|q-q'|\leq4} \int_0^t \Vert u_{\varphi} \Vert_{\mathcal{B}^\f12}^\f12 \Vert \pa_y u_{\varphi} \Vert_{\mathcal{B}^\f12}^\f12 \Vert \Delta_{q'}^h \pa_x u_{\varphi} \Vert_{L^2} \Vert \Delta_q^h (\pa_tR^1)_{\varphi} \Vert_{L^2}dt'\\
	& \lesssim \sum_{|q-q'|\leq4} \Vert u^\eps_{\varphi} \Vert_{L^\infty_t(\mathcal{B}^\f12)}^\f12  \Vert \Delta_{q'}^h \pa_x u_{\varphi} \Vert_{L^2_t(L^2)} \Big(\int_0^t \Vert \pa_y u^\eps_{\varphi} \Vert_{\mathcal{B}^\f12} \Vert \Delta_q^h (\pa_tR^1)_{\varphi} \Vert_{L^2}^2 dt'\Big)^\f12 \\
	&+ \sum_{|q-q'|\leq4} \Vert u_{\varphi} \Vert_{L^\infty_t(\mathcal{B}^\f12)}^\f12  \Vert \Delta_{q'}^h \pa_x u_{\varphi} \Vert_{L^2_t(L^2)} \Big(\int_0^t \Vert \pa_y u_{\varphi} \Vert_{\mathcal{B}^\f12} \Vert \Delta_q^h (\pa_tR^1)_{\varphi} \Vert_{L^2}^2 dt'\Big)^\f12 \\
	& \leq C 2^{-q} d_q^2 \big(\Vert u^\eps_{\varphi} \Vert_{\tilde{L}^\infty_t(\mathcal{B}^\f12)}^\f12 + \Vert u_{\varphi} \Vert_{\tilde{L}^\infty_t(\mathcal{B}^\f12)}^\f12  \big) \Vert \pa_y u_{\varphi} \Vert_{\tilde{L}^2_t(\mathcal{B}^1)} \Vert (\pa_t R^1)_{\varphi} \Vert_{\tilde{L}^2_{t,\dot{\eta}(t)}(\mathcal{B}^1)}. 
\end{align*}
Now, we recall that
$$ \Vert S_{q-1}^h \pa_x u_{\varphi} \Vert_{L^{\infty}} \leq \sum_{l\leq q-2} 2^{\frac{3l}{2}} \Vert \Delta_l^h u_{\varphi} \Vert_{L^2}^{\frac{1}{2}} \Vert \Delta_l^h \pa_y u_\varphi \Vert_{L^2}^{\frac{1}{2}} \lesssim 2^q \Vert \pa_y u_\varphi \Vert_{\mathcal{B}^{\f12}},$$
so we can deduce 
\begin{align*}
	I_{22,1}^q &= \int_0^t \abs{\psca{\Delta_q^h  (T^h_{\pa_xu}R^1)_{\varphi},\Delta_q^h (\pa_tR^1)_\varphi}}  \leq \sum_{|q-q'|\leq 4} \int_0^t  \Vert S_{q'-1}^h \pa_x u_{\varphi} \Vert_{L^{\infty}} \Vert \Delta_{q'}^h R^1_\varphi \Vert_{L^2} \Vert \Delta_{q}^h (\pa_tR^1)_\varphi \Vert_{L^2} \\
	&\lesssim \sum_{|q-q'|\leq 4} \int_0^t  2^{q'} \Vert \pa_y u_\varphi \Vert_{\mathcal{B}^{\f12}} \Vert \Delta_{q'}^h R^1_\varphi \Vert_{L^2} \Vert \Delta_{q}^h (\pa_tR^1)_\varphi \Vert_{^2} \\
	&\lesssim  \sum_{|q-q'|\leq 4} 2^{q'} \left( \int_0^t \Vert \pa_y u_\varphi \Vert_{\mathcal{B}^{\f12}} \Vert \Delta_{q'}^h R^1_\varphi \Vert_{L^2}^2 dt' \right)^\frac{1}{2}  \left( \int_0^t \Vert \pa_y u_\varphi \Vert_{\mathcal{B}^{\f12}} \Vert \Delta_{q}^h (\pa_tR^1)_\varphi \Vert_{L^2}^2 dt' \right)^\frac{1}{2} 
\end{align*}
Using the definition of $\dot{\eta}(t)$ and Definition \ref{def:CLweightbis12} we have
$$ \left( \int_0^t \Vert \pa_y u_\varphi \Vert_{\mathcal{B}^{\f12}} \Vert \Delta_{q}^h R^1_\varphi \Vert_{L^2}^2 dt' \right)^\frac{1}{2} \lesssim 2^{-q} d_q(R^1_\varphi) \Vert R^1_{\varphi} \Vert_{\tilde{L}^2_{t,\dot{\eta}(t)}(\mathcal{B}^{1})}. $$
Then,
$$ I_{22,1}^q \lesssim  2^{-q} d_q^2 \Vert R^1_{\varphi} \Vert_{\tilde{L}^2_{t,\dot{\eta}(t)}(\mathcal{B}^{1})} \Vert (\pa_tR^1)_{\varphi} \Vert_{\tilde{L}^2_{t,\dot{\eta}(t)}(\mathcal{B}^{1})}$$
where
$$ d_q^2 = d_q \left(\sum_{|q-q'|\leq 4} d_{q'} \right) $$

In a similar way, we have 
\begin{align*}
	I_{22,3}^q &= \int_0^t \abs{\psca{\Delta_q^h  (R^h(R^1,\pa_xu))_{\varphi},\Delta_q^h (\pa_tR^1)_\varphi}}dt' \\ &\lesssim 2^{\frac{q}{2}} \sum_{q'\geq q-3} \int_0^t \Vert \Delta_{q'}^h R^1_\varphi \Vert_{L^2} \Vert \tilde{\Delta}_{q'}^h \pa_x u_\varphi \Vert_{L^2_h(L^\infty_v)} \Vert \Delta_{q}^h (\pa_tR^1)_\varphi \Vert_{L^2} dt' \\
	&\lesssim 2^{\frac{q}{2}} \sum_{q'\geq q-3} \int_0^t 2^{\frac{q'}{2}} \Vert \Delta_{q'}^h R^1_\varphi \Vert_{L^2} \Vert \pa_y u_\varphi \Vert_{\mathcal{B}^{\f12}} \Vert \Delta_{q}^h (\pa_tR^1)_\varphi \Vert_{L^2} \\
	&\lesssim 2^{\frac{q}{2}} \sum_{q'\geq q-3} 2^{\frac{q'}{2}} \left(\int_0^t  \Vert \pa_y u_\varphi \Vert_{\mathcal{B}^{\f12}} \Vert \Delta_{q'}^h R^1_\varphi \Vert_{L^2} dt'  \right) \left(\int_0^t  \Vert \pa_y u_\varphi \Vert_{\mathcal{B}^{\f12}} \Vert \Delta_{q}^h (\pa_tR^1)_\varphi \Vert_{L^2} dt'  \right) \\
	&\lesssim  d_q^2 2^{-q}  \Vert  R^1_\varphi \Vert_{\tilde{L}^2_{t,\dot{\eta}(t)}(\mathcal{B}^{1})}\Vert  (\pa_tR^1)_\varphi \Vert_{\tilde{L}^2_{t,\dot{\eta}(t)}(\mathcal{B}^{1})}
\end{align*}

Then we conclude that 
\begin{align}\label{I22bis1}
I_{22} \lesssim  d_q^2 2^{-q} \Big( \Vert  (R^1,\pa_tR^1)_\varphi \Vert_{\tilde{L}^2_{t,\dot{\eta}(t)}(\mathcal{B}^{1})}^2 + \Vert u_{\varphi} \Vert_{\tilde{L}^\infty_t(\mathcal{B}^{\f32})} \Vert \pa_y R^1_\varphi \Vert_{\tilde{L}^2_t(\mathcal{B}^{\f12})}^2 \Big)
\end{align}

For the term $I_3^q$, we write 
\begin{align}
	\label{eq:I3bis1} I_3^q = \int_0^t \abs{\psca{ v^\eps \pa_yR^1 + R^2 \pa_y u))_{\varphi}, \Delta_q^h (\pa_tR^1)_{\varphi}}_{L^2}} dt' \leq I_{31}^q + I_{32}^q,
\end{align}
where
\begin{align*}
	I_{31}^q &= \int_0^t \abs{\psca{\Delta_q^h(v^\eps \pa_y R^1)_\varphi , \Delta_q^h (\pa_tR^1)_\varphi }_{L^2}} dt'\\
	I_{32}^q &= \int_0^t \abs{\psca{\Delta_q^h(R^2 \pa_y u)_\varphi, \Delta_q^h (\pa_tR^1)_\varphi}_{L^2}} dt'.
\end{align*}

Since 
$$ v^\eps \pa_y R^1 = (R^2 + v)\pa_y R^1 = R^2\pa_y R^1 +  v\pa_y R^1,$$
we get 
\begin{align*}
	I_{31}^q \leq I_{31,1}^q + I_{31,2}^q,
\end{align*}
with
\begin{align*}
	I_{31,1}^q &= \int_0^t \abs{\psca{ \Delta_q^h(R^2\pa_y R^1)_\varphi , \Delta_q^h (\pa_t R^1)_\varphi }_{L^2}} dt'\\
	I_{31,2}^q &= \int_0^t \abs{\psca{ \Delta_q^h( v\pa_y R^1)_\varphi , \Delta_q^h (\pa_tR^1)_\varphi }_{L^2}}dt'.
\end{align*}
Lemma $\ref{lem:Apa_yBCbis1}$ implies
\begin{align} \label{eq:I311bis1}
	 I_{31,1}^q =  \lesssim  d_q^2 2^{-q} \Vert R^1_\varphi \Vert_{\tilde{L}^2_{t,\dot{\eta}(t)}(\mathcal{B}^{1})}\Vert (\pa_tR^1)_\varphi \Vert_{\tilde{L}^2_{t,\dot{\eta}(t)}(\mathcal{B}^{1})}.
\end{align}
For the term $I_{31,2}^q$, we apply Bony's decomposition with respect to the horizontal variable  
$$ v\pa_y R^1 = T^h_v \pa_yR^1 + T^h_{\pa_yR^1}v +R^h(v,\pa_yR^1). $$
Using \eqref{vvbis1}, we have 
\begin{align*}
\Vert S_{q'-1}^h v_\varphi \Vert_{L^\infty} =  \Vert S_{q'-1}^h \int_0^y \pa_x u_\varphi(t,x,s)ds \Vert_{L^\infty} &\lesssim \sum_{l\leq q'-2} 2^{\frac{3l}{2}} \Vert \Delta_l^h u_\varphi \Vert_{L^2}^{\frac{1}{2}} \Vert \Delta_l^h \pa_y u_\varphi \Vert_{L^2}^{\frac{1}{2}} \\
&\lesssim 2^{\frac{q'}{2}} \Vert u_\varphi \Vert_{\mathcal{B}^{\f32}}^{\frac{1}{2}} \Vert \pa_yu_\varphi \Vert_{\mathcal{B}^{\f12}}^{\frac{1}{2}},
\end{align*}
from which, we infer 
\begin{multline}
\int_0^t \abs{\psca{\Delta_q^h( T^h_v \pa_yR^1)_\varphi , \Delta_q^h (\pa_tR^1)_\varphi }_{L^2}} dt'  \\ \lesssim \sum_{|q'-q| \leq 4} \int_0^t \Vert S_{q'-1}^h v_\varphi \Vert_{L^\infty} \Vert \Delta_{q'}^h \pa_y R^1_\varphi \Vert_{L^2} \Vert \Delta_q^h (\pa_tR^1)_\varphi \Vert_{L^2}dt' \\ \leq \sum_{|q'-q| \leq 4} \int_0^t  2^{\frac{q'}{2}} \Vert u_\varphi \Vert_{\mathcal{B}^{\f32}}^{\frac{1}{2}} \Vert \pa_yu_\varphi \Vert_{\mathcal{B}^{\f12}}^{\frac{1}{2}} \Vert \Delta_{q'}^h \pa_y R^1_\varphi \Vert_{L^2} \Vert \Delta_q^h (\pa_t R^1)_\varphi \Vert_{L^2} dt' \\
\lesssim  d_q^2 2^{-q} \Vert u_\varphi \Vert_{\tilde{L}^\infty_{t} (\mathcal{B}^{\f32})}^{\frac{1}{2}} \Vert \pa_y R^1_\varphi \Vert_{\tilde{L}^2_t(\mathcal{B}^{\f12})}  \Vert (\pa_t R^1)_\varphi \Vert_{\tilde{L}^2_{t,\dot{\eta}(t)}(\mathcal{B}^{1})},
\end{multline}
where $\set{d_q}$ forms a suitable sequence. 

As we have $R^1 = u^\eps - u$, so
\begin{align*}
 \Vert S_{q'-1}^h \pa_y R^1_{\varphi} \Vert_{L^\infty} =  \Vert S_{q'-1}^h \pa_y(u^\eps - u)_{\varphi} \Vert_{L^\infty} \leq \Vert S_{q'-1}^h \pa_y u^\eps_{\varphi} \Vert_{L^\infty} + \Vert S_{q'-1}^h \pa_y u_{\varphi} \Vert_{L^\infty},
\end{align*}
then,
\begin{align*}
&\int_0^t \abs{\psca{\Delta_q^h( T^h_{\pa_yR^1}v)_\varphi , \Delta_q^h (\pa_tR^1)_\varphi }_{L^2}} dt' \\ &\lesssim \sum_{|q'-q| \leq 4} \int_0^t \Vert S_{q'-1}^h \pa_yR^1_\varphi \Vert_{L^\infty_h(L^2_v)} \Vert \Delta_{q'}^h v_\varphi \Vert_{L^2_h(L^\infty_v)} \Vert \Delta_q^h (\pa_t R^1)_\varphi \Vert_{L^2} dt' \\
&\lesssim \sum_{|q'-q| \leq 4} \int_0^t\Vert S_{q'-1}^h \pa_y u^\eps_{\varphi} \Vert_{L^\infty_h(L^2_v)} \Vert \Delta_{q'}^h v_\varphi \Vert_{L^2_h(L^\infty_v)} \Vert \Delta_q^h (\pa_t R^1)_\varphi \Vert_{L^2} dt' \\
& + \sum_{|q'-q| \leq 4} \int_0^t\Vert S_{q'-1}^h \pa_y u_{\varphi} \Vert_{L^\infty_h(L^2_v)} \Vert \Delta_{q'}^h v_\varphi \Vert_{L^2_h(L^\infty_v)} \Vert \Delta_q^h (\pa_t R^1)_\varphi \Vert_{L^2} dt',
\end{align*}
We note that we get by applying Bernstein lemma \ref{lem:Bernsteinbis1}
\begin{align*}
\Vert S_{q'-1}^h \pa_y u^\eps_{\varphi} \Vert_{L^\infty_h(L^2_v)} \lesssim \Vert \pa_y u^\eps_{\varphi} \Vert_{\mathcal{B}^\f12} \text{ \ and \ } \Vert S_{q'-1}^h u_{\varphi} \Vert_{L^\infty_h(L^2_v)} \lesssim \Vert \pa_y u_{\varphi} \Vert_{\mathcal{B}^\f12}.
\end{align*}
So,
\begin{align*}
I_{22,2}^q & \lesssim \sum_{|q-q'|\leq4} \int_0^t \Vert \pa_y u^\eps_{\varphi} \Vert_{\mathcal{B}^\f12} \Vert \Delta_{q'}^h v_{\varphi} \Vert_{L^2_h(L^\infty_v)} \Vert \Delta_q^h (\pa_tR^1)_{\varphi} \Vert_{L^2} dt' \\
	&+ \sum_{|q-q'|\leq4} \int_0^t \Vert \pa_y u_{\varphi} \Vert_{\mathcal{B}^\f12} \Vert \Delta_{q'}^h v_{\varphi} \Vert_{L^2_h(L^\infty_v)} \Vert \Delta_q^h (\pa_tR^1)_{\varphi} \Vert_{L^2}dt'\\
	& \lesssim \sum_{|q-q'|\leq4} \Vert \pa_y u^\eps_{\varphi} \Vert_{L^2_t(\mathcal{B}^\f12)}^\f12  \Vert \Delta_{q'}^h v_{\varphi} \Vert_{L^\infty_t(L^2_h(L^\infty_v))}^\f12 \Vert \Delta_{q'}^h v_{\varphi} \Vert_{L^2_t(L^2_h(L^\infty_v))}^\f12 \Big(\int_0^t \Vert \pa_y u^\eps_{\varphi} \Vert_{\mathcal{B}^\f12} \Vert \Delta_q^h (\pa_tR^1)_{\varphi} \Vert_{L^2}^2 dt'\Big)^\f12 \\
	&+ \sum_{|q-q'|\leq4} \Vert \pa_y u_{\varphi} \Vert_{L^\infty_t(\mathcal{B}^\f12)}^\f12  \Vert \Delta_{q'}^h \pa_x \pa_y u_{\varphi} \Vert_{L^2_t(L^2)} \Big(\int_0^t \Vert \pa_y u_{\varphi} \Vert_{\mathcal{B}^\f12} \Vert \Delta_q^h (\pa_tR^1)_{\varphi} \Vert_{L^2}^2 dt'\Big)^\f12 \\
	& \leq C 2^{-q} d_q^2 \big(\Vert \pa_y u^\eps_{\varphi} \Vert_{\tilde{L}^2_t(\mathcal{B}^\f12)}^\f12 \Vert \pa_y u_{\varphi} \Vert_{\tilde{L}^\infty_t(\mathcal{B}^1)}^\f12 + \Vert \pa_y u_{\varphi} \Vert_{\tilde{L}^\infty_t(\mathcal{B}^\f12)}^\f12 \Vert \pa_y u_{\varphi} \Vert_{\tilde{L}^2_t(\mathcal{B}^1)}^\f12  \big) \Vert \pa_y u_{\varphi} \Vert_{\tilde{L}^2_t(\mathcal{B}^1)}^\f12 \Vert (\pa_t R^1)_{\varphi} \Vert_{\tilde{L}^2_{t,\dot{\eta}(t)}(\mathcal{B}^1)}. 
\end{align*}

where $\set{d_q}$ forms a suitable sequence. 

Finally, by using that 
$$ \Vert \Delta_{q'}^h v_\varphi(t,x,y) \Vert_{L^2_h(L^\infty_v)} \lesssim  d_{q'}(u_\phi)\Vert u_\varphi \Vert_{\mathcal{B}^{\f32}}^{\frac{1}{2}} \Vert \pa_yu_\varphi \Vert_{\mathcal{B}^{\f12}}^{\frac{1}{2}}, $$
we have 
\begin{align*}
&\int_0^t \abs{\psca{ \Delta_q^h( R^h(v,\pa_yR^1))_\varphi , \Delta_q^h (\pa_tR^1)_\varphi }_{L^2}} dt' \\ &\lesssim 2^{\frac{q}{2}} \sum_{q'\geq q -3} \int_0^t  \Vert \Delta_{q'}^h v_\varphi \Vert_{L^2_h(L^\infty_v)} \Vert \tilde{\Delta}_{q'}^h \pa_yR^1_\varphi \Vert_{L^2} \Vert \Delta_q^h (\pa_tR^1)_\varphi \Vert_{L^2} dt' \\
&\lesssim 2^{\frac{q}{2}} \sum_{q'\geq q-3} \int_0^t  \Vert u_\varphi \Vert_{\mathcal{B}^{\f32}}^{\frac{1}{2}}  \Vert \tilde{\Delta}_{q'}^h \pa_yR^1_\varphi \Vert_{L^2} \Vert \pa_yu_\varphi \Vert_{\mathcal{B}^{\f12}}^{\frac{1}{2}}  \Vert \Delta_q^h (\pa_tR^1)_\varphi \Vert_{L^2}dt' \\
&\lesssim  d_q^2 2^{-q} \Vert u_\varphi \Vert_{\tilde{L}^\infty_t(\mathcal{B}^{\f32})}^{\frac{1}{2}} \Vert \pa_y R^1_\varphi \Vert_{\tilde{L}^2_t(\mathcal{B}^{\f12})}  \Vert (\pa_tR^1)_\varphi \Vert_{\tilde{L}^2_{t,\dot{\eta}(t)}(\mathcal{B}^{1})},
\end{align*}

Then we obtaine the following estimates, 
\begin{eqnarray} \label{eq:I312bis1} I_{31,2}^q  \lesssim 2^{-q}  d_q^2 \Vert u_\varphi \Vert_{\tilde{L}^\infty_{t} (\mathcal{B}^{\f32})}^{\frac{1}{2}} \Vert \pa_y R^1_\varphi \Vert_{\tilde{L}^2_t(\mathcal{B}^\f12)}\Vert (\pa_tR^1)_\varphi \Vert_{\tilde{L}^2_{t,\dot{\eta}(t)}(\mathcal{B}^{1})}. \end{eqnarray}

\medskip

Now we estimate the term $I_{32}^q$ in \eqref{eq:I3bis1}. Bony's decomposition with respect to the horizontal variable implies 
\begin{align*}
I_{32}^q = \int_0^t \abs{\psca{\Delta_q^h( T_{R^2}^h\pa_y u + T^h_{\pa_yu} R^2 + R^h(R^2,\pa_y u))_\varphi, \Delta_q^h (\pa_tR^1)_\varphi}_{L^2}} dt' \leq  I_{32,1}^q + I_{32,2}^q+I_{32,3}^q,
\end{align*}
where
\begin{align*}
I_{32,1}^q &= \int_0^t \abs{\psca{\Delta_q^h( T_{R^2}^h \pa_yu)_\varphi, \Delta_q^h (\pa_tR^1)_\varphi}_{L^2}} dt' \\
I_{32,2}^q &= \int_0^t \abs{\psca{\Delta_q^h(  T^h_{\pa_yu} R^2)_\varphi, \Delta_q^h (\pa_tR^1)_\varphi}_{L^2}} dt' \\
I_{32,3}^q &= \int_0^t \abs{\psca{\Delta_q^h( R^h(R^2,\pa_y u))_\varphi, \Delta_q^h (\pa_tR^1)_\varphi}_{L^2}} dt'.
\end{align*}
We first observe that 
\begin{align*}
I_{32,1}^q & \lesssim \sum_{|q'-q|\leq 4} \int_0^t \Vert S_{q'-1}^h R^2_{\varphi} \Vert_{L^\infty} \Vert \Delta_{q'}^h \pa_y u_\varphi \Vert_{L^2} \Vert \Delta_q^h (\pa_t R^1)_\varphi \Vert_{L^2} dt' \\
&\lesssim \sum_{|q'-q|\leq 4} \int_0^t 2^{\frac{-q'}{2}} \Vert S_{q'-1}^h R^2_{\varphi} \Vert_{L^\infty} \Vert \pa_y u_\varphi \Vert_{\mathcal{B}^\f12} \Vert \Delta_q^h (\pa_tR^1)_\varphi \Vert_{L^2}dt'.
\end{align*}
Due to the fact that $R^2(t,x,y) = -\int_0^y \pa_x R^1(t,x,s) ds$, we deduce
\begin{align*}
I_{32,1}^q & \lesssim \sum_{|q'-q|\leq 4} \int_0^t 2^{\frac{-q'}{2}} \Vert S_{q'-1}^h R^2_{\varphi} \Vert_{L^\infty} \Vert \pa_y u_\varphi \Vert_{\mathcal{B}^\f12} \Vert \Delta_q^h (R^1)_\varphi \Vert_{L^2} \\
&\lesssim \sum_{|q'-q|\leq 4} 2^{q'} \left( \int_0^t \Vert \pa_y u_\varphi \Vert_{\mathcal{B}^\f12}  \int_0^y \Vert S_{q'-1}^h \pa_x R^1_\varphi(t,x,s) \Vert_{L^2}^2 ds dt' \right)^\frac{1}{2}  \left(  \int_0^t \Vert \pa_y u_\varphi \Vert_{\mathcal{B}^\f12} \Vert \Delta_q^h R^1_\varphi \Vert_{L^2}^2 dt' \right)^\frac{1}{2},
\end{align*}
Taking into account the definition of $\dot{\eta}(t)$ and Definition \ref{def:CLweightbis12} we obtain
$$ \left( \int_0^t \Vert \pa_y u_\varphi \Vert_{\mathcal{B}^{\f12}} \Vert \Delta_{q}^h (\pa_t R^1)_\varphi \Vert_{L^2}^2 dt' \right)^\frac{1}{2} \lesssim 2^{-\frac{q}{2}} d_q \Vert (\pa_tR^1)_{\varphi} \Vert_{\tilde{L}^2_{t,\dot{\eta}(t)}(\mathcal{B}^{1})}. $$
Then,
\begin{align} 
I_{32,1}^q \lesssim  2^{-q}  d_q^2\Vert R^1_{\varphi} \Vert_{\tilde{L}^2_{t,\dot{\eta}(t)}(\mathcal{B}^{1})} \Vert (\pa_tR^1)_{\varphi} \Vert_{\tilde{L}^2_{t,\dot{\eta}(t)}(\mathcal{B}^{1})}
\end{align}

Now, for $I_{32,2}^q$, we have 
\begin{align*}
I_{32,2}^q &\lesssim \sum_{|q'-q|\leq 4} \int_0^t \Vert S_{q'-1}^h \pa_y u_\varphi \Vert_{L^\infty_h(L^2_v)} \Vert \Delta_{q'}^h R^1_\varphi \Vert_{L^2_h(L^\infty_v)} \Vert \Delta_q^h (\pa_tR^1)_\varphi \Vert_{L^2} dt' \\
&\lesssim \sum_{|q'-q|\leq 4} \int_0^t 2^{\frac{q'}{2}} \Vert \pa_y u_\varphi \Vert_{\mathcal{B}^{\f12}} 2^{\frac{q'}{2}} \Vert \Delta_{q'}^h R^1_\varphi \Vert_{L^2} \Vert \Delta_q^h (\pa_tR^1)_\varphi \Vert_{L^2}dt' \\
&\lesssim  \sum_{|q'-q|\leq 4} 2^{q'} \left( \int_0^t \Vert \pa_y u_\varphi \Vert_{\mathcal{B}^{\f12}}  \Vert \Delta_{q'}^h R^1_\varphi \Vert_{L^2}^2 dt' \right)^\frac{1}{2}  \left( \int_0^t \Vert \pa_y u_\varphi \Vert_{\mathcal{B}^{\f12}}  \Vert \Delta_{q}^h (\pa_tR^1)_\varphi \Vert_{L^2}^2 dt' \right)^\frac{1}{2} \\
&\lesssim  d_q^2 2^{-q} \Vert R^1_\varphi \Vert_{\tilde{L}^2_{t,\dot{\eta}(t)}(\mathcal{B}^{1})} \Vert (\pa_tR^1)_\varphi \Vert_{\tilde{L}^2_{t,\dot{\eta}(t)}(\mathcal{B}^{1})}.
\end{align*}

We end by estimate $I_{32,3}^q$, in the same way, we have
\begin{align*}
I_{32,3}^q &\lesssim 2^{\frac{q}{2}} \sum_{q'\geq q-3} \int_0^t \Vert \Delta_{q'}^h R^2_\varphi \Vert_{L^2_h(L^\infty_v)} \Vert \tilde{\Delta}_{q'}^h \pa_y u_\varphi \Vert_{L^2} \Vert \Delta_q^h (\pa_tR^1)_\varphi \Vert_{L^2} dt' \\
&\lesssim 2^{\frac{q}{2}} \sum_{q'\geq q-3} \int_0^t 2^{\frac{q'}{2}}  \Vert \Delta_{q'}^h R^1_\varphi \Vert_{L^2} \Vert \pa_y u_\varphi \Vert_{\mathcal{B}^{\f12}} \Vert \Delta_q^h (\pa_tR^1)_\varphi \Vert_{L^2} dt'\\
&\lesssim 2^{\frac{q}{2}} \sum_{q'\geq q-3} 2^{\frac{q'}{2}} \left( \int_0^t \Vert \pa_y u_\varphi \Vert_{\mathcal{B}^{\f12}}  \Vert \Delta_{q'}^h R^1_\varphi \Vert_{L^2}^2 dt' \right)^\frac{1}{2}  \left( \int_0^t \Vert \pa_y u_\varphi \Vert_{\mathcal{B}^{\f12}}  \Vert \Delta_{q}^h (\pa_tR^1)_\varphi \Vert_{L^2}^2 dt' \right)^\frac{1}{2} \\
&\lesssim  d_q^2 2^{-q} \Vert R^1_\varphi \Vert_{\tilde{L}^2_{t,\dot{\eta}(t)}(\mathcal{B}^{1})} \Vert (\pa_tR^1)_\varphi \Vert_{\tilde{L}^2_{t,\dot{\eta}(t)}(\mathcal{B}^{1})},
\end{align*}

Summing all the resulting estimate, we can achieve 
\begin{align} \label{eq:I32bis1}
	I_{32}^q  \lesssim   d_q^2 2^{-q} \Vert R^1_{\varphi} \Vert_{\tilde{L}^2_{t,\dot{\eta}(t)}(\mathcal{B}^{1})}\Vert '\pa_tR^1)_{\varphi} \Vert_{\tilde{L}^2_{t,\dot{\eta}(t)}(\mathcal{B}^{1})}.
\end{align}
\begin{rem}
For the proof when we have $R^1$ and not $\pa_tR^1$, it also the some we need just to replace in all the proof with $R^1$ instead of $\pa_tR^1$.  
\end{rem}

By summing up \eqref{I1qbis1}-\eqref{eq:I32bis1}, we conclude the proof of \eqref{eq:G1bis1}.

\section{Proof of lemma \ref{lem:estinonlinéaire}}\label{app2}
In this appendix we present the proof of the lemma \ref{lem:estinonlinéaire}, we start first by giving the proof of the estimate \eqref{estimUU1}. By using the Bony's decomposition \eqref{eq:Bonybis1}, we can write 
$$u\pa_x u = T^h_u\pa_x u + T^h_{\pa_x u}u + R^h(u,\pa_x u). $$
Then 
\begin{align*}
&\Big(\int_0^t \|e^{\mathcal{R} t'} \Delta_q^h(u\pa_x u)_\phi\|_{L^2}^2 dt'\Big)^\f12  \lesssim \Big(\int_0^t \|e^{\mathcal{R} t'} \Delta_q^h(T^h_u\pa_x u)_\phi\|_{L^2}^2 dt'\Big)^\f12 \\
& +\Big(\int_0^t \|e^{\mathcal{R} t'} \Delta_q^h(T^h_{\pa_x u}u)_\phi\|_{L^2}^2 dt'\Big)^\f12  + \Big(\int_0^t \|e^{\mathcal{R} t'} \Delta_q^h(R^h(u,\pa_x u))_\phi\|_{L^2}^2 dt'\Big)^\f12 
\end{align*}
We use the definition of $T^h$, we have 
\begin{align*}
\Big(\int_0^t \|e^{\mathcal{R} t'} \Delta_q^h(T^h_u\pa_x u)_\phi\|_{L^2}^2 dt'\Big)^\f12 \lesssim \Big(\int_0^t \sum_{|q'-q|\leq 4} \|e^{\mathcal{R} t'} (S^h_{q'-1} u \Delta_{q'}^h \pa_x u)_\phi\|_{L^2}^2 dt'\Big)^\f12.
\end{align*}
The lemma \ref{estfg+} allows us to obtain
\begin{align}\label{appendixuu}
\Big(\int_0^t \sum_{|q'-q|\leq 4} \|e^{\mathcal{R} t'} (S^h_{q'-1} u \Delta_{q'}^h \pa_x u)_\phi\|_{L^2}^2 dt'\Big)^\f12 &\lesssim \sum_{|q'-q|\leq 4} \Big(\int_0^t \|S_{q'-1}^h u^+_\phi \|_{L^\infty_h(L^2_v)}^2 \|e^{\mathcal{R}t'} \Delta_{q'}^h \pa_x u_\phi \|_{L^2_h(L^\infty_v)}^2 dt'\Big)^\f12 \\
&\lesssim \sum_{|q'-q|\leq 4} \Big(\int_0^t \|S_{q'-1}^h u^+_\phi \|_{L^\infty_h(L^2_v)}^2 2^{2q'}\|e^{\mathcal{R}t'} \Delta_{q'}^h u_\phi \|_{L^2_h(L^\infty_v)}^2 dt'\Big)^\f12
\end{align}
 
Using  Bernstein Lemma \ref{lem:Bernsteinbis1} we have
 $$\|\Delta_{q'}^h u_\phi \|_{L^\infty_h} \lesssim 2^{\frac{q'}{2}} \|\Delta_{q'}^h u_\phi\|_{L^2_h}.$$
 So we can deduce that 
 \begin{align*}
 \|\Delta_{q'}^h u_\phi \|_{L^\infty_h(L^2_v)} &\lesssim 2^{\frac{q'}{2}} \|\Delta_{q'}^h u_\phi\|_{L^2} \\
 & \lesssim d_{q'}(u_\phi) \|u_\phi \|_{\mathcal{B}^\f12}
 \end{align*}
 Here and in all that follows, we always denote $(d_q(u_\phi))_{q\in\ZZ}$ to be a generic element of $\ell^{1}(\ZZ) $ so that $\sum\limits_{q\in\ZZ} d_q(u_\phi) \leq 1$. 
 
 Then,
\begin{align*}
	\Vert S^h_{q'-1} u_{\phi} (t') \Vert_{L^{\infty}_h(L^2_v)} \lesssim  \Vert u_{\phi}(t') \Vert_{\mathcal{B}^{\f12}}.
\end{align*}

While using the inclusion $H^1_y\hookrightarrow L^\infty_y$, 
$$\|\Delta_{q'}^h u_\phi\|_{L^\infty_v} \lesssim \|\Delta_{q'}^h u_\phi\|_{L^2_v}^\f12 \|\Delta_{q'}^h \pa_y u_\phi\|_{L^2_v}^\f12, $$
and the Poincaré inequality on the interval $\lbrace 0<y<1 \rbrace$ on $u$ (as we have that $u = 0$ when y = 0,1)
 $$ \|\Delta_{q'}^h u_\phi\|_{L^2} \lesssim \|\Delta_{q'}^h \pa_y u_\phi \|_{L^2},$$ 
 we obtain
\begin{align} \label{eq:normB12bis1bis2}
	\Vert \Delta^h_{q'} u_{\phi}(t') \Vert_{L^2_h(L^\infty_v)} &\lesssim \Vert \Delta_{q'}^h \pa_y u_{\phi}(t') \Vert_{L^2}.
\end{align}

We replace in \eqref{appendixuu} we obtain 
\begin{align}
\Big(\int_0^t \|e^{\mathcal{R} t'} \Delta_q^h(T^h_u\pa_x u)_\phi\|_{L^2}^2 dt'\Big)^\f12 &\lesssim \sum_{|q'-q|\leq 4} \Vert u_{\phi} \Vert_{\tilde{L}^\infty_t(\mathcal{B}^{\f12})} 2^{q'}\Big(\int_0^t \|e^{\mathcal{R}t'} \Delta_{q'}^h \pa_y u_\phi \|_{L^2}^2 dt'\Big)^\f12 \nonumber \\
&\lesssim  \sum_{|q'-q|\leq 4} \Vert u_{\phi} \Vert_{\tilde{L}^\infty_t(\mathcal{B}^{\f12})} 2^{q'} \Big( 2^{-\frac{5q'}{2}} d_{q'}(\pa_y u_\phi) \|e^{\mathcal{R}t} \pa_y u_\phi \|_{\tilde{L}^2_t(\mathcal{B}^\frac{5}{2})}\Big) \nonumber\\
&\leq Cd_q 2^{-\frac{3q}{2}} \Vert u_{\phi} \Vert_{\tilde{L}^\infty_t(\mathcal{B}^{\f12})} \|e^{\mathcal{R}t} \pa_y u_\phi \|_{\tilde{L}^2_t(\mathcal{B}^\frac{5}{2})},  
\end{align}
where 
$$ d_q = \sum_{|q'-q|\leq 4} d_{q'}(\pa_y u_\phi) 2^{\frac{3}{2}(q-q')}. $$
Along the same way we can obtain that 
\begin{align*}
\Big(\int_0^t \|e^{\mathcal{R} t'} \Delta_q^h(T^h_{\pa_x u}u)_\phi\|_{L^2}^2 dt'\Big)^\f12 \leq Cd_q 2^{-\frac{3q}{2}} \Vert u_{\phi} \Vert_{\tilde{L}^\infty_t(\mathcal{B}^{\f12})} \|e^{\mathcal{R}t} \pa_y u_\phi \|_{\tilde{L}^2_t(\mathcal{B}^\frac{5}{2})}.  
\end{align*}
We still have to find the estimate of the rest term, we have 
\begin{align*}
\Big(\int_0^t \|e^{\mathcal{R} t'} \Delta_q^h(R^h(u,\pa_x u))_\phi\|_{L^2}^2 dt'\Big)^\f12  &\lesssim \sum_{q'\geq q-3} \Big(\int_0^t \|\tilde{D}_{q'}^h u_\phi \|_{L^\infty_h(L^2_v)}^2 \|e^{\mathcal{R}t'} \Delta_{q'}^h \pa_x u_\phi \|_{L^2_h(L^\infty_v)}^2 dt'\Big)^\f12 \\
&\lesssim \sum_{q'\geq q-3} \Big(\int_0^t \| u_\phi \|_{\mathcal{B}^\f12}^2 2^{2q'}\|e^{\mathcal{R}t'} \Delta_{q'}^h u_\phi \|_{L^2_h(L^\infty_v)}^2 dt'\Big)^\f12.
\end{align*}
By using the Poincaré inequality we achieve 
\begin{align*}
\|e^{\mathcal{R}t'} \Delta_{q'}^h u_\phi \|_{L^2_h(L^\infty_v)} \lesssim \|e^{\mathcal{R}t'} \Delta_{q'}^h \pa_y u_\phi \|_{L^2}.
\end{align*}
Then 
\begin{align*}
\Big(\int_0^t \|e^{\mathcal{R} t'} \Delta_q^h(R^h(u,\pa_x u))_\phi\|_{L^2}^2 dt'\Big)^\f12  &\lesssim \sum_{q'\geq q-3}\Big(\int_0^t \| u_\phi \|_{\mathcal{B}^\f12}^2 2^{2q'}\|e^{\mathcal{R}t'} \Delta_{q'}^h \pa_y u_\phi \|_{L^2}^2 dt'\Big)^\f12 \\
&\leq Cd_q 2^{-\frac{3q}{2}} \Vert u_{\phi} \Vert_{\tilde{L}^\infty_t(\mathcal{B}^{\f12})} \|e^{\mathcal{R}t} \pa_y u_\phi \|_{\tilde{L}^2_t(\mathcal{B}^\frac{5}{2})}.
\end{align*}
By summing all the resulting estimates, and then multiplying by $2^{\frac{3q}{2}}$ and taking the sum over $\ZZ$ we obtain the proof of the term $\Vert e^{\mathcal{R}t} (u\pa_xu)_\phi \Vert_{\tilde{L}^2_t(\mathcal{B}^{\f32})}.$

Now we give the proof of the second estimate in \eqref{estimUU1}. By using the Bony's decomposition \eqref{eq:Bonybis1}, we can write 
$$v\pa_y u = T^h_v\pa_y u + T^h_{\pa_y u}v + R^h(v,\pa_y u). $$
Then 
\begin{align*}
&\Big(\int_0^t \|e^{\mathcal{R} t'} \Delta_q^h(v\pa_yu)_\phi\|_{L^2}^2 dt'\Big)^\f12  \lesssim \Big(\int_0^t \|e^{\mathcal{R} t'} \Delta_q^h(T^h_v\pa_y u)_\phi\|_{L^2}^2 dt'\Big)^\f12 \\
& +\Big(\int_0^t \|e^{\mathcal{R} t'} \Delta_q^h(T^h_{\pa_y u}v)_\phi\|_{L^2}^2 dt'\Big)^\f12  + \Big(\int_0^t \|e^{\mathcal{R} t'} \Delta_q^h(R^h(v,\pa_y u))_\phi\|_{L^2}^2 dt'\Big)^\f12 
\end{align*}
We use the definition of $T^h$, we have 
\begin{align*}
\Big(\int_0^t \|e^{\mathcal{R} t'} \Delta_q^h(T^h_v\pa_y u)_\phi\|_{L^2}^2 dt'\Big)^\f12 \lesssim \Big(\int_0^t \sum_{|q'-q|\leq 4} \|e^{\mathcal{R} t'} (S^h_{q'-1} v \Delta_{q'}^h \pa_y u)_\phi\|_{L^2}^2 dt'\Big)^\f12.
\end{align*}
The lemma \ref{estfg+} allows us to obtain
\begin{align}\label{appendixvu}
\Big(\int_0^t \sum_{|q'-q|\leq 4} \|e^{\mathcal{R} t'} (S^h_{q'-1} v \Delta_{q'}^h \pa_y u)_\phi\|_{L^2}^2 dt'\Big)^\f12 &\lesssim \sum_{|q'-q|\leq 4} \Big(\int_0^t \|S_{q'-1}^h v^+_\phi \|_{L^\infty}^2 \|e^{\mathcal{R}t'} \Delta_{q'}^h \pa_y u_\phi \|_{L^2_h}^2 dt'\Big)^\f12.
\end{align}
We have by applying Bernstein Lemma \ref{lem:Bernsteinbis1} that
\begin{align*}
	\Vert \Delta^h_{q'} v_{\phi}(t') \Vert_{L^\infty} \lesssim 2^{\frac{3q}{2}} \Vert \Delta_{q'}^h \int_0^y u_{\phi}(t',.,s)ds \Vert_{L^2_h(L^\infty_v)},
\end{align*}
we use the fact that $\|\int_0^y fdy'\|_{L^\infty_y}\leq \int_0^1|f|dy'\leq \|f\|_{L^2_y}$. As a result, it comes out 
\begin{align*}
	\Vert \Delta^h_{q'} v_{\phi}(t') \Vert_{L^\infty} &\lesssim 2^{\frac{3q}{2}} \Vert \Delta_{q'}^h  u_{\phi}(t') \Vert_{L^2} \\
	&\leq C d_{q'}(u_\phi) 2^{q'} \|u_\phi\|_{\mathcal{B}^\f12},
\end{align*}
where $(d_{q'}(u_\phi))_{q'\in\ZZ}$ is a generic element of $\ell^{1}(\ZZ) $ such that  $\sum d_{q'}(u_\phi) \leq 1$. Then 

\begin{align*}
	\Vert S^h_{q'-1} v_{\phi} (t') \Vert_{L^\infty} \lesssim 2^{q'} \Vert u_{\phi}(t') \Vert_{\mathcal{B}^{\f12}},
\end{align*}

We replace in \eqref{appendixvu} we obtain 
\begin{align}
\Big(\int_0^t \|e^{\mathcal{R} t'} \Delta_q^h(T^h_v\pa_y u)_\phi\|_{L^2}^2 dt'\Big)^\f12 &\lesssim \sum_{|q'-q|\leq 4} \Vert u_{\phi} \Vert_{\tilde{L}^\infty_t(\mathcal{B}^{\f12})} 2^{q'}\Big(\int_0^t \|e^{\mathcal{R}t'} \Delta_{q'}^h \pa_y u_\phi \|_{L^2}^2 dt'\Big)^\f12 \nonumber \\
&\lesssim  \sum_{|q'-q|\leq 4} \Vert u_{\phi} \Vert_{\tilde{L}^\infty_t(\mathcal{B}^{\f12})} 2^{q'} \Big( 2^{-\frac{5q'}{2}} d_{q'}(\pa_y u_\phi) \|e^{\mathcal{R}t} \pa_y u_\phi \|_{\tilde{L}^2_t(\mathcal{B}^\frac{5}{2})}\Big) \nonumber\\
&\leq Cd_q 2^{-\frac{3q}{2}} \Vert u_{\phi} \Vert_{\tilde{L}^\infty_t(\mathcal{B}^{\f12})} \|e^{\mathcal{R}t} \pa_y u_\phi \|_{\tilde{L}^2_t(\mathcal{B}^\frac{5}{2})},  
\end{align}
where 
$$ d_q = \sum_{|q'-q|\leq 4} d_{q'}(\pa_y u_\phi) 2^{\frac{3}{2}(q-q')}. $$
Along the same way we can obtain that 
\begin{align*}
\Big(\int_0^t \|e^{\mathcal{R} t'} \Delta_q^h(T^h_{\pa_y u}v)_\phi\|_{L^2}^2 dt'\Big)^\f12 \leq Cd_q 2^{-\frac{3q}{2}} \Vert u_{\phi} \Vert_{\tilde{L}^\infty_t(\mathcal{B}^{\f12})} \|e^{\mathcal{R}t} \pa_y u_\phi \|_{\tilde{L}^2_t(\mathcal{B}^\frac{5}{2})}.  
\end{align*}
We finish our proof by giving the estimate of the last term $\int_0^t \|e^{\mathcal{R} t'} \Delta_q^h(R^h(v,\pa_y u))_\phi\|_{L^2}^2 dt'$. We use the definition of $R^h$, we have 
\begin{align*}
\Big(\int_0^t \|e^{\mathcal{R} t'} \Delta_q^h(R^h(v,\pa_y u))_\phi\|_{L^2}^2 dt'\Big)^\f12  &\lesssim \sum_{q'\geq q-3} \Big(\int_0^t \|\tilde{\Delta}_{q'}^h v_\phi \|_{L^\infty}^2 \|e^{\mathcal{R}t'} \Delta_{q'}^h \pa_y u_\phi \|_{L^2}^2 dt'\Big)^\f12 \\
&\lesssim \sum_{q'\geq q-3} \Big(\int_0^t \| u_\phi \|_{\mathcal{B}^\f52}^2 2^{-2q'}\|e^{\mathcal{R}t'} \Delta_{q'}^h \pa_y u_\phi \|_{L^2}^2 dt'\Big)^\f12\\
&\leq Cd_q 2^{-\frac{3q}{2}} \Vert u_{\phi} \Vert_{\tilde{L}^\infty_t(\mathcal{B}^{\f52})} \|e^{\mathcal{R}t} \pa_y u_\phi \|_{\tilde{L}^2_t(\mathcal{B}^\frac{1}{2})}.
\end{align*}
By summing all the resulting estimates, and then multiplying by $2^{\frac{3q}{2}}$ and taking the sum over $\ZZ$ we obtain the proof of the term $\Vert e^{\mathcal{R}t} (v\pa_yu)_\phi \Vert_{\tilde{L}^2_t(\mathcal{B}^{\f32})}.$

\subsection{Proof of the estimate \eqref{eq:G2bis1}}
We first deduce from $\pa_x u + \pa_yv = 0$, that

\begin{align}\label{ffffbis1}
	&\eps^2  \int_0^t \abs{\psca{\Delta_q^h(\pa_t v)_\varphi ,\Delta_q^h (\pa_tR^2)_\varphi}_{L^2}} dt' \lesssim   d_q^2 2^{-q} \eps^2  \Vert (\pa_t u)_\varphi \Vert_{\tilde{L}^2_t(\mathcal{B}^{\f32})}\Vert  (\pa_tR^2)_\varphi \Vert_{\tilde{L}^2_t(\mathcal{B}^\f12)} \nonumber \\
	&\eps^2  \int_0^t \abs{\psca{\Delta_q^h(\pa_t^2 v)_\varphi ,\Delta_q^h (\pa_tR^2)_\varphi}_{L^2}} dt' \lesssim   d_q^2 2^{q} \eps^2  \Vert (\pa_t^2 u)_\varphi \Vert_{\tilde{L}^2_t(\mathcal{B}^{\f32})} \Vert  (\pa_tR^2)_\varphi \Vert_{\tilde{L}^2_t(\mathcal{B}^\f12)},\nonumber \\
	&\eps^2 \int_0^t \abs{\psca{\Delta_q^h(\pa_y^2 v)_\varphi ,\Delta_q^h (\pa_tR^2)_\varphi}_{L^2}} dt' \lesssim  d_q^2 2^{-q}  \eps^2  \Vert \pa_y^2 u_\varphi \Vert_{\tilde{L}^2_t(\mathcal{B}^{\f32})} \Vert (\pa_t R^2)_\varphi \Vert_{\tilde{L}^2_t(\mathcal{B}^\f12)},\nonumber \\
	&\eps^4 \int_0^t \abs{\psca{ \Delta_q^h(\pa_x^2 v)_\varphi ,\Delta_q^h (\pa_tR^2)_\varphi}_{L^2}}dt' \lesssim   d_q^2 2^{-q}\eps^4 \Vert \pa_y u_\varphi \Vert_{\tilde{L}^2_t(\mathcal{B}^{\f72})} \Vert (\pa_tR^2)_\varphi \Vert_{\tilde{L}^2_t(\mathcal{B}^\f12)}.
\end{align}

Then now we still have to control
\begin{align*}
	J_4^q = \eps^2\int_0^t \abs{\psca{\Delta_q^h(u^\eps \pa_x v^\eps)_{\varphi},\Delta_q^h (\pa_tR^2)_\varphi}_{L^2}}dt'
\end{align*}
and 
\begin{align*}
	J_5^q = \eps^2\int_0^t \abs{\psca{\Delta_q^h(v^\eps \pa_y v^\eps)_\varphi, \Delta_q^h (\pa_tR^2)_\varphi}_{L^2}} dt'.
\end{align*}

We start first by $J^4_q$, we have 
\begin{align*}
	J_4^q \leq \eps^2\pare{J_{41}^q + J_{42}^q},
\end{align*}
where
\begin{align*}
	J_{41}^q &= \int_0^t \abs{\psca{\Delta_q^h(u^\eps \pa_x R^2)_{\varphi},\Delta_q^h (\pa_t R^2)_\varphi}_{L^2}} dt' \\
	J_{42}^q &= \int_0^t \abs{\psca{\Delta_q^h(u^\eps \pa_x v)_{\varphi},\Delta_q^h (\pa_t R^2)_\varphi}_{L^2}} dt'.
\end{align*}
It follows from Lemma $\ref{lem:ABCbis1}$ that 
\begin{align} \label{eq:J41qbis1}
	J_{41}^q = &\int_0^t \abs{\psca{\Delta_q^h(u^\eps \pa_x R^2)_\varphi,\Delta_q^h (\pa_tR^2)_\varphi}_{L^2}} dt' \nonumber \\ &\lesssim d_q^2 2^{-q}  \Vert R^2_\varphi \Vert_{\tilde{L}^2_{t,\dot{\eta}(t)}(\mathcal{B}^{1})}\Vert (\pa_t R^2)_\varphi \Vert_{\tilde{L}^2_{t,\dot{\eta}(t)}(\mathcal{B}^{1})}.
\end{align}

For the second term, Bony's decomposition for the horizontal variable gives 
\begin{align*}
J_{42}^q = \int_0^t \abs{\psca{\Delta_q^h(u^\eps \pa_x v)_{\varphi},\Delta_q^h (\pa_tR^2)_\varphi}_{L^2}} dt' \leq J_{421}^q + J_{422}^q + J_{423}^q,
\end{align*}
with
\begin{align*}
J_{421}^q &= \int_0^t \abs{\psca{\Delta_q^h(T^h_{u^\eps} \pa_x v)_{\varphi},\Delta_q^h (\pa_tR^2)_\varphi}_{L^2}} dt' \\
J_{422}^q &= \int_0^t \abs{\psca{\Delta_q^h(T^h_{\pa_x v} u^\eps)_{\varphi},\Delta_q^h (\pa_tR^2)_\varphi}_{L^2}} dt'\\
J_{423}^q &= \int_0^t \abs{\psca{\Delta_q^h(R^h(u^\eps,\pa_x v))_{\varphi},\Delta_q^h (\pa_tR^2)_\varphi}_{L^2}} dt'.
\end{align*}

Due to
$$ \Vert S_{q'-1}^h u^\eps_\varphi \Vert_{L^\infty} \lesssim \Vert u_\varphi^\eps \Vert^{\frac{1}{2}}_{\mathcal{B}^{\f12}} \Vert \pa_y u_\varphi^\eps \Vert^{\frac{1}{2}}_{\mathcal{B}^{\f12}} ,$$
and the relation \eqref{vvbis1}, we have 
\begin{align*}
J_{421}^q &= \int_0^t \abs{\psca{\Delta_q^h(T^h_{u^\eps} \pa_x v)_\varphi,\Delta_q^h (\pa_tR^2)_\varphi}_{L^2}} dt' \\ &\lesssim \sum_{|q'-q| \leq 4} \int_0^t  \Vert S_{q'-1}^h u^\eps_\varphi \Vert_{L^\infty} \Vert \Delta_{q'}^h \pa_x v_\varphi \Vert_{L^2} \Vert \Delta_q^h (\pa_tR^2)_\varphi \Vert_{L^2} dt' \\
&\lesssim \sum_{|q'-q| \leq 4} \int_0^t \Vert u_\varphi^\eps \Vert^{\frac{1}{2}}_{\mathcal{B}^{\f12}} \Vert \pa_y u_\varphi^\eps \Vert^{\frac{1}{2}}_{\mathcal{B}^{\f12}} 2^{2q'} \Vert \Delta_{q'}^h u_\varphi \Vert_{L^2} \Vert \Delta_q^h (\pa_tR^2)_\varphi \Vert_{L^2} dt' \\
&\lesssim  d_q^2 2^{-q} \Vert u_\varphi^\eps \Vert^{\frac{1}{2}}_{\tilde{L}^\infty_{t}(\mathcal{B}^{\f12})} \Vert \pa_y u_{\varphi} \Vert_{\tilde{L}^2_{t}(\mathcal{B}^{2})}  \Vert (\pa_tR^2)_\varphi \Vert_{\tilde{L}^2_{t,\dot{\eta}(t)}(\mathcal{B}^{1})}.
\end{align*}

While again thanks to \eqref{vvbis1}, we find
\begin{align*}
\Vert S_{q'-1}^h \pa_x v_\varphi \Vert_{L^\infty} \lesssim \int_0^y \Vert S_{q'-1}^h \pa_x (\pa_x u_\varphi(t,x,s) \Vert_{L^\infty} ds \lesssim 2^{\frac{q'}{2}} \Vert \pa_y u_\varphi \Vert_{\mathcal{B}^{2}},
\end{align*}
which leads to 
\begin{align*}
J_{422}^q &= \int_0^t \abs{\psca{\Delta_q^h(T^h_{\pa_x v} u^\eps)_\varphi,\Delta_q^h (\pa_t R^2)_\varphi}_{L^2}}dt' \\ &\lesssim  d_q^2 2^{-q} \Vert u_\varphi^\eps \Vert^{\frac{1}{2}}_{\tilde{L}^\infty_{t}(\mathcal{B}^{\f12})} \Vert \pa_y u_{\varphi} \Vert_{\tilde{L}^2_{t}(\mathcal{B}^{2})}  \Vert (\pa_t R^2)_\varphi \Vert_{\tilde{L}^2_{t,\dot{\eta}(t)}(\mathcal{B}^{1})}.
\end{align*}

Along the same way, we obtain  

\begin{align*}
J_{423}^q &= \int_0^t \abs{\psca{ \Delta_q^h(R^h(\pa_x v, u^\eps))_\varphi,\Delta_q^h (\pa_t R^2)_\varphi}_{L^2}} \\ &\lesssim  d_q^2 2^{-q} \Vert u_\varphi^\eps \Vert^{\frac{1}{2}}_{\tilde{L}^\infty_{t}(\mathcal{B}^{\f12})} \Vert \pa_y u_{\varphi} \Vert_{\tilde{L}^2_{t}(\mathcal{B}^{2})}  \Vert (\pa_tR^2)_\varphi \Vert_{\tilde{L}^2_{t,\dot{\eta}(t)}(\mathcal{B}^{1})}.
\end{align*}

This gives rise to
\begin{align} \label{J42bis1}
 J_{42}^q \lesssim  d_q^2 2^{-q} \Vert u_\varphi^\eps \Vert^{\frac{1}{2}}_{\tilde{L}^\infty_{t}(\mathcal{B}^{\f12})} \Vert \pa_y u_{\varphi} \Vert_{\tilde{L}^2_{t}(\mathcal{B}^{2})}\Vert \pa_t R^2)_\varphi \Vert_{\tilde{L}^2_{t,\dot{\eta}(t)}(\mathcal{B}^{1})}. 
\end{align}

Now for $J_5^q$, We first note that
$$ v^\eps \pa_y v^\eps = v\pa_yR^2 + R^2\pa_yR^2 + v\pa_yv + R^2\pa_yv. $$
Lemma \ref{lem:ABCbis1} yields 
$$ \eps^2  \int_0^t \abs{\psca{ \Delta_q^h( R^2\pa_yR^2)_\varphi, \Delta_q^h (\pa_tR^2)_\varphi}_{L^2}} dt'\lesssim  d_q^2 2^{-q} \Vert (R^1_\varphi,\eps (\pa_tR^2)_\varphi) \Vert^2_{\tilde{L}^2_{t,\dot{\eta}(t)}(\mathcal{B}^{1})}. $$
From $\eqref{eq:I312bis1}$, we have 
$$ \int_0^t \abs{\psca{ \Delta_q^h( v\pa_yR^2)_\varphi, \Delta_q^h (\pa_tR^2)_\varphi}_{L^2}}dt' \lesssim   d_q^2 2^{-q} \Vert u_\varphi \Vert_{\tilde{L}^\infty_t(\mathcal{B}^{\f32})}^\frac{1}{2} \Vert \pa_y R^2_\varphi \Vert_{\tilde{L}^2_t(\mathcal{B}^\f12)}\Vert (\pa_tR^2)_\varphi \Vert_{\tilde{L}^2_{t,\dot{\eta}(t)}(\mathcal{B}^{1})}.$$
As for $\eqref{I22bis1}$, we obtain
\begin{align*}
\int_0^t &\abs{\psca{\Delta_q^h  (R^2\pa_x u)_{\varphi},\Delta_q^h (\pa_tR^2)_\varphi}_{L^2}} dt' \\ &\lesssim  d_q^2 2^{-q}\Vert  (\pa_tR^2)_\varphi \Vert_{\tilde{L}^2_{t,\dot{\eta}(t)}(\mathcal{B}^{1})} \Vert u_\varphi \Vert_{\tilde{L}^\infty_t(\mathcal{B}^{\f32})}^{\frac{1}{2}} \Vert \pa_y R^2_\varphi  \Vert_{\tilde{L}^2_t(\mathcal{B}^{\f12})} . \hspace{2cm}
\end{align*}
Then, we deduce from the proof of $\eqref{eq:I312bis1}$ that 
\begin{align*}
\int_0^t &\abs{\psca{\Delta_q^h  (v\pa_y v)_{\varphi},\Delta_q^h (\pa_t R^2)_\varphi}_{L^2}}dt'\\ &\lesssim  d_q^2 2^{-q} \Vert u_\varphi \Vert_{\tilde{L}^\infty_t(\mathcal{B}^{\f32})}^\frac{1}{2} \Vert \pa_y v_\varphi \Vert_{\tilde{L}^2_t(\mathcal{B}^\f12)}\Vert (\pa_tR^2)_\varphi \Vert_{\tilde{L}^2_{t,\dot{\eta}(t)}(\mathcal{B}^{1})}\\ & \lesssim  d_q^2 2^{-q} \Vert u_\varphi \Vert_{\tilde{L}^\infty_t(\mathcal{B}^{\f32})}^\frac{1}{2} \Vert \pa_y u_\varphi \Vert_{\tilde{L}^2_t(\mathcal{B}^{\f32})}\Vert (\pa_tR^2)_\varphi \Vert_{\tilde{L}^2_{t,\dot{\eta}(t)}(\mathcal{B}^{1})} .
\end{align*}
As a result, it comes out
\begin{align} \label{5.18bis1}
J_5^q &\lesssim  d_q^2 2^{-q} \Big( \Vert (R^1_\varphi,\eps (\pa_t R^2)_\varphi) \Vert_{\tilde{L}^2_{t,\dot{\eta}(t)}(\mathcal{B}^{1})}^2 \nonumber \\ &+ \eps^2 \Vert u_\varphi \Vert_{\tilde{L}^\infty_t(\mathcal{B}^{\f32})}^\frac{1}{2}  (  \Vert \pa_y R^2_\varphi \Vert_{\tilde{L}^2_t(\mathcal{B}^{\f12})} + \Vert \pa_y u_\varphi \Vert_{\tilde{L}^2_t(\mathcal{B}^{\f32})} ) \Vert (\pa_tR^2)_\varphi \Vert_{\tilde{L}^2_{t,\dot{\eta}(t)}(\mathcal{B}^{1})}\Big).
\end{align}

\begin{rem}
For the proof when we have $R^2$ and not $\pa_tR^2$, it also the same we need just to replace in all the proof with $R^2$ instead of $\pa_tR^2$.  
\end{rem}

By summing up \eqref{ffffbis1}-\eqref{5.18bis1}, we conclude the proof of \eqref{eq:G2bis1}

\bibliographystyle{plain}

\bibliography{bib-file}

\end{document}